\documentclass[12pt]{amsart}
\usepackage[utf8]{inputenc}
\usepackage{amssymb}
\usepackage{hyperref}
\usepackage[final]{showkeys} %[final] is the option that removes them

\input xy
\xyoption{all}

\newtheorem{mydef}{Definition}[subsection]
\newtheorem{lem}[mydef]{Lemma}
\newtheorem{thm}[mydef]{Theorem}
\newtheorem{conjecture}[mydef]{Conjecture}
\newtheorem{cor}[mydef]{Corollary}
\newtheorem{claim}[mydef]{Claim}
\newtheorem{question}[mydef]{Question}

\newtheorem{prop}[mydef]{Proposition}
\newtheorem{defin}[mydef]{Definition}
\newtheorem{example}[mydef]{Example}
\newtheorem{remark}[mydef]{Remark}
\newtheorem{notation}[mydef]{Notation}

\newtheorem{theorem}[mydef]{Theorem}
\newtheorem*{thm*}{Theorem}

\newcommand{\fct}[2]{{}^{#1}#2}

\newcommand{\ba}{\bar{a}}
\newcommand{\bb}{\bar{b}}
\newcommand{\bc}{\bar{c}}

\newcommand{\bigL}{\widehat{L}}
\newcommand{\bigK}{\widehat{K}}
\newcommand{\bigKp}[1]{\widehat{K}^{<#1}}

\newcommand{\Ksatpp}[2]{{#1}^{#2\text{-sat}}}
\newcommand{\Ksatp}[1]{\Ksatpp{\K}{#1}}

\newcommand{\sea}{\mathfrak{C}}
\newcommand{\dom}[1]{\operatorname{dom}(#1)}

\newcommand{\cf}[1]{\operatorname{cf} (#1)}
\newcommand{\seq}[1]{\langle #1 \rangle}
\newcommand{\rest}{\upharpoonright}
\newcommand{\bkappa}{\bar \kappa}
\newcommand{\s}{\mathfrak{s}}

\newcommand{\is}{\mathfrak{i}}

\def\lta{<}
\def\lea{\le}
\def\gta{>}
\def\gea{\ge}

\newbox\noforkbox \newdimen\forklinewidth
\forklinewidth=0.3pt \setbox0\hbox{$\textstyle\smile$}
\setbox1\hbox to \wd0{\hfil\vrule width \forklinewidth depth-2pt
 height 10pt \hfil}
\wd1=0 cm \setbox\noforkbox\hbox{\lower 2pt\box1\lower
2pt\box0\relax}
\def\unionstick{\mathop{\copy\noforkbox}\limits}
\newcommand{\nf}{\unionstick}
\newcommand{\nfs}[4]{#2 \nf_{#1}^{#4} #3}

\def\1nf{\unionstick^{(1)}}

\def\2nf{\unionstick^{(2)}}
\def\3nf{\unionstick^{(3)}}

\newcommand{\gtp}{\operatorname{gtp}}
\newcommand{\Ss}{S}
\newcommand{\gS}{\operatorname{gS}}

\newcommand{\hanf}[1]{h (#1)}

\newcommand{\EM}{\operatorname{EM}}

\newcommand{\Aut}{\operatorname{Aut}}
\newcommand{\EC}{\operatorname{EC}}

\newcommand{\Ll}{\mathbb{L}}
\newcommand{\Eat}{E_{\text{at}}}
\newcommand{\cl}{\operatorname{cl}}

\newcommand{\F}{\mathcal{F}}
\newcommand{\LS}{\operatorname{LS}}
\newcommand{\LST}{\operatorname{LST}}

\newcommand{\BI}{\mathbf{I}}
\newcommand{\slc}[1]{\bkappa_{#1}}

\newcommand{\clc}[1]{\kappa_{#1}}

\newcommand{\K}{\mathbf{K}}

\newcommand{\dnf}{\unionstick}
\newcommand{\ran}{\operatorname{ran}}

\title{A survey on tame abstract elementary classes}
\date{\today
}
\author{Will Boney}
\email{wboney@math.harvard.edu}
\urladdr{http://math.harvard.edu/\textasciitilde wboney/}
\address{Mathematics Department, Harvard University, Cambridge, MA, USA}
\thanks{This material is based upon work done while the first author was supported by the National Science Foundation under Grant No. DMS-1402191.}

\author{Sebastien Vasey}
\email{sebv@cmu.edu}
\urladdr{http://math.cmu.edu/\textasciitilde svasey/}
\address{Department of Mathematical Sciences, Carnegie Mellon University, Pittsburgh, Pennsylvania, USA}
\thanks{This material is based upon work done while the second author was supported by the Swiss National Science Foundation under Grant No.\ 155136.}

\begin{document}

\begin{abstract}
  Tame abstract elementary classes are a broad nonelementary framework for model theory that encompasses several examples of interest. In recent years, progress toward developing a classification theory for them has been made. Abstract independence relations such as Shelah's good frames have been found to be key objects. Several new categoricity transfers have been obtained. We survey these developments using the following result (due to the second author) as our guiding thread:

  \begin{thm*}
    If a universal class is categorical in cardinals of arbitrarily high cofinality, then it is categorical on a tail of cardinals.
  \end{thm*}
\end{abstract}

\maketitle

%% Include only the sections and not the subsections in the table of content
\setcounter{tocdepth}{1}

\tableofcontents

\section{Introduction}\label{intro-sec}

Abstract elementary classes (AECs) are a general framework for nonelementary model theory. They encompass many examples of interest while still allowing some classification theory, as exemplified by Shelah's recent two-volume book \cite{shelahaecbook, shelahaecbook2} titled \emph{Classification Theory for Abstract Elementary Classes}.

So why study the classification theory of \emph{tame} AECs in particular? Before going into a technical discussion, let us simply say that several key results can be obtained assuming tameness that provably cannot be obtained (or are very hard to prove) without assuming it. Among these results are the construction, assuming a stability or superstability hypothesis, of certain global independence notions akin to first-order forking. Similarly to the first-order case, the existence of such notions allows us to prove several more structural properties of the class (such as a bound on the number of types or the fact that the union of a chain of saturated models is saturated). After enough of the theory is developed, categoricity transfers (in the spirit of Morley's celebrated categoricity theorem \cite{morley-cip}) can be established. 

A survey of such results (with an emphasis on forking-like independence) is in Section \ref{tame-indep-sec}. However, we did not want to overwhelm the reader with a long list of sometimes technical theorems, so we thought we would first present an application: the categoricity transfer for universal classes from the abstract (Section \ref{universal-class-sec}). We chose this result for several reasons. First, its statement is simple and does not mention tameness or even abstract elementary classes. Second, the proof revolves around several notions (such as good frames) that might seem overly technical and ill-motivated if one does not see them in action first. Third, the result improves on earlier categoricity transfers in several ways, for example not requiring that the categoricity cardinal be a successor and not assuming the existence of large cardinals. Finally, the method of proof leads to Theorem \ref{event-categ-primes}, the currently best known ZFC approximation to Shelah's eventual categoricity conjecture (which is the main test question for AECs, see below).

Let us go back to what tameness is. Briefly, tameness is a property of AECs saying that Galois (or orbital) types are determined locally: two distinct Galois types must already be distinct when restricted to some small piece of their domain. This holds in elementary classes: types as sets of formulas can be characterized in terms of automorphisms of the monster model and distinct types can be distinguished by a finite set of parameters. However, Galois types in general AECs are not syntactic and their behavior can be wild, with ``new'' types springing into being at various cardinalities and increasing chains of Galois types having no unique upper bound (or even no upper bound at all).  This wild behavior makes it very hard to transfer results between cardinalities.  

For a concrete instance, consider the problem of developing a forking-like independence notion for AECs. In particular, we want to be able to extend each (Galois) type $p$ over $M$ to a unique nonforking\footnote{In the sense of the independence notion mentioned above. This will often be different from the first-order definition.} extension over every larger set $N$. If the AEC, $\K$, is nice enough, one might be able to develop an independence notion allows one to obtain a nonforking extension $q$ of $p$ over $N$. But suppose that $\K$ is not tame and that this non-tameness is witnessed by $q$.  Then there is another type $q'$ over $N$ that has all the same small restrictions as $q$.  In particular it extends $p$ and (assuming our independence notion has a reasonable continuity property) is a nonforking extension.  In this case the quest to have a unique nonforking extension is (provably, see Example \ref{examples-subsec}.(\ref{stabletame})) doomed.

This failure has led, in part, to Shelah's work on a local approach where the goal is to build a structure theory cardinal by cardinal without any ``traces of compactness'' (see \cite[p.~5]{sh576}). The central concept there is that of a good $\lambda$-frame (the idea is, roughly, that if an AEC $\K$ has a good $\lambda$-frame, then $\K$ is ``well-behaved in $\lambda$''). Multiple instances of categoricity together with non-ZFC hypotheses (such as the weak generalized continuum hypothesis: $2^{\mu} < 2^{\mu^+}$ for all $\mu$) are used to build a good $\lambda$-frame \cite{sh576}, to push it up to models of size $\lambda^+$ (changing the class in the process) \cite[Chapter II]{shelahaecbook}, and finally to push it to models of sizes $\lambda^{+\omega}$ and beyond in \cite[Chapter III]{shelahaecbook} (see Section \ref{frame-sec}).

In contrast, the amount of compactness offered by tameness and other locality properties has been used to prove similar results in simpler ways and with fewer assumptions (after tameness is accounted for). In particular, the work can often be done in ZFC. 

In tame AECs, Galois types are determined by their small restrictions and the behavior of these small restrictions can now influence the behavior of the full type.  An example can be seen in uniqueness results for nonsplitting extensions: in general AECs, uniqueness results hold for non-$\mu$-splitting extensions to models of size $\mu$, but no further (Theorem \ref{gen-unique-dns-fact}).  However, in $\mu$-tame AECs, uniqueness results hold for non-$\mu$-splitting extensions to models \emph{of all sizes} (Theorem \ref{ns-uniq-tame-fact}).  Indeed, the parameter $\mu$ in non-$\mu$-splitting becomes irrelevant.  Thus, tameness can replace several extra assumptions.  Compared to the good frame results above, categoricity in a single cardinal, tameness, and amalgamation are enough to show the existence of a good frame (Theorem \ref{ss-categ}) and tameness and amalgamation are enough to transfer the frame upwards without any change of the underlying class (Theorem \ref{good-frame-transfer-2}).

Although tameness seems to only be a weak approximation to the nice properties enjoyed by first-order logic, it is still \emph{strong enough} to import more of the model-theoretic intuition and technology from first-order.  When dealing with tame AECs, a type can be identified with the set of all of its restrictions to small domains, and these small types play a role similar to formulas. This can be made precise: one can even give a sense in which types are sets of (infinitary) formulas (see Theorem \ref{morleyization-thm}).  This allows several standard arguments to be painlessly repeated in this context.  For instance, the proof of the properties of $<\kappa$-satisfiability and the equivalence between Galois-stability and no order property in this context are similar to their first-order counterparts (see Section \ref{stab-indep-sec}). On the other hand, several arguments from the theory of tame AECs have no first-order analog (see for example the discussion around amalgamation in Section \ref{universal-class-sec}).

On the other side, while tameness is in a sense a form of the first-order compactness theorem, it is \emph{sufficiently weak} that several examples of nonelementary classes are tame.  Section \ref{examples-subsec} goes into greater depth, but diverse classes like locally finite groups, rank one valued fields, and Zilber's pseudoexponentiation all turn out to be tame. Tameness can also be obtained for free from a large cardinal axiom, and a weak form of it follows from model-theoretic hypotheses such as the combination of amalgamation and categoricity in a high-enough cardinal.  

Indeed, examples of \emph{non}-tame AECs are in short supply (Section \ref{counterex-ssec}). All known examples are set-theoretic in nature, and it is open whether there are non-tame ``natural'' mathematical classes (see (\ref{open-q-natural}) in Section \ref{conclusion-sec}). The focus on ZFC results for tame AECs allows us to avoid situations where, for example, conclusions about rank one valued fields depend on whether $2^{\aleph_0} < 2^{\aleph_1}$. This replacing of set-theoretic hypotheses with model-theoretic ones suggests that developing a classification theory for tame AECs is possible \emph{within ZFC}.

Thus, tame AECs seem to strike an important balance: they are \emph{general} enough to encompass several nonelementary classes and yet \emph{well-behaved/specific} enough to admit a classification theory. Even if one does not believe that tameness is a justified assumption, it can be used as a first approximation and then one can attempt to remove (or weaken) it from the statement of existing theorems. Indeed, there are several results in the literature (see the end of Section \ref{categ-notame-sec}) which do not directly assume tameness, but whose proof starts by deducing some weak amount of tameness from the other assumptions, and then use this tameness in crucial ways.

We now highlight some results about tame AECs that will be discussed further in the rest of this survey. We first state two motivating test questions. The first is the well-known categoricity conjecture which can be traced back to an open problem in \cite{shelahfobook78}. The following version appears as \cite[Conjecture N.4.2]{shelahaecbook}:

\begin{conjecture}[Shelah's eventual categoricity conjecture]\label{shelah-event-conj}
  There exists a function $\mu \mapsto \lambda(\mu)$ such that if $\K$ is an AEC categorical in \emph{some} $\lambda \ge \lambda(\LS (\K))$, then $\K$ is categorical in \emph{all} $\lambda' \ge \lambda (\LS (\K))$.  
\end{conjecture}

Shelah's categoricity conjecture is the main test question for AECs and remains the yardstick by which progress is measured. Using this yardstick, tame AECs are well-developed. Grossberg and VanDieren \cite{tamenessone} isolated tameness from Shelah's proof of a downward categoricity transfer in AECs with amalgamation \cite{sh394}. Tameness was one of the key (implicit) properties there (in the proof of \cite[Main claim 2.3]{sh394}, where Shelah proves that categoricity in a high-enough successor implies that types over Galois-saturated models are determined by their small restrictions\footnote{In \cite[Definition 0.24]{sh576}, Shelah defines a type to be \emph{local} if it is defined by all its $\LS (\K)$-sized restrictions.}, this property has later been called weak tameness). Grossberg and VanDieren defined tameness without the assumption of saturation and developed the theory in a series of papers \cite{tamenessone, tamenesstwo, tamenessthree}, culminating in the proof of Shelah's eventual categoricity conjecture from a successor in tame AECs with amalgamation\footnote{The work on \cite{tamenessone} was done in 2000-2001 and preprints of \cite{tamenesstwo, tamenessthree} were circulated in 2004.}.

Progress towards other categoricity transfers often proceed by first proving tameness and then using it to transfer categoricity. One of the achievements of developing the classification theory of tame AECs is the following result, due to the second author \cite{ap-universal-v9}:

\begin{thm}\label{main-thm}
  Shelah's eventual categoricity conjecture is true when $\K$ is a universal class with amalgamation. In this case, one can take $\lambda (\mu) := \beth_{\left(2^{\mu}\right)^+}$. Moreover amalgamation can be derived from categoricity in cardinals of arbitrarily high cofinality.
\end{thm}

The proof starts by observing that every universal class is tame (a result of the first author \cite{tameness-groups}, see Theorem \ref{tame-uc}). 

The second test question is more vague and grew out of the need to generalize some of the tools of first-order stability theory to AECs.

\begin{question}\label{indep-quest}
  Let $\K$ be an AEC categorical in a high-enough cardinal. Does there exists a cardinal $\chi$ such that $\K_{\ge \chi}$ admits a notion of independence akin to first-order forking?
\end{question}

The answer is positive for universal classes with amalgamation (see Theorem \ref{univ-classes-goodness}), and more generally for classes with amalgamation satisfying a certain strengthening of tameness:

\textbf{Theorem \ref{fully-good-indep}}. 
{\it  Let $\K$ be a fully $<\aleph_0$-tame and -type short AEC with amalgamation. If $\K$ is categorical in a $\lambda > \LS (\K)$, then $\K_{\ge \beth_{\left(2^{\LS (\K)}\right)^+}}$ has (in a precise sense) a superstable forking-like independence notion.}

Varying the locality assumption, one can obtain weaker, but still powerful, conclusions that are key in the proof of Theorem \ref{main-thm}.

One of the big questions in developing classification theory outside of first-order is which of the characterizations of dividing lines to take as the definition (see Section \ref{class-thy-sec}).  This is especially true when dealing with superstability.  In the first-order context, this is characterizable by forking having certain properties or the union of saturated models being saturated or one of several other properties.  In tame AECs, these characterizations have recently been proven to also be equivalent! See Theorem \ref{gvsuperstab}.

This survey is organized as follows: Section \ref{primer-no-tameness} reviews concepts from the study of general AEC.  This begins with definitions and basic notions (Galois type, etc.) that are necessary for work in AECs. Subsection \ref{class-thy-sec} is a review of classification theory without tameness.  The goal here is to review the known results that do not involve tameness in order to emphasize the strides that assuming tameness makes. Of course, we also setup notation and terminology. Previous familiarity with the basics of AECs as laid out in e.g.\ \cite[Chapter 4]{baldwinbook09} would be helpful. We also assume that the reader knows the basics of first-order model theory. 

Section \ref{tame-sec} formally introduces tameness and related principles.  Subsection \ref{examples-subsec} reviews the known examples of tameness and non-tameness.  

Section \ref{universal-class-sec} outlines the proof of Shelah's Categoricity Conjecture for universal classes. The goal of this outline is to highlight several of the tools that exist in the classification theory of tame AECs and tie them together in a single result.  After whetting the reader's appetite, Section \ref{tame-indep-sec} goes into greater detail about the classification-theoretic tools available in tame AECs. 

This introduction has been short on history and attribution and the historical remarks in Section \ref{hist-sec} fill this gap. We have written this survey in a somewhat informal style where theorems are not attributed when they are stated: the reader should also look at Section \ref{hist-sec}, where proper credits are given. \emph{It should not be assumed that an unattributed result is the work of the authors}.

Let us say a word about what is \emph{not} discussed: We have chosen to focus on recent material which is not already covered in Baldwin's book \cite{baldwinbook09}, so while its interest cannot be denied, we do not for example discuss the proof of the Grossberg-VanDieren upward categoricity transfer \cite{tamenesstwo, tamenessthree}. Also, we focus on tame AECs, so tameness-free results (such as Shelah's study of Ehrenfeucht-Mostowski models in \cite[Chapter IV]{shelahaecbook}, \cite{sh893}, or the work of VanDieren and the second author on the symmetry property \cite{vandieren-symmetry-v4-toappear, vandieren-chainsat-apal, vv-symmetry-transfer-v3}) are not emphasized.  Related frameworks which we do not discuss much are homogeneous model theory (see Example \ref{examples-subsec}.(\ref{hommod})), tame finitary AECs (Example \ref{examples-subsec}.(\ref{finitaryAEC})), and tame metric AECs (see Example \ref{examples-subsec}.(\ref{tdense})).

Finally, let us note that the field is not a finished body of work but is still very much active. Some results may become obsolete soon after, or even before, this survey is published\footnote{Indeed, since this paper was first circulated (in December 2015) the amalgamation assumption has been removed from Theorem \ref{main-thm} \cite{categ-universal-2-v1} and Question \ref{sat-quest} has been answered positively \cite{categ-saturated-v2}.}  . Still, we felt there was a need to write this paper, as the body of work on tame AECs has grown significantly in recent years and there is, in our opinion, a need to summarize the essential points.

\subsection{Acknowledgments}

This paper was written while the second author was working on a Ph.D.\ thesis under the direction of Rami Grossberg at Carnegie Mellon University and he would like to thank Professor Grossberg for his guidance and assistance in his research in general and in this work specifically.

We also thank Monica VanDieren and the referee for useful feedback that helped us improve the presentation of this paper.

\section{A primer in abstract elementary classes without tameness}\label{primer-no-tameness}

In this section, we give an overview of some of the main concepts of the study of abstract elementary classes. This is meant both as a presentation of the basics and as a review of the ``pre-tameness'' literature, with an emphasis of the difficulties that were faced. By the end of this section, we give several state-of-the-art results on Shelah's categoricity conjecture. While tameness is not assumed, deriving a weak version from categoricity is key in the proofs.

We only sketch the basics here and omit most of the proofs. The reader who wants a more thorough introduction should consult \cite{grossberg2002}, \cite{baldwinbook09}, or the upcoming \cite{grossbergbook}.  We are light on history and motivation for this part; interested readers should consult one of the references or Section \ref{hist-sec}.

Abstract elementary classes (AECs) were introduced by Shelah in the mid-seventies. The original motivation was axiomatizing classes of models of certain infinitary logics ($\Ll_{\omega_1, \omega}$ and $\Ll (Q)$), but the definition can also be seen as extracting the category-theoretic essence of first-order model theory (see \cite{lieberman-categ}).

\begin{defin}
  An abstract elementary class (AEC) is a pair $(\K, \lea)$ satisfying the following conditions:

  \begin{enumerate}
    \item $\K$ is a class of $L$-structures for a fixed language $L := L (\K)$.
    \item $\lea$ is a reflexive and transitive relation on $\K$.
    \item Both $\K$ and $\lea$ are closed under isomorphisms: If $M, N \in \K$, $M \lea N$, and $f: N \cong N'$, then $f[M], N' \in \K$ and $f[M] \lea N'$.
    \item If $M \lea N$, then $M$ is an $L$-substructure of $N$ (written\footnote{We write $|M|$ for the universe of an $L$-structure $M$ and $\|M\|$ for the cardinality of the universe. Thus $M \subseteq N$ means $M$ is a substructure of $N$ while $|M| \subseteq |N|$ means that the universe of $M$ is a subset of the universe of $N$.} $M \subseteq N$).
    \item Coherence axiom: If $M_0, M_1, M_2 \in \K$, $M_0 \subseteq M_1 \lea M_2$, and $M_0 \lea M_2$, then $M_0 \lea M_1$.
    \item Tarski-Vaught chain axioms: If $\delta$ is a limit ordinal and $\seq{M_i : i < \delta}$ is an increasing chain (that is, for all $i < j < \delta$, $M_i \in \K$ and $M_i \lea M_j$), then:

      \begin{enumerate}
        \item $M_\delta := \bigcup_{i < \delta} M_i \in \K$.
        \item $M_i \lea M_\delta$ for all $i < \delta$.
        \item If $N \in \K$ and $M_i \lea N$ for all $i < \delta$, then $M_\delta \lea N$.
      \end{enumerate}

    \item Löwenheim-Skolem-Tarski axiom\footnote{This axiom was initially called the Löwenheim-Skolem axiom, which explains why it is written $\LS(\K)$.  However, later works have referred to it this way (and sometimes written $\LST(\K)$) as an acknowledgment of Tarski's role in the corresponding first-order result.}: There exists a cardinal $\mu \ge |L (\K)| + \aleph_0$ such that for every $M \in \K$ and every $A \subseteq |M|$, there exists $M_0 \lea M$ so that $A \subseteq |M_0|$ and $\|M_0\| \le \mu + |A|$. We define the Löwenheim-Skolem-Tarski number of $\K$ (written $\LS (\K)$) to be the least such cardinal.

  \end{enumerate}

  We often will not distinguish between $\K$ and the pair $(\K, \lea)$. We write $M \lta N$ when $M \lea N$ and $M \neq N$.
\end{defin}
\begin{example}
  $(\text{Mod} (T), \preceq)$ for $T$ a first-order theory, and more generally $(\text{Mod} (\psi), \preceq_{\Phi})$ for $\psi$ an $\Ll_{\lambda, \omega}$ sentence and $\Phi$ a fragment containing $\psi$ are among the motivating examples. The Löwenheim-Skolem-Tarski numbers in those cases are respectively $|L (T)| + \aleph_0$ and $|\Phi| + |L (\Phi)| + \aleph_0$. In the former case, we say that the class is \emph{elementary}. See the aforementioned references for more examples.
\end{example}
\begin{notation}
  For $\K$ an AEC, we write $\K_\lambda$ for the class of $M \in \K$ with $\|M\| = \lambda$, and similarly for variations such as $\K_{\ge \lambda}$, $\K_{<\lambda}$, $\K_{[\lambda, \theta)}$, etc.
\end{notation}
\begin{remark}[Existence of resolutions]
  Let $\K$ be an AEC and let $\lambda > \LS (\K)$. If $M \in \K_\lambda$, it follows directly from the axioms that there exists an increasing chain $\seq{M_i : i \le \lambda}$ which is continuous\footnote{That is, for every limit $i$, $M_i = \bigcup_{j < i} M_j$.} and so that $M_\lambda = M$ and $M_i \in \K_{<\lambda}$ for all $i < \lambda$; such a chain is called a \emph{resolution} of $M$.  We also use this name to refer to the initial segment $\seq{M_i: i < \lambda}$ with $M_\lambda = M = \bigcup_{i<\lambda} M_i$ left implicit.
\end{remark}
\begin{remark}
  Let $\K$ be an AEC. A few quirks are not ruled out by the definition:

  \begin{itemize}
    \item $\K$ could be empty.
    \item It could be that $\K_{<\LS (\K)}$ is nonempty. This can be remedied by replacing $\K$ with $\K_{\ge \LS (\K)}$ (also an AEC with the same Löwenheim-Skolem-Tarski number as $\K$). Note however that in some examples, the models below $\LS (\K)$ give a lot of information on the models of size $\LS (\K)$, see Baldwin, Koerwein, and Laskowski \cite{locally-finite-aec-toappear}.
  \end{itemize}

  Most authors implicitly assume that $\K_{<\LS (\K)} = \emptyset$ and $\K_{\LS (\K)} \neq \emptyset$, and the reader can safely make these assumptions throughout. However, we will try to be careful about these details when stating results.%\footnotei{WB: Is this really important? It doesn't really come up very much \\
%  SV: Well, why not try to state correct theorems? Especially because this is a survey that is not aimed (only) at specialists... We are often using without comments that if $\K$ is an AEC, then $\K_{\ge \lambda}$ is an AEC. Sometimes if we are careless $\K_{\ge \lambda}$ might be empty... In any case I don't think it does any harm to think for five seconds before stating a theorem and figure out whether it needs the AEC to be non-empty (I do this for my  own papers, see for example the use of ``arbitrarily large models'' instead of ``no maximal models'' in the equivalence between all versions of superstability).\\
%  WB: This seems to go the opposite direction of a survey: a little imprecision is alright for a clearer picture.  Also, is there a place here where this distinction is actually made? SV: It is important in the definition of a good frame, especially when we say that a good frame can be extended... A little imprecision is allright if being precise makes one write too much but this is not the case here... I am not sure how this causes a problem: if you don't want to have to think about empty classes or classes with models below the LS number, just don't and I will edit your statements if needed.}

\end{remark}

An AEC $\K$ may not have certain structural properties that always hold in the elementary case:

\begin{defin}\label{struct-props}
  Let $\K$ be an AEC. 

  \begin{enumerate}
    \item $\K$ has \emph{amalgamation} if for any $M_0, M_1, M_2 \in \K$ with $M_0 \lea M_\ell$, $\ell = 1,2$, there exists $N \in \K$ and $f_\ell : M_\ell \xrightarrow[M_0]{} N$, $\ell = 1,2$.
   \[
 \xymatrix{\ar @{} [dr] M_1  \ar@{.>}[r]^{f_1}  & N\\
M_0 \ar[u] \ar[r] & M_2 \ar@{.>}[u]_{f_2}
 }
\]

    \item $\K$ has \emph{joint embedding} if for any $M_1, M_2 \in \K$, there exists $N \in \K$ and $f_\ell : M_\ell \rightarrow N$, $\ell = 1,2$.
    \item $\K$ has \emph{no maximal models} if for any $M \in \K$ there exists $N \in \K$ with $M \lta N$.
    \item $\K$ has \emph{arbitrarily large models} if for any cardinal $\lambda$, $\K_{\ge \lambda} \neq \emptyset$.
  \end{enumerate}

  We define localizations of these properties in the expected way. For example, we say that \emph{$\K_\lambda$ has amalgamation} or $\K$ has \emph{amalgamation in $\lambda$} (or \emph{$\lambda$-amalgamation}) if the definition of amalgamation holds when all the models are required to be of size $\lambda$.
\end{defin}

There are several easy relationships between these properties. We list here a few:

\begin{prop}
  Let $\K$ be an AEC, $\lambda \ge \LS (\K)$.

  \begin{enumerate}
    \item If $\K$ has joint embedding and arbitrarily large models, then $\K$ has no maximal models.
    \item If $\K$ has joint embedding in $\lambda$, $\K_{<\lambda}$ has no maximal models, and $\K_{\ge \lambda}$ has amalgamation, then $\K$ has joint embedding.
    \item If $\K$ has amalgamation in every $\mu \ge \LS (\K)$, then $\K_{\ge \LS (\K)}$ has amalgamation.
  \end{enumerate}
\end{prop}

In a sense, joint embedding says that the AEC is ``complete''. Assuming amalgamation, it is possible to partition the AEC into disjoint classes each of which has amalgamation and joint embedding.

\begin{prop}\label{jep-decomp}
  Let $\K$ be an AEC with amalgamation. For $M_1, M_2 \in \K$, say $M_1 \sim M_2$ if and only if $M_1$ and $M_2$ embed inside a common model (i.e.\ there exists $N \in \K$ and $f_\ell : M_\ell \rightarrow N$). Then $\sim$ is an equivalence relation, and its equivalence classes partition $\K$ into at most $2^{\LS (\K)}$-many AECs with joint embedding and amalgamation.
\end{prop}

Thus if $\K$ is an AEC with amalgamation and arbitrarily large models, we can find a sub-AEC of it which has amalgamation, joint embedding, and no maximal models. In that sense, global amalgamation implies all the other properties (see also Corollary \ref{cor-jep-from-ap}). 

Using  the existence of resolutions, it is not difficult to see that an AEC $(\K, \lea)$ is determined by its restriction to size $\lambda$ $(\K_\lambda, \lea \cap (\K_\lambda \times \K_{\lambda}))$. Thus, there is only a set of AECs with a fixed Löwenheim-Skolem-Tarski number and hence there is a Hanf number for the property that the AEC has arbitrarily large models. 

While this analysis only gives an existence proof for the Hanf number, Shelah's presentation theorem actually allows a computation of the Hanf number by establishing a connection between $\K$ and $\Ll_{\infty, \omega}$.

%Next, we remark that this Hanf number is quite low\footnotei{WB: You have an interesting notion of ``quite low.''  Maybe computable? SV: I think Baldwin describes it as quite low in his book... It is low compared to what Hanf's non-constructive proof, and also it is not a large cardinals (in ZFC anyway).}. We use the following powerful theorem which shows that any AEC is in fact the reduct of a class of models of a first-order theory omitting a collection of types:

\begin{thm}[Shelah's presentation theorem]\label{pres-thm}
  If $\K$ is an AEC with $L(\K) = L$, there exists a language $L' \supseteq L$ with $|L'| + \LS(\K)$, a first-order $L'$-theory $T'$, and a set of $T'$-types $\Gamma$ such that
  $$\K = \text{PC} (T', \Gamma, L) := \{M' \rest L \mid M' \models T' \text{ and } M' \text{ omits all the types in } \Gamma\}$$
\end{thm}

The proof proceeds by adding $\LS(\K)$-many functions of each arity.  For each $M$, we can write it as the union of a directed system $\{N_{\ba} \in \K_{\LS(\K)} : \ba \in {}^{<\omega}|M|\}$ with $\ba \in N_{\ba}$.  Then, the intended expansion $M'$ of $M$ is where the universe of $N_{\ba}$ is enumerated by new functions of arity $\ell(\ba)$ applied to $\ba$.  The types of $\Gamma$ are chosen such that $M'$ omits them if and only if the reducts of the substructures derived in this way actually form a directed system\footnote{Note that there are almost always the maximal number of types in $\Gamma$.}.

 In particular, $\K$ is the reduct of a class of models of an $\Ll_{\LS (\K)^+, \omega}$-theory.  An important caveat is that if $\K$ was given by the models of some $\Ll_{\LS(\K)^+,\omega}$-theory, the axiomatization given by Shelah's Presentation Theorem is different and uninformative.  However, it is enough to allow the computation of the Hanf number for existence.

\begin{cor}
  If $\K$ is an AEC such that $\K_{\ge \chi} \neq \emptyset$ for all $\chi < \beth_{(2^{\LS (\K)})^+}$, then $\K$ has arbitrarily large models.
\end{cor}

The cardinal $\beth_{(2^{\LS (\K)})^+}$ appears frequently in studying AECs, so has been given a name:

\begin{notation}\label{hanf-notation}
  For $\lambda$ an infinite cardinal, write $\hanf{\lambda} := \beth_{(2^{\lambda})^+}$. When $\K$ is a fixed AEC, we write $H_1 := \hanf{\LS (\K)}$ and $H_2 := \hanf{\hanf{\LS (\K)}}$.
\end{notation}

We obtain for example that any AEC with amalgamation and joint embedding in a single cardinal eventually has all the structural properties of Definition \ref{struct-props}.

\begin{cor}\label{cor-jep-from-ap}
  Let $\K$ be an AEC with amalgamation. If $\K$ has joint embedding in some $\lambda \ge \LS (\K)$, then there exists $\chi < H_1$ so that $\K_{\ge \chi}$ has amalgamation, joint embedding, and no maximal models. More precisely, there exists an AEC $\K^\ast$ such that:

  \begin{enumerate}
    \item $\K^\ast \subseteq \K$.
    \item $\LS (\K^\ast) = \LS (\K)$.
    \item $\K^\ast$ has amalgamation, joint embedding, and no maximal models.
    \item $\K_{\ge \chi} = (\K^\ast)_{\ge \chi}$.
  \end{enumerate}
\end{cor}
\begin{proof}[Proof sketch]
  First use Proposition \ref{jep-decomp} to decompose the AEC into at most $2^{\LS (\K)}$ many subclasses, each of which has amalgamation and joint embedding. Now if one of these partitions does not have arbitrarily large models, then there must exists a $\chi_0 < H_1$ in which it has no models. Take the sup of all such $\chi_0$s and observe that $\cf{H_1} = (2^{\LS (\K)})^+ > 2^{\LS (\K)}$.
\end{proof}

If $\K$ is an AEC with joint embedding, amalgamation, and no maximal models, we may build a proper-class\footnote{To make sense of this, we have to work in Gödel-Von Neumann-Bernays set theory. Alternatively, we can simply ask for the monster model to be bigger than any sizes involved in our proofs. In any case, the way to make this precise is the same as in the elementary theory, so we do not elaborate.} sized model-homogeneous universal model $\sea$, where:

\begin{defin}
Let $\K$ be an AEC, let $M \in \K$, and let $\lambda$ be a cardinal.
\begin{enumerate}
  \item $M$ is \emph{$\lambda$-model-homogeneous} if for every $M_0 \lea M$, $M_0' \gea M_0$ with $\|M\| < \lambda$, there exists $f: M_0' \xrightarrow[M_0]{} M$. When $\lambda = \|M\|$, we omit it.
  \item $M$ is \emph{universal} if for every $M' \in \K$ with $\|M'\| \le \|M\|$, there exists $f: M' \rightarrow M$.
\end{enumerate} 
\end{defin}

\begin{defin}
  We say that an AEC $\K$ \emph{has a monster model} if it has a model $\sea$ as above. Equivalently, it has amalgamation, joint embedding, and arbitrarily large models.
\end{defin}
\begin{remark}
  Even if $\K$ only has amalgamation and joint embedding, we can construct a monster model, but it may not be proper-class sized. If in addition joint embedding fails, for any $M \in \K$ we can construct a big model-homogeneous model $\sea \gea M$.
\end{remark}

Note that if $\K$ were in fact an elementary class, then the monster model constructed here is the same as the classical concept.

When $\K$ has a monster model $\sea$, we can define a semantic notion of type\footnote{A semantic (as opposed to syntactic) notion of type is the only one that makes sense in a general AEC as there is no natural logic to work in. Even in AECs axiomatized in a logic such as $\Ll_{\omega_1, \omega}$, syntactic types do not behave as they do in the elementary framework; see the discussion of the Hart-Shelah example in Section \ref{examples-subsec}.} by working inside $\sea$ and specifying that $\bb$ and $\bc$ have the same type over $A$ if and only if there exists an automorphism of $\sea$ taking $\bb$ to $\bc$ and fixing $A$. In fact, this can be generalized to arbitrary AECs:

\begin{defin}[Galois types] \label{gtp-def}
  Let $\K$ be an AEC.

  \begin{enumerate}
    \item For an index set $I$, an \emph{$I$-indexed Galois triple} is a triple $(\bb, A, N)$, where $N \in \K$, $A \subseteq |N|$, and $\bb \in \fct{I}{|N|}$.
    \item We say that the $I$-indexed Galois triples $(\bb_1, A_1, N_1)$, $(\bb_2, A_2, N_2)$ are \emph{atomically equivalent} and write $(\bb_1, A_1, N_1) \Eat^I (\bb_2, A_2, N_2)$ if $A_1 = A_2$, and there exists $N \in \K$ and $f_\ell : N_\ell \xrightarrow[A]{} N$ so that $f_1 (\bb_1) = f_2 (\bb_2)$. When $I$ is clear from context, we omit it.
    \item Note that $\Eat$ is a symmetric and reflexive relation. We let $E$ be its transitive closure.
    \item For an $I$-indexed Galois triple $(\bb, A, N)$, we let $\gtp (\bb / A; N)$ (the \emph{Galois type} of $\bb$ over $A$ in $N$) be the $E$-equivalence class of $(\bb, A, N)$.
    \item For $N \in \K$ and $A \subseteq |N|$, we let $\gS^I (A; N) := \{\gtp (\bb / A; N) \mid \bb \in \fct{I}{|N|}\}$. We also let $\gS^I (N) := \bigcup_{N' \gea N} \gS^I (N; N')$. When $I$ is omitted, this means that $|I| = 1$, e.g.\ $\gS (N)$ is $\gS^1 (N)$.
    \item We can define restrictions of Galois types in the natural way: for $p \in \gS^I (A; N)$, $I_0 \subseteq I$ and $A_0 \subseteq A$, write $p \rest A_0$ for the restriction of $p$ to $A_0$ and $p^{I_0}$ for the restriction of $p$ to $I_0$. For example, if $p = \gtp (\bb / A; N)$ and $A_0 \subseteq A$, $p \rest A_0 := \gtp (\bb / A_0; N)$ (this does not depend on the choice of representative for $p$).
    \item Given $p \in \gS^I (M)$ and $f: M \cong M'$, we can also define $f (p)$ in the natural way.
  \end{enumerate}
\end{defin}

\begin{remark} \
  \begin{enumerate}
    \item If $M \lea N$, then $\gtp (\bb / A; M) = \gtp (\bb / A; N)$. Similarly, if $f: M \cong_A N$, then $\gtp (\bb / A; M) = \gtp (f (\bb) / A; N)$. Equivalence of Galois types is the coarsest equivalence relation with these properties.
    \item If $\K$ has amalgamation, then $E = \Eat$.
    \item If $\sea$ is a monster model for $\K$, $\bb_1, \bb_2 \in \fct{<\infty}{|\sea|}$, $A \subseteq |\sea|$, then $\gtp (\bb_1 / A; \sea) = \gtp (\bb_2 / A; \sea)$ if and only if there exists $f \in \text{Aut}_A (\sea)$ so that $f (\bb_1) = \bb_2$. When working inside $\sea$, we just write $\gtp (\bb / A)$ for $\gtp (\bb / A; \sea)$, but in general, the model in which the Galois type is computed is important.
    \item The cardinality of the index set is all that is important.  However, when discussing type shortness later, it is convenient to allow the index set to be arbitrary.
  \end{enumerate}
\end{remark}

When dealing with Galois types, one has to be careful about distinguishing between types over \emph{models} and types over \emph{sets}.  Most of the basic definitions work the same for types over sets and models, and both require just amalgamation \emph{over models} to make the transitivity of atomic equivalence work. Allowing types over sets gives slightly more flexibility in the definitions. For example, we can say what is meant to be $<\aleph_0$-tame or to be $(<\LS(\K))$-tame in $\K_{\ge \LS(\K)}$. See the discussion around Definition \ref{tamedef}.

On the other hand, several basic results--such as the construction of $\kappa$-saturated models--require amalgamation over the sort of object (set or model) desired in the conclusion.  For instance, the following is true.

\begin{prop} \label{settype-bad}
Suppose that $\K$ is an AEC with amalgamation\footnote{Recall that this is defined to mean over models.}.
\begin{enumerate}
	\item The following are equivalent.
	\begin{itemize}
		\item $A$ is an amalgamation base\footnote{This should be made precise, for example by considering the embedding of $A$ inside a fixed monster model.}.
		\item For every $p \in \gS^1(A; N)$ and $M \supseteq A$, there is an extension of $p$ to $M$.
	\end{itemize}
	\item The following are equivalent.
	\begin{itemize}
		\item $\K$ has amalgamation over sets.
		\item For every $M$ and $\kappa$, there is an extension $N \gea M$ with the following property:
		\begin{center}
		For every $A \subseteq |N|$ and $|M^*| \supseteq A$ with $|A| < \kappa$, any $p \in \gS^{<\kappa}(A; M^*)$ is realized in $N$.
		\end{center}
	\end{itemize}
	
\end{enumerate}
\end{prop}

A more substantial result is \cite[Claim 3.3]{sh394}, which derives a local character for splitting in stable AECs (see Lemma \ref{stabsplit} below), but only in the context of Galois types over models.

One can give a natural definition of saturation in terms of Galois types.

\begin{defin}
  A model $M \in \K$ is \emph{$\lambda$-Galois-saturated} if for any $A \subseteq |M|$ with $|A| < \lambda$, any $N \gea M$, any $p \in \gS (A; N)$ is realized in $M$. When $\lambda = \|M\|$, we omit it.
\end{defin}

Note the difference between this definition and Proposition \ref{settype-bad}.(2) above.  When $\K$ does not have amalgamation or when $\lambda \le \LS (\K)$, it is not clear that this definition is useful. But if $\K$ has amalgamation and $\lambda > \LS (\K)$, the following result of Shelah is fundamental:

\begin{thm}\label{mod-homog-sat}
  Assume that $\K$ is an AEC with amalgamation and let $\lambda > \LS (\K)$. Then $M \in \K$ is $\lambda$-Galois-saturated if and only if it is $\lambda$-model-homogeneous.
\end{thm}

\subsection{Classification Theory}\label{class-thy-sec}  One theme of the classification theory of AECs is what Shelah has dubbed the ``schizophrenia'' of superstability (and other dividing lines) \cite[p.~19]{shelahaecbook}.  Schizophrenia here refers to the fact that, in the elementary framework, dividing lines are given by several equivalent characterizations (e.g.\ stability is no order property or few types), typically with the existence of a definable, combinatorial object on the ``high'' or bad side and some good behavior of forking on the ``low'' or good side.  However, this equivalence relies heavily on compactness or other ideas central to first-order and breaks down when dealing with general AECs.  Thus, the search for stability, superstability, etc. is in part a search for the ``right'' characterization of the dividing line and in part a search for equivalences between the different faces of the dividing line.

One can roughly divide approaches towards the classification of AECs into two categories: local approaches and global approaches\footnote{Monica VanDieren suggested set-theoretic scaffolding and model-theoretic scaffolding as alternate names for the local and global approaches.}. Global approaches typically assume one or more structural properties (such as amalgamation or no maximal models) as well as a classification property (such as categoricity in a high-enough cardinal or Galois stability in a particular cardinality), and attempt to derive good behavior on a tail of cardinals. The local approach is a more ambitious strategy pioneered by Shelah in his book \cite{shelahaecbook}. The idea is to first show (assuming e.g.\ categoricity in a proper class of cardinals) that the AEC has good behavior in some suitable cardinal $\lambda$. Shelah precisely defines ``good behavior in $\lambda$'' as having a good $\lambda$-frame (see Section \ref{frame-sec}). In particular, this implies that the class is superstable in $\lambda$. The second step in the local approach is to argue that good behavior in some $\lambda$ transfers upward to $\lambda^+$ and, if the behavior is good enough, to all cardinals above $\lambda$. Having established global good behavior, one can rely on the tools of the global approach to prove the categoricity conjecture. 

The local approach seems more general but comes with a price: increased complexity, and often the use of non-ZFC axioms (like the weak GCH: $2^\lambda < 2^{\lambda^+}$ for all $\lambda$), as well as stronger categoricity hypotheses. The two approaches are not exclusive. In fact in recent years, tools from local approach have been used and studied in a more global framework. We now briefly survey results in both approaches that do not use tameness.

\subsection{Stability} \label{genstab-subsec}

Once Galois types have been defined, one can define \emph{Galois-stability}:

\begin{defin}
  An AEC $\K$ is \emph{Galois-stable in $\lambda$} if for any $M \in \K$ with $\|M\|\leq \lambda$, we have $|\gS (M)| \le \lambda$.
\end{defin}

One can ask whether there is a notion like forking in stable AECs. The next sections discuss this problem in detail. A first approximation is \emph{$\mu$-splitting}:

\begin{defin}\label{splitting-def}
  Let $\K$ be an AEC, $\mu \ge \LS (\K)$. Assume that $\K$ has amalgamation in $\mu$. Let $M \lea N$ both be in $\K_{\ge \mu}$. A type $p \in \gS^{<\infty} (N)$ \emph{$\mu$-splits} over $M$ if there exists $N_1, N_2 \in \K_\mu$ with $M \lea N_\ell \lea N$, $\ell = 1,2$, and $f: N_1 \cong_{M} N_2$ so that $f (p \rest N_1) \neq p \rest N_2$.
\end{defin}

One of the early results was that $\mu$-splitting has a local character properties in stable AECs:

\begin{lem} \label{stabsplit}
  Let $\K$ be an AEC, $\mu \ge \LS (\K)$. Assume that $\K$ has amalgamation in $\mu$ and is Galois-stable in $\mu$. For any $N \in \K_{\ge \mu}$ and $p \in \gS (N)$, there exists $N_0 \in \K_\mu$ with $N_0 \lea N$ so that $p$ does not $\mu$-split over $N_0$.
\end{lem}

With stability and amalgamation, we also get that there are unique non-$\mu$-splitting extensions to universal models \emph{of the same size}. 

\begin{thm} \label{gen-unique-dns-fact}
  Let $\K$ be an AEC, $\mu \ge \LS (\K)$. Assume that $\K$ has amalgamation in $\mu$ and is Galois-stable in $\mu$. If $M_0, M_1, M_2 \in \K_\mu$ with $M_{1}$ universal over $M_0$\footnote{That is, for every $M' \in \K_\mu$ with $M_0 \lea M'$, there exists $f: M' \xrightarrow[M_0]{} M_1$.}, then each $p \in \gS(M_1)$ that does not $\mu$-split over $M_0$ has a unique extension $q \in \gS(M_2)$ that does not split over $M_0$. Moreover, $p$ is algebraic if and only if $q$ is.
\end{thm}

Note in passing that stability gives existence of universal extensions:

\begin{lem}\label{univ-exist}
  Let $\K$ be an AEC and let $\lambda \ge \LS (\K)$ be such that $\K$ has amalgamation in $\lambda$ and is Galois-stable in $\lambda$. For any $M \in \K_\lambda$, there exists $N \in \K_\lambda$ which is universal over $M$.
\end{lem}

Similar to first-order model theory, there is a notion of an order property in AECs.  The order property is more parametrized due to the lack of compactness.  In the elementary framework, the order property is defined as the existence of a definable order of order type $\omega$. However, the essence of it is that any order type can be defined. Thus the lack of compactness forces us to make the the order property in AECs longer in order to be able to build complicated orders:

\begin{defin}\label{op-def} \
\begin{enumerate}
\item $\K$ has the \emph{$\kappa$-order property of length $\alpha$} if there exists $N \in \K$, $p \in \gS^{<\kappa}(\emptyset; N)$, and $\seq{\ba_i \in {}^{<\kappa} |M|: i < \alpha}$ such that:
  
  $$i < j \Leftrightarrow \gtp (\ba_i \ba_j/\emptyset; N) = p$$
\item $\K$ has the \emph{$\kappa$-order property} if it has the $\kappa$-order property of all lengths.
\item $\K$ has the \emph{order property} if it has the $\kappa$-order property for some $\kappa$.
\end{enumerate}
\end{defin}

From the presentation theorem, having the $\kappa$-order property of all lengths less than $h(\kappa)$ is enough to imply the full $\kappa$-order property. In this case, one can show that $\alpha$ above can be replaced by any linear ordering.

\subsection{Superstability}

A first-order theory $T$ is superstable if it is stable on a tail of cardinals. One might want to adapt this definition to AECs, but it is not clear that it is enough to derive the property that we really want here: an analog of $\kappa (T) = \aleph_0$. A possible candidate is to say that a class is superstable if every type does not $\mu$-split over a finite set. However, splitting is only defined for models, and, as remarked above, types over arbitrary sets are not too well-behaved. Instead, as with Galois types, we take an implication of the desired property as the new definition: no long splitting chains.

\begin{defin}\label{ss-def}
  An AEC $\K$ is \emph{$\mu$-superstable}%\footnotei{WB: It seems as though the ship has sailed, but I'm unenthusiastic about fixing (essentially) ``no long splitting chains'' as \emph{the} definition for superstability among the many options... SV: I tend to agree... Rami introduced this terminology in his survey and I adopted it in my hydra paper, perhaps I was too enthusiastic and naive. While there is evidence it is not a bad definition, calling it ``superstable'' is already a judgment of value rather than a technical description. Maybe a remark could be written about it somewhere (in the historical notes maybe?)? Or alternative terminology could be adopted, but it would contradict several papers... If you want new terminology, I suggest:
%    \begin{enumerate}
%      \item An AEC $\K$ is \emph{$\mu$-reasonable} if the first three conditions above are satisfied.
%      \item An AEC $\K$ has \emph{long $\mu$-splitting chains} if the fourth condition above is not satisfied.
%      \item Then we can state results such as If $\K$ is $\mu$-tame, $\mu$-reasonable, has amalgamation, and no long $\mu$-splitting chains, then it is $\lambda$-reasonable and has no long $\lambda$-splitting chains for all $\lambda \ge \mu$.
%    \end{enumerate}
%    Not sure if this makes it better or worse... *}
 (or \emph{superstable in $\mu$}) if:

  \begin{enumerate}
    \item $\mu \ge \LS (\K)$.
    \item $\K_\mu$ is nonempty, has amalgamation\footnote{This requirement is not made in several other variations of the definition but simplifies notation. See the historical remarks for more.}, joint embedding, and no maximal models.
    \item $\K$ is Galois-stable in $\mu$.
    \item\label{split assm} for all limit ordinal $\delta < \mu^+$ and every increasing continuous sequence $\seq{M_i : i \le \delta}$ in $\K_\mu$ with $M_{i + 1}$ universal over $M_i$ for all $i < \delta$, if $p \in \gS (M_\delta)$, then there exists $i < \delta$ so that $p$ does not $\mu$-split over $M_i$.
  \end{enumerate}
\end{defin}

  If $\K$ is the class of models of a first-order theory $T$, then $\K$ is $\mu$-superstable if and only if $T$ is stable in every $\lambda \ge \mu$.

\begin{remark}
  In (\ref{split assm}), note that $M_{i + 1}$ is required to be universal over $M_i$, rather than just strong extension.  For reasons that we do not completely understand, it unknown whether this variation follows from categoricity (see Theorem \ref{shvi}). On the other hand, it seems to be sufficient for many purposes. In the tame case, the good frames derived from superstability (see Theorem \ref{ss-implies-all}) will have this stronger property.
\end{remark}

Another possible definition of superstability in AECs is the uniqueness of limit models:

\begin{defin}\label{lim-def}
  Let $\K$ be an AEC and let $\mu \ge \LS (\K)$. 

  \begin{enumerate}
    \item A model $M \in \K_{\mu}$ is \emph{$(\mu, \delta)$-limit} for limit $\delta < \mu^+$ if there exists a strictly increasing continuous chain $\seq{M_i \in \K_\mu : i \le \delta}$ such that $M_\delta = M$ and for all $i < \delta$, $M_{i + 1}$ is universal over $M_i$. If we do not specify the $\delta$, it means that there is one.  We say that $M$ is \emph{limit over $N$} when such a chain exists with $M_0 = N$.
    \item $\K$ has \emph{uniqueness of limit models in $\mu$} if whenever $M_0, M_1, M_2 \in \K_\mu$ and both $M_1$ and $M_2$ are limit over $M_0$, then $M_1 \cong_{M_0} M_2$.
    \item $\K$ has \emph{weak uniqueness of limit models in $\mu$} if whenever $M_1, M_2 \in \K_\mu$ are limit models, then $M_1 \cong M_2$ (the difference is that the isomorphism is not required to fix $M_0$).
  \end{enumerate}
\end{defin}

Limit models and their uniqueness have come to occupy a central place in the study of superstability of AECs. $(\mu^+, \mu^+)$-limit models are Galois-saturated, so even weak uniqueness of limit models in $\mu^+$ implies that $(\mu^+, \omega)$-limit models are Galois-saturated.  This tells us that Galois-saturated models can be built in fewer steps than expected, which is reminiscent of first-order characterizations of superstability.  As an added benefit, the analysis of limit models can be carried out in a single cardinal (as opposed to Galois-saturated models, which typically need smaller models) and, thus, lends itself well to the local analysis\footnote{For example, it gives a way to define what it means for a model of size $\LS (\K)$ to be saturated.}.

The following question is still open (the answer is positive for elementary classes):

\begin{question}
  Let $\K$ be an AEC and let $\mu \ge \LS (\K)$. If $\K_\mu$ is nonempty, has amalgamation, joint embedding, no maximal models, and is Galois-stable in $\mu$, do we have that $\K$ has uniqueness of limit models in $\mu$ if and only if $\K$ is superstable in $\mu$?
\end{question}

This phenomenon of having two potentially non-equivalent definitions of superstability that are equivalent in the first-order case is an example of the ``schizophrenia'' of superstability mentioned above%\svnote{Will, you wrote ``...is an example of Shelah's schizophrenia.'' Shelah may or may not be schizophrenic, but I rewrote this sentence to be more politically correct :-)! *}
.

%In \cite{shvi635}, Shelah and Villaveces claimed to prove that superstability implies uniqueness of limit models but an error in their argument was isolated by VanDieren\footnotei{WB: I'm not really a big fan of listing errors, especially when Monica's own fix turned out to be erroneous. SV: How do you suggest to rewrite this? Perhaps the entire end of this entire paragraph and the theorem can be removed. What do you think?}. Recently, she has introduced a symmetry property for $\mu$-splitting which can be assumed to obtain:
Shelah and Villaveces \cite{shvi635} started the investigation of whether superstability implies the uniqueness of limit models.  Eventually, VanDieren introduced a symmetry property for $\mu$-splitting to show the following.

\begin{thm}\label{uq-limit-symmetry}
  If $\K$ is a $\mu$-superstable\footnote{Recall that the definition includes amalgamation and no maximal models in $\mu$.} AEC such that $\mu$-splitting has symmetry, then $\K$ has uniqueness of limit models in $\mu$.
\end{thm}

In fact, the full strength of amalgamation in $\mu$ is not needed, see \cite{vandieren-charact-uq-limit-v2} in this volume for more.

\subsection{Categoricity}\label{categ-notame-sec}

For an AEC $\K$, let us denote by $I (\K, \lambda)$ the number of non-isomorphic models in $\K_\lambda$. We say that $\K$ is \emph{categorical in $\lambda$} if $I (\K, \lambda) = 1$. One of Shelah's motivation for introducing AECs was to make progress on the following test question:

\begin{conjecture}[Shelah's categoricity conjecture for $\Ll_{\omega_1, \omega}$]\label{conj-inf-logic}
  If a sentence $\psi \in \Ll_{\omega_1, \omega}$ is categorical in \emph{some} $\lambda \ge \beth_{\omega_1}$, then it is categorical in \emph{all} $\lambda' \ge \beth_{\omega_1}$.
\end{conjecture}

Note that the lower bound is the Hanf number of this class. One of the best results toward the conjecture is:

\begin{thm}\label{sh87-thm}
  Let $\psi \in \Ll_{\omega_1, \omega}$ be a sentence in a countable language. Assume\footnote{Much weaker set-theoretic hypotheses suffice.} $\mathbf{V} = \mathbf{L}$. If $\psi$ is categorical in all $\aleph_n$, $1 \le n < \omega$, then $\psi$ is categorical in all uncoutable cardinals.
\end{thm}

Shelah's categoricity conjecture for $\Ll_{\omega_1, \omega}$ can be generalized to AECs, either by requiring only ``eventual'' categoricity (Conjecture \ref{shelah-event-conj}) or by asking for a specific Hanf number. 

This makes a difference: using the axiom of replacement and the fact that every AEC $\K$ is determined by its restrictions to models of size at most $\LS (\K)$, it is easy to see that Shelah's eventual categoricity conjecture is equivalent to the following statement: If an AEC is categorical in a proper class of cardinals, then it is categorical on a tail of cardinals. Thus requiring that the Hanf number can in some sense be explicitly computed makes sure that one cannot ``cheat'' and automatically obtain a free upward transfer.

When the Hanf number is $H_1$ (recall Notation \ref{hanf-notation}), we call the resulting statement \emph{Shelah's categoricity conjecture for AECs}. This is widely recognized as the main test question\footnote{It is not expected that solving it will produce a useful lemma in solving other problems.  Rather, like Morley's Theorem, it is expected that the solution will necessitate the development of ideas that will be useful in solving other problems.} in the study of AECs.

\begin{conjecture}\label{categ-conj-aec}
  If an AEC $\K$ is categorical in \emph{some} $\lambda > H_1$, then it is categorical in \emph{all} $\lambda' \ge H_1$.
\end{conjecture}

One of the milestone results in the global approach to this conjecture is Shelah's downward transfer from a successor in AECs with amalgamation.

\begin{thm}\label{shelah-succ-downward}
  Let $\K$ be an AEC with amalgamation. If $\K$ is categorical in a \emph{successor} $\lambda \ge H_2$, then $\K$ is categorical in all $\mu \in [H_2, \lambda]$.
\end{thm}

The structure of the proof involves first deriving a weak version of tameness from the categoricity assumption (see Example \ref{examples-subsec}.(\ref{weaktamecat})). A striking feature of this result (and several other categoricity transfers) is the successor requirement, which is, of course, missing from similar results in the first-order case. Removing it is a major open question, even in the tame framework (see Shelah \cite[Problem 6.14]{sh702}). We see at least three difficulties when working with an AEC categorical in a limit cardinal $\lambda > \LS (\K)$:

\begin{enumerate}
  \item It is not clear that the model of size $\lambda$ should be Galois-saturated, see Question \ref{sat-quest}. 
  \item It is not clear how to transfer ``internal characterizations'' of categoricity such as no Vaughtian pairs or unidimensionality. In the first-order framework, compactness is a key tool to achieve this.
  \item It is not clear how to even get that categoricity implies such an internal characterization (assuming $\lambda = \lambda_0^+$ is a successor, there is a relatively straightforward argument for the non-existence of Vaughtian pairs in $\lambda_0$). In the first-order framework, all the arguments we are aware of use in some way \emph{primary models} but here we do not know if they exist or are well-behaved. For example, we cannot imitate the classical argument that primary models are primes (this relies on the compactness theorem).
\end{enumerate}

Assuming tameness, the first two issues can be solved (see Theorem \ref{ss-categ} and the proof of Theorem \ref{succ-transfer-2}). It is currently not known how to solve the third in general, but adding the assumption that $\K$ has prime models over sets of the form $M \cup \{a\}$ is enough. See Theorem \ref{event-categ-primes}.

%% In some sense, the successor requirement can be seen as a fallout from the lack of compactness: categoricity proofs often proceed by contradiction to build a non-saturated model in the categoricity cardinal.  With compactness, the non-saturation can be transferred to all cardinalities on general grounds\svnote{Really? How? I thought that the reason was the existence of primary models...]\\WB: Maybe I'm wrong on why.  My intuition at some point was that compactness from first order meant that Vaughtian pairs could be moved around in cardinality very easily.  I think this is what I meant rather than non-saturated.  Also, should we move this paragraph earlier to the discussion on categoricity in general AECs? SV: I moved it. Assuming tameness, there is a way to move Vaughtian pairs around, the problem is really how to get no Vaughtian pair in the first place... For this the problem (as I understand it now) is having prime models...}. Without compactness, the construction generally works from a chain of $\lambda$-sized models which contradicts categoricity in $\lambda^+$.  

A key tool in the proof of Theorem \ref{shelah-succ-downward} is the existence of Ehrenfeucht-Mostowski models which follow from the presentation theorem. In AECs with amalgamation and no maximal models, several structural properties can be derived below the categoricity cardinal. For example:

\begin{thm}[The Shelah-Villaveces theorem, \cite{shvi635}]\label{shvi}
  Let $\K$ be an AEC with amalgamation and no maximal models. Let $\mu \ge \LS (\K)$. If $\K$ is categorical in a $\lambda > \mu$, then $\K$ is $\mu$-superstable.
\end{thm}

Note that Theorem \ref{shvi} fails to generalize to $\lambda \ge \mu$. In general, $\K$ may not even be Galois-stable in $\lambda$, see the Hart-Shelah example (Section \ref{counterex-ssec} below). In the presence of tameness, the difficulty disappears: superstability can be transferred all the way up (see Theorem \ref{ss-implies-all}). This seems to be a recurring feature of the study of AECs without tameness: some structure can be established below the categoricity cardinal (using tools such as Ehrenfeucht-Mostowski models), but transferring this structure upward is hard due to the lack of locality. For example, in the absence of tameness the following question is open:

\begin{question}\label{sat-quest}
  Let $\K$ be an AEC with amalgamation and no maximal models. If $\K$ is categorical in a $\lambda > \LS (\K)$, is the model of size $\lambda$ Galois-saturated?
\end{question}

It is easy to see that (if $\cf{\lambda} > \LS (\K)$), the model of size $\lambda$ is $\cf{\lambda}$-Galois-saturated. Recently, it has been shown that categoricity in a high-enough cardinal implies some degree of saturation:

\begin{thm}\label{sat-categ}
  Let $\K$ be an AEC with amalgamation and no maximal models. Let $\lambda \ge \chi > \LS (\K)$. If $\K$ is categorical in $\lambda$ and $\lambda \ge \hanf{\chi}$, then the model of size $\lambda$ is $\chi$-Galois-saturated.
\end{thm}

What about the uniqueness of limit models? In the course of establishing Theorem \ref{shelah-succ-downward}, Shelah proves that categoricity in a successor $\lambda$ implies weak uniqueness of limit models in all $\mu < \lambda$. Recently, VanDieren and the second author have shown:

\begin{thm}\label{uq-limit-categ}
  Let $\K$ be an AEC with amalgamation and no maximal models. Let $\mu \ge \LS (\K)$. If $\K$ is categorical in a $\lambda \ge \hanf{\mu}$, then $\K$ has uniqueness of limit models in $\mu$.
\end{thm}

\subsection{Good frames} \label{frame-sec}

Roughly speaking, an AEC $\K$ \emph{has a good $\lambda$-frame} if it is well-behaved in $\lambda$ (i.e.\ it is nonempty, has amalgamation, joint embedding, no maximal model, and is Galois-stable, all in $\lambda$) and there is a forking-like notion for types of length one over models in $\K_\lambda$ that behaves like forking in superstable first-order theories\footnote{There is an additional parameter, the set of \emph{basic} types.  These are a dense set of types over models of size $\lambda$ such that forking is only required to behave well with respect to them.  However, basic types play little role in the discussion of tameness (and eventually are eliminated in most cases even in general AECs), so we do not discuss them here, see the historical remarks.}. In particular, it is $\lambda$-superstable. One motivation for good frames was the following (still open) question:

\begin{question}\label{baldwin-question}
  If an AEC is categorical in $\lambda$ and $\lambda^+$, does it have a model of size $\lambda^{++}$?
\end{question}

Now it can be shown that if $\K$ has a good $\lambda$-frame (or even just $\lambda$-superstable), then it has a model of size $\lambda^{++}$. Thus it would be enough to obtain a good frame to solve the question. Shelah has shown the following:

\begin{thm}\label{shelah-local-good-frame}
  Assume $2^{\lambda} < 2^{\lambda^+} < 2^{\lambda^{++}}$.
  
  Let $\K$ be an AEC and let $\lambda \ge \LS (\K)$. If:

  \begin{enumerate}
    \item $\K$ is categorical in $\lambda$ and $\lambda^+$.
    \item $0 < I (\K, \lambda^{++}) < \mu_{\text{unif}} (\lambda^{++}, 2^{\lambda^+})$\footnote{The cardinal $\mu_{\text{unif}} (\lambda^{++}, 2^{\lambda^+})$ should be interpreted  as $2^{\lambda^{++}}$; this is true when $\lambda \geq \beth_\omega$ and there is no example of inequality when $2^{\lambda^+}< 2^{\lambda^{++}}$.}.
  \end{enumerate}
  
  Then $\K$ has a good $\lambda^+$-frame.
\end{thm}

\begin{cor}\label{baldwin-q-answer}
  Assume $2^{\lambda} < 2^{\lambda^+} < 2^{\lambda^{++}}$. If $\K$ is categorical in $\lambda$, $\lambda^{+}$, and $\lambda^{++}$, then $\K$ has a model of size $\lambda^{+++}$.
\end{cor}

Note the non-ZFC assumptions as well as the strong categoricity hypothesis\footnote{It goes without saying that the proof is also long and complex, see the historical remarks.}. We will see that this can be removed in the tame framework, or even already by making some weaker (but global) assumptions than tameness.

Recall from the beginning of Section \ref{class-thy-sec} that Shelah's local approach aims to transfer good behavior in $\lambda$ upward. The successor step is to turn a good $\lambda$-frame into a good $\lambda^+$ frame . Shelah says a good $\lambda$-frame is \emph{successful} if it satisfies a certain (strong) technical condition that allows it to extend it to a good $\lambda^{+}$-frame. 

\begin{thm}\label{successful-cond}
  Assume $2^{\lambda} < 2^{\lambda^+} < 2^{\lambda^{++}}$. If an AEC $\K$ has a good $\lambda$-frame $\s$ and $0 < I (\K, \lambda^{++}) < \mu_{\text{unif}} (\lambda^{++}, 2^{\lambda^+})$, then there exists a good $\lambda^{+}$-frame $\s^+$ with underlying class the Galois-saturated models of size $\lambda^+$ (the ordering will also be different).
\end{thm}

The proof goes by showing that the weak GCH and few models assumptions imply that any good frame is successful.

So assuming weak GCH and few models in every $\lambda^{+n}$, one obtains an increasing sequence $\bar{\s} = \s, \s^{+}, \s^{++}, \ldots$ of good frames. One of the main results of Shelah's book is that the natural limit of $\bar{\s}$ is also a good frame (the strategy is to show that a good frame in the sequence is \emph{excellent}). Let us say that a good frame is \emph{$\omega$-successful} if $\s^{+n}$ is successful for all $n < \omega$. At the end of Chapter III of his book, Shelah claims the following result and promises a proof in \cite{sh842}:

\begin{claim}\label{claim-xxx}
  Assume $2^{\lambda^{+n}} < 2^{\lambda^{+(n + 1)}}$ for all $n < \omega$. If an AEC $\K$ has an $\omega$-successful good $\lambda$-frame, is categorical in $\lambda$, and $\Ksatp{\lambda^{+\omega}}$ (the class of $\lambda^{+\omega}$-Galois-saturated models in $\K$) is categorical in a $\lambda' > \lambda^{+\omega}$, then $\Ksatp{\lambda^{+\omega}}$ is categorical in all $\lambda'' > \lambda^{+\omega}$.
\end{claim}

Can one build a good frame in ZFC? In Chapter IV of his book, Shelah proves: 

\begin{thm}\label{good-frame-categ}
  Let $\K$ be an AEC categorical in cardinals of arbitrarily high cofinality. Then there exists a cardinal $\lambda$ such that $\K$ is categorical in $\lambda$ and $\K$ has a good $\lambda$-frame.
\end{thm}
\begin{thm}\label{good-frame-categ-2}
  Let $\K$ be an AEC with amalgamation and no maximal models. If $\K$ is categorical in a $\lambda \ge \hanf{\aleph_{\LS (\K)^+}}$, then there exists $\mu < \aleph_{\LS (\K)^+}$ such that $\Ksatp{\mu}$ has a good $\mu$-frame.
\end{thm}

The proofs of both theorems first get some tameness (and amalgamation in the first case), and then use it to define a good frame in $\lambda$ by making use of the lower cardinals (as in Theorem \ref{stable-forking-def}).

Assuming amalgamation and weak GCH, Shelah shows that the good frame can be taken to be $\omega$-successful. Combining this with Claim \ref{claim-xxx}, Shelah deduces the eventual categoricity conjecture in AECs with amalgamation:

\begin{thm}\label{categ-transfer-wgch}
  Assume Claim \ref{claim-xxx} and $2^{\theta} < 2^{\theta^+}$ for all cardinals $\theta$. Let $\K$ be an AEC with amalgamation. If $\K$ is categorical in \emph{some} $\lambda \ge \hanf{\aleph_{\LS (\K)^+}}$, then $\K$ is categorical in \emph{all} $\lambda' \ge \hanf{\aleph_{\LS (\K)^+}}$.
\end{thm}

Note that the first steps in the proof are again proving enough tameness to make the construction of an $\omega$-successful good frame.

\section{Tameness: what and where} \label{tame-sec}

\subsection{What -- Definitions and basic results}\label{what-subsec}

Syntactic types have nice locality properties: different types must differ on a formula and this difference can be seen by restricting the type to the finite\footnote{Or larger if the logic allows infinitely many free variables.} set of parameters in such a formula.  Galois types do not necessarily have this property.  Indeed, assuming the existence of a monster model $\sea$, this would imply a strong closure property on $\Aut (\sea)$.  Nonetheless, a generalization of this idea, called \emph{tameness}, has become a key tool in the study of AECs.

For a set $A$, we write $P_{\kappa} A$ for the collection of subsets of $A$ of size less than $\kappa$. We also define an analog notation for models: for $M \in \K_{\geq \kappa}$: $$P^*_\kappa M : = \{M_0 \in \K_{<\kappa} : M_0 \lea M\}$$

\begin{defin}\label{tamedef}
$\K$ is $<\kappa$-tame if, for all $M \in \K$ and $p \neq q \in \gS^1(M)$, there is $A \in P_\kappa |M|$ such that $p \rest A \neq q \rest A$.
\end{defin}

For $\kappa > \LS(\K)$, it is equivalent if we quantify over $P_\kappa^* M$ (models) rather than $P^*_\kappa |M|$ (sets). Quantifying over sets is useful to isolate notions such as $<\aleph_0$-tameness. Several parametrizations (e.g.\, of the length of type) and variations exist.  Below we list a few that we use; note that, in all cases, writing ``$\kappa$'' in place of ``$<\kappa$'' should be interpreted as ``$<\kappa^+$''.

\begin{defin}  \label{tame-var-def}
Suppose $\K$ is an AEC with $\kappa \leq \lambda$.
\begin{enumerate}

	\item $\K$ is \emph{$(< \kappa, \lambda)$-tame} if for any $M \in \K_\lambda$ and $p \neq q \in \gS^1(M)$, there is some $A \in P_{\kappa} |M|$ such that $p \rest A \neq q \rest A$.
	
	\item $\K$ is \emph{$< \kappa$-type short} if for any $M \in \K$, index set $I$, and $p \neq q \in \gS^I(M)$, there is some $I_0 \in P_{\kappa} I$ such that $p^{I_0} \neq q^{I_0}$.

	\item $\K$ is \emph{$\kappa$-local} if for any increasing, continuous $\seq{M_i \in \K : i \leq \kappa}$ and any $p \neq q \in \gS(M_\kappa)$, there is $i_0 < \kappa$ such that $p \rest M_{i_0} \neq q \rest M_{i_0}$.

	\item $\K$ is \emph{$\kappa$-compact} if for any increasing, continuous $\seq{M_i : i \leq \kappa}$ and increasing $\seq{p_i \in \gS(M_i) : i < \kappa}$, there is $p \in \gS(M)$ such that $p_i \leq p$ for all $i < \kappa$.
	
	\item $\K$ is \emph{fully $<\kappa$-tame and -type short} if for any $M \in \K$, index set $I$, and $p \neq q \in \gS^I(M)$, there are $A \in P_\kappa |M|$ and $I_0 \in P_\kappa I$ such that $p^{I_0} \rest A \neq q^{I_0} \rest A$.

\end{enumerate}	

When $\kappa$ is omitted, we mean that there exists $\kappa$ such that the property holds at $\kappa$. For example, ``$\K$ is tame'' means that there exists $\kappa$ such that $\K$ is $<\kappa$-tame.  Note that definitions of locality and compactness implicitly assume $\kappa$ is regular.
\end{defin}

These types of properties are often called locality properties for AECs because they assert, in different ways, that Galois types are locally defined.

%% A particularity of the definition is whether the small witnesses of difference are required to be models or allowed to be sets.  We give the definition in terms of models because it is standard, but allowing sets replaces ``$M_0 \in P^*_\kappa M$'' in Definition \ref{tamedef} with ``$A_0 \in P_\kappa |M|$''.  We specify which we mean when necessary, but, as a general rule, results proving a particular class is tame actually prove tameness over sets, while results assuming tameness actually assume tameness over models.  The former is particularly useful as it allows the isolation of $<\aleph_0$-tameness.

Each of these notions also has a \emph{weak} version: weak $<\kappa$-tameness, etc.  This variation means that the property holds when the domain is Galois-saturated.

A brief summary of the ideas is below.  In each (and throughout this paper), ``small'' is used to mean ``of size less than $\kappa$''.

\begin{itemize}
	\item $<\kappa$-tameness says that different types differ over some small subset of the domain.
	\item $<\kappa$-type shortness says that different types differ over some small subset of their length.
	\item $\kappa$-locality says that each increasing chain of Galois types of length $\kappa$ has \emph{at most} one upper bound.
	\item $\kappa$-compactness says that each increasing chain of Galois types of length $\kappa$ has \emph{at least} one upper bound.\footnote{All AECs are $\omega$-compact and global compactness statements have large cardinal strength; see \cite[Section 2]{sh932}.}
\end{itemize}

A combination of tameness and type shortness allows us to conceptualize Galois types as sets of smaller types.%\svnote{I think the two parameter definition is more confusing than it is worth at that point. If it is needed later, we can introduce it at that point. Also it might be worth giving the definition with types over sets, since it \emph{is} used later. You have a long remark about it after which is very good and should stay.\\WB: I moved/changed the remark to earler.  I also kept the parameterizations, while emphasizing the main definition. SV: I don't mind keeping the parametrized definitions so much. However I do think that defining tameness with types over sets is the way to go: this is a generalization of the classical definition that has better closure properties (closed under taking a tail of the AEC). The classical definition creates many ugly situations when you get e.g.\ $\LS (\K)$-tameness but $(<\aleph_0)$-shortness, for no good reason except the definition. In addition it is used later. Why have two definitions when we can have one? If you really want, you can state as a remark than when $\kappa \ge \LS (\K)$, $\kappa$-tameness is equivalent to ``for every $p \neq q \in \gS (M)$, there exists $M_0 \lea M$ in $\K_{\kappa}$ so that $p \rest M_0 \neq q \rest M_0$''. Further discussion can also be added to the historical notes section. I think people have been scared of Galois types over sets historically (when we asked Rami about it, he spoke of a ``ghost''), I don't think it is a good thing to perpetuate that tradition.}\svnote{Also, I think one should write $(<\kappa)$-tame, and not $<\kappa$-tame...}

There are several relations between the properties:
\begin{prop}\label{easy-implications} \
\begin{enumerate}
	\item For $\kappa > \LS (\K)$, $<\kappa$-type shortness implies $<\kappa$-tameness.
	\item $<\cf \kappa$-tameness implies $\kappa$-locality.
	\item $\mu$-locality for all $\mu < \lambda$ implies $(\LS(\K), \lambda)$-tameness.
	\item $\mu$-locality for all $\mu < \lambda$ implies $\lambda$-compactness.
	
\end{enumerate}
\end{prop}

As discussed, one of the draws of working in a short and tame AEC is that Galois types behave much more like first-order syntactic types in the sense that a Galois type $p \in \gS(M)$ can be identified with the collection $\{p^{I_0} \rest M_0 : I_0 \in P_\kappa \ell(p) \text{ and } M_0 \in P^*_\kappa M\}$ of its small restrictions: 

\begin{prop}\label{injective-map}
$\K$ is fully $<\kappa$-tame and -type short if and only if the map:
$$p \in \gS(M) \mapsto \{p^{I_0} \rest M_0 : I_0 \in P_\kappa I, M_0 \in P^*_\kappa M\}$$
is injective.
\end{prop}

In fact, one can see these small restrictions as formulas (this will be used later to generalize heir and coheir to AECs).  This productive intuition can be made exact using \emph{Galois Morleyization}. Start with an AEC $\K$ and add to the language an $\alpha$-ary predicate $R_p$ for each $N \in \K$, each $p \in \gS^{\alpha}(\emptyset; N)$, and each $\alpha < \kappa$. This gives us an infinitary language $\bigL$. Then expand each $M \in \K$ to a $\bigL$-structure $\widehat M$ by setting $R_p(\ba)$ to be true in $M$ if and only if $\gtp(\ba/\emptyset; M) = p$. We obtain a class $\bigKp{\kappa} := \{\widehat{M} \mid M \in \K\}$. $\bigK$ has relations of infinite arity but it still behaves like an AEC. We call $\bigKp{\kappa}$ the \emph{$<\kappa$-Galois Morleyization} of $\K$. The connection between tameness and $\bigK$ is given by the following theorem:

\begin{theorem}\label{morleyization-thm}
  Let $\K$ be an AEC. The following are equivalent:
\begin{enumerate}
	\item\label{morleyization-1} $\K$ is fully $<\kappa$-tame and -type short.
	\item\label{morleyization-2} The map $\gtp(\bb/M; N) \mapsto \text{tp}_{\text{qf-}\Ll_{\kappa, \kappa} (\bigL)}(\bb/\widehat{M};\widehat{N})$  is an injection. 
\end{enumerate}
\end{theorem}

Here, the Galois type is computed in $\K$, and the type on the right is the (syntactic) quantifier-free $\Ll_{\kappa, \kappa}$-type in the language $\bigL$. Note that the locality hypothesis in (\ref{morleyization-1}) can be weakened to $<\kappa$-tameness if in (\ref{morleyization-2}) we ask that $\ell (\bb) = 1$. Several other variations are possible.

The Galois Morleyization gives a way to directly use syntactic tools (such as the results of stability theory inside a model, see for example \cite[Chapter V.A]{shelahaecbook2}) in the study of tame AECs. See for example Theorem \ref{stab-spectrum}.

Another way to see tameness is as a topological separation principle: consider the set $X_M$ of Galois types over $M$. For a fixed $\kappa$, we can give a topology on $X_M$ with basis given by sets of the form $U_{p, A} := \{q \in \gS (M) \mid A \subseteq |M| \land q \rest A = p\}$, for $p$ a Galois type over $A$ and $|A| < \kappa$. This is the same topology as that generated by quantifier-free $\Ll_{\kappa, \kappa}$-formulas in the $<\kappa$-Galois Morleyization. Thus one can show:

\begin{thm}\label{tameness-topo}
  Let $\K$ be an AEC and let $\lambda \ge \LS (\K)$. $\K$ is $(<\kappa, \lambda)$-tame if and only if for any $M \in \K_{\lambda}$, the topology on $X_M$ defined above is Hausdorff.
\end{thm}

\subsection{Where -- Examples and counterexamples} \subsubsection{Examples}\label{examples-subsec}  Several ``mathematically interesting'' classes turn out to be tame.  Moreover, there are several general ways to \emph{derive} tameness from structural assumptions.  We list some here, roughly in decreasing order of generality.

\begin{enumerate}
	\item {\bf Locality from large cardinals} \label{tamelc}\\
	Large cardinals $\kappa$ allow the generalization of compactness results from first-order logic to $\Ll_{\kappa, \omega}$ in various ways (see, for instance, \cite[Lemma 20.2]{jechbook}).  Since tameness is a weak form of compactness, these generalizations correspond to compactness results in AECs that can be ``captured'' by $\Ll_{\kappa, \omega}$.  We state a simple version of these results here:
	\begin{thm} \label{tamelc-fact}
	Suppose $\K$ is an AEC with $\LS(\K) < \kappa$.
	\begin{itemize}
		\item If $\kappa$ is weakly compact, then $\K$ is $(<\kappa, \kappa)$-tame.
		\item If $\kappa$ is measurable, then $\K$ is $\kappa$-local.
		\item If $\kappa$ is strongly compact, then $\K$ is fully $<\kappa$-tame and -type short.
	\end{itemize}
	\end{thm}
	These results can be strengthened in various ways.  First, they apply also to AECs that are explicitly axiomatized in $\Ll_{\kappa, \omega}$. The key fact is that ultraproducts by $\kappa$-complete ultrafilters preserve the AEC (the proof uses the presentation theorem, Theorem \ref{pres-thm}).  Second, each large cardinal can be replaced by its ``almost'' version: for example, almost strongly compact means that, \emph{for each $\delta < \kappa$}, $\Ll_{\delta, \delta}$ is $\kappa$-compact; equivalently, given a $\kappa$-complete filter, \emph{for each $\delta < \kappa$}, it can be extended to a \emph{$\delta$-complete} ultrafilter. See \cite[Definition 2.1]{lc-tame-v2} for a full list of the ``almost'' versions.
	
        Note that other structural properties (such as amalgamation) follow from the combination of large cardinals with categoricity. Thus these large cardinals make the development of a structure theory (culminating for example in the existence of well-behaved independence notion, see Corollary \ref{lc-cor}) much easier.	
	
	\item {\bf Weak tameness from categoricity under amalgamation}\label{weaktamecat}\\
          Recall from Section \ref{what-subsec} that an AEC $\K$ is $(\chi_0, <\chi)$-weakly tame if for every \emph{Galois-saturated} $M \in \K_{<\chi}$, every $p \neq q \in \gS (M)$, there exists $M_0 \lea M$ with $M_0 \in \K_{\le \chi_0}$ such that $p \rest M_0 \neq q \rest M_0$. It is known that, in AECs with amalgamation categorical in a sufficiently high cardinal, weak tameness holds below the categoricity cardinal. More precisely: 

          \begin{thm}\label{tameness-from-categ}
            Assume that $\K$ is an AEC with amalgamation and no maximal models which is categorical in a $\lambda > \LS (\K)$. 

            \begin{enumerate}
            \item Let $\chi$ be a limit cardinal such that $\cf{\chi} > \LS (\K)$. If the model of size $\lambda$ is $\chi$-Galois-saturated, then there exists $\chi_0 < \chi$ such that $\K$ is $(\chi_0, <\chi)$-weakly tame.
            \item If the model of size $\lambda$ is $H_1$-Galois-saturated, then there exists $\chi_0 < H_1$ such that whenever $\chi \ge H_1$ is so that the model of size $\lambda$ is $\chi$-Galois-saturated, we have that $\K$ is $(\chi_0, <\chi)$-weakly tame\footnote{Note that $\chi_0$ does not depend on $\chi$.}.
            \end{enumerate}            
          \end{thm}
          \begin{remark}
          The model in the categoricity cardinal $\lambda$ is $\chi$-Galois-saturated whenever $\cf{\lambda} \ge \chi$ (e.g.\ if $\lambda$ is a successor) or (by Theorem \ref{sat-categ}) if\footnote{A more clever application of Theorem \ref{sat-categ} shows that it is enough to have $\lambda \ge \sup_{\theta < \chi} \hanf{\theta^+}$.} $\lambda \ge \hanf{\chi}$.
          \end{remark}
          
	  The proof of Theorem \ref{tameness-from-categ} heavily uses Ehrenfeucht-Mostowski models to transfer the behavior below $H_1$ to a larger model that is generated by a nice enough linear order.  Then the categoricity assumption is used to embed every model of size $\chi$ into such a model of size $\lambda$.

          Theorem \ref{tameness-from-categ} is key to prove several of the categoricity transfers listed in Section \ref{categ-notame-sec}.
        \item {\bf Tameness from categoricity and large cardinals}\label{tame-categ-lc} \\
          The hypotheses in the last two examples can be combined advantageously. 

          \begin{thm}\label{tameness-from-categ-2}
          Let $\K$ be an AEC and let $\kappa > \LS (\K)$ be a measurable cardinal. If $\K$ is categorical in a $\lambda \ge \kappa$, then $\K_{[\kappa, \lambda)}$ has amalgamation and is $(\kappa, <\lambda)$-tame.
          \end{thm}
          
          In particular, if there exists a proper class of measurable cardinals and $\K$ is categorical in a proper class of cardinals, then $\K$ is tame. It is conjectured that the large cardinal hypothesis is not necessary.  Note that the tameness here is ``full'', i.e.\ not the weak tameness in Theorem \ref{tameness-from-categ}.

	\item {\bf Tameness from stable forking} \label{stabletame}\\
	Suppose that the AEC $\K$ has amalgamation and a stable ``forking-like'' relation $\nf$ (see Definition \ref{stable-indep}). That is, we ask that there is a notion ``$p \in \gS (N)$ does not fork over $M$'' for $M \lea N$ satisfying the usual monotonicity properties, uniqueness, and local character\footnote{The extension property is not needed here.} : there exists a cardinal $\bkappa = \bkappa (\nf)$ such that for every $p \in \gS (N)$, there is $M \lea N$ of size less than $\bkappa$ such that $p$ does not fork over $M$ (see more on such relations in Section \ref{tame-indep-sec}).

        Then, given any two types $p, q \in \gS(N)$ we can find $M \lea N$ over which both types do not fork over and so that $\|M\| < \bkappa$.  If $p \rest M = q \rest M$, then uniqueness implies $p = q$.  Thus, $\K$ is $(<\bkappa)$-tame.
	
	\item {\bf Universal Classes} \label{universalclasses}\\
          A \emph{universal class} is a class $\K$ of structures in a fixed language $L (\K)$ that is closed under isomorphism, substructure, and unions of increasing chains. In particular, $(\K, \subseteq)$ is an AEC with Löwenheim-Skolem-Tarski number $|L (\K)| + \aleph_0$.

          In a universal class, any partial isomorphism extends uniquely to an isomorphism (just take the closure under the functions). This fact is key in the proof of:

          \begin{thm}\label{tame-uc}
            Any universal class is fully $(<\aleph_0)$-tame and short.
          \end{thm}
          
          Thus, for instance, the class of locally finite groups (ordered with subgroup) is tame. Theorem \ref{tame-uc} generalizes to any AEC $\K$ equipped with a notion of ``generated by'' which is (in a sense) canonical (for universal classes, this notion is just the closure under the functions). Note that this does not need to assume that $\K$ has amalgamation.
	
	\item {\bf Tame finitary AECs}\label{finitaryAEC}\\
	  A finitary AEC $\K$ is defined by several properties (including amalgamation and $\LS (\K) = \aleph_0$), but the key notion is that the strong substructure relation $\lea$ has \emph{finite character}.  This means that, for $M, N \in \K$, we have $M \lea N$ if and only if $M \subseteq N$ and:
	\begin{center}
	For every $\ba \in {}^{<\omega} M$, we have that $\gtp(\ba/\emptyset; M) = \gtp(\ba/\emptyset; N)$.
	\end{center}
	This means that there is a finitary test for when $\lea$ holds between two models that are already known to be members of $\K$.  This definition is motivated by the observation that this condition holds for any AEC axiomatized in a countable fragment of $\Ll_{\omega_1, \omega}$ by the Tarski-Vaught test\footnote{Kueker \cite{kueker2008} has asked whether any finitary AEC must be $L_{\infty, \omega}$-axiomatizable.}. Homogeneous model theory can be seen as a special case of the study of finitary AECs. Hyttinen and K\"{e}s\"{a}la have shown that every $\aleph_0$-stable $\aleph_0$-tame finitary AEC is $(<\aleph_0)$-tame. These classes seem very amenable to some classification theory. For example, an $\aleph_0$-tame finitary AEC categorical in some uncountable $\lambda$ is categorical in \emph{all} $\lambda' \ge \min (\lambda, H_1)$. Recent work has even developed some geometric stability theory in a larger class (Finite $U$-Rank classes, which included quasiminimal classes below) \cite{group-config-kangas-apal}.
	
	\item {\bf Homogeneous model theory}\label{hommod}\\
	Homogeneous model theory takes place in the context of a large ``monster model'' (for a first-order theory $T$) that omits a set of types $D$, but is still as saturated as possible with respect to this omission.  The notion of ``as saturated as possible'' is captured by requiring it to be sequentially homogeneous rather than model homogeneous. Note that the particular case when $D = \emptyset$ is the elementary case. In this context, amalgamation, joint embedding, and no maximal models hold for free and Galois types are first-order syntactic types.  This identification means that the AEC of models of $T$ omitting $D$ (ordered with elementary substructure) is fully $(<\aleph_0)$-tame and short. Homogeneous model theory has a rich classification theory in its own right, with connections to continuous first-order logic (see the historical remarks).
	
	\item {\bf Averageable Classes} \label{avclass}\\
	Averageable classes are type omitting classes $\EC(T, \Gamma)$ (ordered with a relation $\lea$) that are nice enough to have a relativized ultraproduct that preserves the omission of types in $\Gamma$ and satisfies enough of \L o\'{s}' Theorem to interact well with $\lea$.  This relativized ultraproduct gives enough compactness to show that types are syntactic (and much more), which implies that an averageable class is fully $(<\aleph_0)$-tame and short.  Examples of averageable classes include torsion modules over PIDs and densely ordered abelian groups with a cofinal and coinitial $\mathbb{Z}$-chain.
	
	\item {\bf Continuous first-order logic} \label{tdense}\\
	Continuous first-order logic can be studied in a fragment of $\Ll_{\omega_1, \omega}$ by using the infinitary logic to have a standard copy of $\mathbb{Q}$ and then studying dense subsets of complete metric spaces.  Although the logic $\Ll_{\omega_1, \omega}$ is incompact, the fragment necessary to code this information is compact (as evidenced by the metric ultrapower and compactness results in continuous first-order logic), so the classes are fully $(<\aleph_0)$-tame and short.\\
	Beyond first-order, continuous model theory can be done in the so-called metric AECs, where a notion of tameness ($d$-tameness) can also be defined.
		
	\item {\bf Quasiminimal Classes} \label{quasimin}\\
	A quasiminimal class is an AEC satisfying certain additional axioms; most importantly, the structures carry a pregeometry with certain nice properties. The axioms directly imply that Galois types over \emph{countable} models are quantifier-free first-order types, and the excellence axiom can be used to transfer this to uncountable models. Therefore quasiminimal classes are $<\aleph_0$-tame.  Examples of quasiminimal classes include covers of $\mathbb{C}^\times$ and Zilber fields with pseudoexponentiation.  Note that it can be shown (from the countable closure axiom) that these classes are strictly $\Ll_{\omega_1, \omega}(Q)$-definable. This gives important examples of categorical AECs that are not finitary.
		
	\item {\bf $\lambda$-saturated models of a superstable first-order theory}\\
          Let $T$ be a first-order superstable theory. We know that unions of increasing chains of $\lambda$-saturated models are $\lambda$-saturated and that models of size $\lambda$ have saturated extensions of size at most $\lambda + 2^{|T|}$.  Thus the class of $\lambda$-saturated models of $T$ (ordered with elementary substructure) forms an AEC $\K^T_\lambda$ with $\LS(\K^T_\lambda) \le \lambda + 2^{|T|}$.  Furthermore, this class has a monster model and is fully $(<\aleph_0)$-tame and short.  %These properties can be seen either by use of good ultrafilters or by forming saturated extensions.
	
	\item {\bf Superior AECs}\\
          Superior AECs are a generalization of what some call excellent classes. An AEC is \emph{superior} if it carries an axiomatic notion of forking for which one can state multi-dimensional uniqueness and extension properties. A combination of these gives some tameness:
	\begin{thm}
	Let $\K$ be a superior AEC with weak $(\lambda, 2)$-uniqueness and $\lambda$-extension for some $\lambda \geq \LS(\K) + \bkappa(\K)$.  Then $\K$ is $\lambda^+$-local.  In particular, it is $(\lambda, \lambda^+)$-tame.
	\end{thm}	
	
	\item {\bf Hrushovski fusions} \label{hrushfus}\\
	Villaveces and Zambrano have studied Hrushovski's method of fusing pregeometries over disjoint languages as an AEC with strong substructure being given by self-sufficient embedding.  They show that these classes satisfy a weakening of independent 3-amalgamation.  This weakening is still enough to show, as with superior AECs, that the classes are $\LS(\K)$-tame.
	
	\item {\bf ${}^\perp N$ when $N$ is an abelian group} \label{nperp}\\
	Given a module $N$, ${}^\perp N$ is the class of modules $\{M : \text{Ext}^n(M, N) = 0 \text{ for all 1 }\leq n < \omega\}$.  We make this into an AEC by setting $M \lea_\perp M'$ if and only if $M'/M \in {}^\perp N$.  If $N$ is an abelian group, then ${}^\perp N$ is set of all abelian groups that are $p$-torsion free for all $p$ in some collection of primes $P$.
	
	\begin{thm}
	If $N$ is an abelian group, then ${}^\perp N$ is $<\aleph_0$-tame.
	\end{thm}
	
	Moreover, such a ${}^\perp N$ is Galois-stable in exactly the cardinals $\lambda = \lambda^\omega$.
	
	\item {\bf Algebraically closed, rank one valued fields} \label{valfields}\\
	Let $\text{ACVF}_{\mathbb{R}}$ be the $\Ll_{\omega_1, \omega}$-theory of an algebraically closed valued field such that the value group is Archimedean; equivalently, the value group can be embedded into $\mathbb{R}$.  After fixing the characteristic, this AEC has a monster model and Galois types are determined by syntactic types.  Thus the class is fully $<\aleph_0$-tame and -type short.  This determination of Galois types can be seen either through algebraic arguments or the construction of an appropriate ultraproduct.
	
	Such a class cannot have an uncountable ordered sequence, so it has the $\aleph_0$-order property of length $\alpha$ for every $\alpha < \omega_1$, but it does not have the $\aleph_0$-order property of length $\omega_1$.

\end{enumerate}

\subsubsection{Counterexamples} \label{counterex-ssec}
Life would be too easy if all AECs were tame. Above we have seen that several natural mathematical classes are tame; in contrast, all the known counterexamples to tameness are pathological\footnote{In the dictionary sense that they were constructed as counter-examples.}, with the most natural being the Baldwin-Shelah example of short exact sequences.  We list the known ones below in increasing ``strength''.

\begin{enumerate}
	\item {\bf The Hart-Shelah example}\\
	The Hart-Shelah examples are a family of examples axiomatized by complete sentences in $\Ll_{\omega_1, \omega}$.
	\begin{thm}\label{hart-shelah-thm}
	For each $n < \omega$, there is an AEC $\K_n$ that is axiomatized by a complete sentence in $\Ll_{\omega_1, \omega}$ with $\LS(\K_n) = \aleph_0$ and disjoint amalgamation such that:
	\begin{enumerate}
		\item $\K_n$ is $(\aleph_0, \aleph_{n-1})$-tame (in fact, the types are first-order syntactic);
		\item $\K_n$ is categorical in $[\aleph_0, \aleph_n]$;
		\item $\K_n$ is Galois-stable in $\mu$ for $\mu \in [\aleph_0, \aleph_{n - 1}]$ ; and
		\item Each of these properties is sharp. That is:
                  \begin{enumerate} 
                    \item\label{hs-1} $\K_n$ is not $(\aleph_0, \aleph_{n})$-tame, 
                    \item $\K_n$ is not categorical in $\aleph_{n + 1}$\footnote{Note that this follows from (\ref{hs-1}) by (the proof of) the upward categoricity transfer of Grossberg and VanDieren \cite{tamenessthree}.}.
                    \item $\K_n$ is not Galois-stable in $\aleph_{n}$.
                  \end{enumerate}
	\end{enumerate}
	\end{thm}
	Each model $M \in \K_n$ begins with an index set $I$ (called the spine); the direct sum $G := \oplus_{[I]^{n+2}} \mathbb{Z}_2$; $G^* \subseteq [I]^k \times G$ with a projection $\pi:G^* \to [I]^{n+2}$ such that each stalk $G^*_{u} = \pi^{-1}\{u\}$ has a regular, transitive action of $G$ on it; and, similarly, $H^* = [I]^{n+2} \times H$ with a projection $\pi':H^* \to [I]^k$ such that each stalk has an action of $\mathbb{Z}_2$ on it.  So far, the structure described (along with the extra information required to code it) is well-behaved and totally categorical.  Added to this is a $n+3$-ary relation $Q \subseteq (G^*)^{n+2} \times H^*$ such that $Q(u_1, \dots, u_{n+2}, v)$ is intended to code
	\begin{itemize}
		\item there are exactly $n+3$ elements of $I$ that make up the projections of $u_1, \dots, u_{n+2}, v$ (so each $(n+2)$-element subset shows up exactly once in the projections); and
		\item the sum of the second coordinates evaluated at $\pi'(v)$ is equal to some fixed function of the $n+3$ elements of the projections.
	\end{itemize}
This coding allows one to ``hide the zeros'' and find nontameness at $\aleph_n$. The example shows that Theorem \ref{sh87-thm} is sharp.%\footnotei{WB: I don't really have a good understanding of what's going on here.  Do you?  Perhaps we could ask John and Alexei? SV: I never looked at this example in details... *}

	%% This example was initial constructed by Hart and Shelah \cite{hartshelah} to show that the hypothesis of categoricity in all $\aleph_n$'s in \cite{sh87a, sh87b} was tight.  Once the categoricity transfer of Grossberg and VanDieren was proven, this became a natural example\svnote{Well, from Grossberg-VanDieren it follows that it is not tame (otherwise it would be categorical in $\aleph_{n + 1}$).\\WB: I don't think that it was known that these classes had amalgamation, etc before Baldwin-Kolesnikov so it wasn't quite so simple.  I'll add a comment to that effect.} to look for a nontame AEC\footnote{Hart and Shelah characterized the categoricity spectrum of these classes, but little more was known about them prior to Baldwin-Kolesnikov.  In particular, they were not known to satisfy amalgamation, etc., so their non-tameness was not an immediate corollary of the Grossberg-VanDieren categoricity transfer.}.  Baldwin and Kolesnikov \cite{untame} analyzed this class to find an explicit description of the failure of nontameness and more (e. g., the class is model complete).
	
	It should be noted that the ideas used in constructing the Hart-Shelah examples come from constructions that characterize various cardinals. Thus, although the construction takes place in ZFC, it still involves set-theoretic ideas.  Work in preparation by Shelah and Villaveces \cite{shvi648v1} contains an extension of the Hart-Shelah example to larger cardinals, proving:

	\begin{thm}
	Assume the generalized continuum hypothesis. For each $\lambda$ and $k < \omega$, there is $\psi^\lambda_k \in \Ll_{(2^\lambda)^+, \omega}$ that is categorical in $\lambda^{+2}, \dots, \lambda^{+(k - 1)}$ but not in $\beth_{k+1}(\lambda)^+$.
	\end{thm}

As for the countable case, this example is likely not to be tame.
	
	\item {\bf The Baldwin-Shelah example}\\
	The Baldwin-Shelah example $\K$ consists of several short exact sequences, each beginning with $\mathbb{Z}$.
\[ 
\xymatrix{	& & H_j \ar[dr] & &\\
	0 \ar[r] & \mathbb{Z} \ar[r] \ar[ur] \ar[dr] & H_i \ar[r] & G \ar[r] & 0\\
	& & H_k \ar[ur] & &}
	\]
	Formally this consists of of sorts $Z$, $G$, $I$, and $H$ with a projection $\pi: H \to I$ and group operations and embeddings such that each fiber $H_i := \pi^{-1}(\{i\})$ is a group that is in the middle of a short exact sequence.
	
	The locality properties of Galois types over a model depend heavily on the group $G$ used.  The key observation is that, given $i, j \in I$, their Galois types are equal precisely when there is an isomorphism of the fibers $\pi^{-1}\{i\}$ and $\pi^{-1}\{j\}$ that commute with the rest of the short exact sequence.  Thus, Baldwin and Shelah consider an $\aleph_1$-free, not free, not Whitehead group\footnote{It is a ZFC theorem that such a group exists at $\aleph_1$. Having such a group at $\kappa$ ($\kappa$-free, not free, not Whitehead of size $\kappa$) is, in the words of Baldwin and Shelah, ``sensitive to set theory''. The known sensitivities are summarized in \cite[Section 8]{tamelc-jsl}, primarily drawing on work in \cite{magidor-shelah-free}  and \cite{ekloffmekler}.} $G^*$ of size $\aleph_1$.  With $G^*$ in hand, we can construct a counterexample to $(\aleph_0, \aleph_1)$-tameness: set $i_0$ and $i_1$ to be in a short exact sequence that ends in $G^*$ such that $\pi^{-1}\{i_0\} = G^* \oplus \mathbb{Z}$ and $\pi^{-1}\{i_1\} = H$ is not isomorphic to $G^* \oplus \mathbb{Z}$; such a group exists exactly because $H$ is not Whitehead.  Then, by the observation above, $i_0$ and $i_1$ have different types over the entire uncountable set $G^*$.  However, any countable approximation $G_0$ of $G^*$ will see that $i_0$ and $i_1$ have the same Galois type over it: the countable approximation will have that the fibers over $i_0$ and $i_1$ are both the middle of a short exact ending in $G_0$.  By the choice of $G^*$, $G_0$ is free, thus Whitehead, so these fibers are both isomorphic to $G_0 \oplus \mathbb{Z}$.  This isomorphism witnesses the equality of the Galois types of $i_0$ and $i_1$ over the countable approximation.
	
	Given a $\kappa$ version $G^*_\kappa$ of this group allows one to construct a counterexample to $(<\kappa, \kappa)$-tameness.  Indeed, all that is necessary is that $G^*_\kappa$ is `almost Whitehead:' it is not Whitehead, but every strictly smaller subgroup of it is.

	\item {\bf The Shelah-Boney-Unger example}\\
	While the Baldwin-Shelah example reveals a connection between tameness and set theory, the Shelah-Boney-Unger example shows an outright equivalence between certain tameness statements and large cardinals.  For each cardinal $\sigma^\omega = \sigma$, there is an AEC $\K_\sigma$ that consists of an index predicate $J$ with a projection $Q: H \to J$ such that each fiber $Q^{-1}\{j\}$ has a specified structure and a projection $\pi: H \to I$\footnote{Although there are two projections, they are used differently: the projection $Q$ is a technical device to code isomorphisms of structures via equality of Galois types, while the interaction of (a fiber of) $H$ and $I$ is more interesting.}.  Given some partial order $(\mathcal{D}, \vartriangleleft)$ and set of functions $\mathcal{F}$ with domain $\mathcal{D}$, filtrations $\{M_{\ell, d}: d \in \mathcal{D}\}$ of a larger models $M_{\ell, \mathcal{D}}$, all from $\K_\sigma$, are built, for $\ell = 1, 2$.  Similar to the Baldwin-Shelah example, types $p_d$ and $q_d$ are defined such that the types are equal if and only if there is a nice isomorphism between $M_{1, d}$ and $M_{2, d}$; the same is true of $p_{\mathcal{D}}$ and $q_{\mathcal{D}}$.  Thus, various properties of type locality ($p_{\mathcal{D}} = q_{\mathcal{D}}$ following from $p_d = q_d$ for all $d \in \mathcal{D}$) is once more coded by ``isomorphism locality''.
	
	In turn, the structure was built so that a nice isomorphism between $M_{1, \mathcal{D}}$ and $M_{2, \mathcal{D}}$ is equivalent to a combinatorial property $\#(\mathcal{D}, \mathcal{F})$.
	\begin{defin} \
	\begin{itemize}
		\item Given functions $f$ and $g$ with the same domain, we define $f \leq^* g$ to hold if and only if there is some $e: \ran g \to \ran f$ such that $f = e \circ g$.
		\item Given a function $f$ with a domain $D$ that is partially ordered by $\leq_D$, we define $\ran^* f = \bigcap_{d \in D} \ran \left(f \rest \{d' \in D : d \leq d'\} \right)$ to be the eventual range of $f$.
		\item $\#(\mathcal{D}, \mathcal{F})$ holds if and only if there are $f^* \in \mathcal{F}$ and a collection of nonempty finite sets $\{ u_f \subseteq \ran^* f: f^* \leq^* f\}$ such that, given any $e$ witnessing $f^* \leq^* f$, $e\rest u_f$ is a bijection from $u_f$ to $u_{f^*}$.
	\end{itemize}
	\end{defin}
	So $\#(\mathcal{D}, \mathcal{F})$ eventually puts some kind of structure on functions in $\mathcal{F}$.  Shockingly, this principle can, under the right assumption on $\mathcal{D}$ and $\mathcal{F}$, define a very complete (ultra)filter on $\mathcal{D}$: for each $f$ with $f\leq^*  f^*$, set $i_f \in u_f$ to be the unique image of $\min{u_f}$ by some $e$ witnessing $f^* \leq^* f$.  Then, for $A \subseteq \mathcal{D}$,
	$$A \in U \iff \exists d \in \mathcal{D}, f \in \mathcal{F} \left( f^{-1}\{i_f\} \cap \{ d' \in \mathcal{D} : d \vartriangleleft d' \} \subseteq A \right)$$
	
	Thus, we get that type locality in $\mathbb{\K}_\sigma$ implies the existence of filters and ultrafilters used in the definitions of large cardinals; the converse is mentioned above.
	
	This argument can be used to give the following theorems.
	
	\begin{theorem}
	Let $\kappa$ such that $\mu^\omega < \kappa$ for all $\mu < \kappa$.
	\begin{enumerate}
		\item If\footnote{The additional cardinal arithmetic here can be dropped at the cost of only concluding $(\kappa$ is weakly compact$)^L$.} $\kappa^{<\kappa} = \kappa$ and every AEC $\K$ with $\LS(\K) < \kappa$ is $(<\kappa, \kappa)$-tame, then $\kappa$ is almost weakly compact.
		\item If every AEC $\K$ with $\LS(\K) < \kappa$ is $\kappa$-local, then $\kappa$ is almost measurable.
		\item If every AEC $\K$ with $\LS(\K) < \kappa$ is $<\kappa$-tame, then $\kappa$ is almost strongly compact.
	\end{enumerate}
	\end{theorem}

        We obtain a characterization of the statement ``all AECs are tame'' in terms of large cardinals.

        \begin{cor}\label{tame-lc-cor}
          All AECs are tame if and only if there is a proper class of almost strongly compact cardinals.
        \end{cor}

        Note that Corollary \ref{tame-lc-cor} says nothing about ``well-behaved'' classes of AECs such as AECs categorical in a proper class of cardinals. In fact, Theorem \ref{tameness-from-categ-2} shows that the consistency strength of the statement ``all AECs are tame'' is much higher than that of the statement ``all unboundedly categorical AECs are tame''.
\end{enumerate}

\section{Categoricity transfer in universal classes: an overview}\label{universal-class-sec}

In this section, we sketch a proof of Theorem \ref{main-thm}, emphasizing the role of tameness in the argument:

\begin{thm}\label{main-thm-2}
  Let $\K$ be a universal class. If $\K$ is categorical in cardinals of arbitrarily high cofinality\footnote{This cofinality restriction is only used to obtain amalgamation. See the historical remarks for more.}, then $\K$ is categorical on a tail of cardinals.
\end{thm} 

The arguments in this section are primarily from Vasey \cite{ap-universal-v9}.

Note that (as pointed out in Section \ref{categ-notame-sec}), we can replace the categoricity hypotheses of Theorem \ref{main-thm-2} by categoricity in a \emph{single} ``high-enough'' cardinal of ``high-enough'' cofinality. 

We avoid technical definitions in this section, instead referring the reader to Section \ref{primer-no-tameness} or Section \ref{tame-indep-sec}. 

So let $\K$ be a universal class categorical in cardinals of arbitrarily high cofinality. To prove the categoricity transfer, we first show that $\K$ has several structural properties that hold in elementary classes. As we have seen, amalgamation is one such property.

\subsection{Step 1: Getting amalgamation}

It is not clear how to directly prove amalgamation in all cardinals, but Theorem \ref{good-frame-categ} is a deep result of Shelah which says (since good frames must have amalgamation) that it holds for models of \emph{some} suitable size\footnote{This is the only place where we use the cofinality assumptions on the categoricity cardinals.}. 

This leads to a new fundamental question:

\begin{question}\label{ap-transfer-q}
  If an AEC $\K$ has amalgamation in a cardinal $\lambda$, under what condition does it have amalgamation above $\lambda$?
\end{question}

One such condition is \emph{excellence} (briefly, excellence asserts strong uniqueness of $n$-dimensional amalgamation results).  However, it is open whether it follows from categoricity, even for classes of models of an uncountable first-order theory. Excellence also gives much more, and (for now) we are only interested in amalgamation. Another condition would be the existence of \emph{large cardinals}. For example, a strongly compact $\kappa$ with $\LS (\K) < \kappa \le \lambda $ would be enough. 

At that point, we recall a key theme in the study of tameness: when large cardinals appear in a model-theoretic result, tameness\footnote{Or really, a ``tameness-like'' property like full tameness and shortness.} can often replace them. For the purpose of an amalgamation transfer it is not clear that this suffices. For one thing, one can ask what tameness really means without amalgamation (of course, its definition makes sense, but how do we get a handle on the transitive closure of atomic equality, Definition \ref{gtp-def}.(2)). In the case of universal classes, this question has a nice answer: even without amalgamation, equality of Galois types is witnessed by an isomorphism and, in fact, tameness holds for free! This is Theorem \ref{tame-uc}. From its proof, we isolate a technical weakening of amalgamation:

\begin{defin}
  An AEC $\K$ has \emph{weak amalgamation} if whenever $\gtp (a_1 / M; N_1) = \gtp (a_2 / M; N_2)$, there exists $M_1 \in \K$ with $M \lea M_1 \lea N_1$ and $a_1 \in |M_1|$ such that $(a_1, M, M_1)$ is \emph{atomically equivalent} to $(a_2, M, N_2)$.
\end{defin}

It turns out that universal classes have weak amalgamation: we can take $M_1$ to be the closure of $|M_1| \cup \{a\}$ under the functions of $N_1$ and expand the definition of equality of Galois types.

We now rephrase Question \ref{ap-transfer-q} as follows:

\begin{question}
  Let $\K$ be an AEC which has amalgamation in a cardinal $\lambda$. Assume that $\K$ is $\lambda$-tame and has weak amalgamation. Under what condition does it have amalgamation above $\lambda$?
\end{question}

To make progress, a characterization of amalgamation will come in handy (this lemma is reminiscent of Proposition \ref{settype-bad}.(1)): 

\begin{lem}\label{galois-ext}
  Let $\K$ be an AEC with weak amalgamation. Then $\K$ has amalgamation if and only if for any $M \in \K$, any Galois type $p \in \gS (M)$, and any $N \gea M$, there exists $q \in \gS (N)$ extending $p$.
\end{lem}

The proof is easy if one assumes that atomic equivalence of Galois types is transitive. Weak amalgamation is a weakening of this property, but allows us to iterate the argument (when atomic equivalence is transitive) and obtain full amalgamation.

Now, it would be nice if we could not only extend Galois types, but also extend them canonically. This is reminiscent of first-order \emph{forking}, a basic property of which is that every type has a (unique under reasonable conditions) nonforking extension. Thus, out of the apparently very set-theoretic problem of obtaining amalgamation, forking, a model-theoretic notion, appears in the discussion. What is an appropriate generalization of forking to AECs? Shelah's answer is that the \emph{bare-bone} generalization are the \emph{good $\lambda$-frames}, see Section \ref{frame-sec}. There are several nonelementary setups where a good frame exists (see the next section). For example, Theorem \ref{good-frame-categ} tells us that a good frame exists in our setup. 

Still with the question of transferring amalgamation up in mind, one can ask whether it is possible to transfer an \emph{entire good frame} up. In particular, given a notion of forking for models of size $\lambda$, is there one for models of size above $\lambda$? This is where tameness starts playing a very important role:

\begin{thm}\label{good-frame-transfer}
  Let $\K$ be an AEC with amalgamation. Let $\s$ be a good $\lambda$-frame with underlying AEC $\K$. Then $\s$ extends to a good $(\ge \lambda)$-frame (i.e.\ all the properties hold for models in $\K_{\ge \lambda}$) \emph{if and only if} $\K$ is $\lambda$-tame.
\end{thm}

We give a sketch of the proof in Theorem \ref{good-frame-transfer-2}. Let us also note that not only do the properties of forking transfer, but also the structural properties of $\K$. Thus $\K_{\ge \lambda}$ has no maximal models (roughly, this is obtained using the extension property and the fact that nonforking extensions are nonalgebraic).

Even better, it turns out that not too much amalgamation is needed for the proof of the frame transfer to go through: weak amalgamation is enough! Moreover types can be extended by simply taking their nonforking extension. Thus we obtain:

\begin{thm}\label{frame-transfer}
  Let $\K$ be an AEC with weak amalgamation. If there is a good $\lambda$-frame $\s$ with underlying AEC $\K$ and $\K$ is $\lambda$-tame, then $\s$ extends to a good $(\ge \lambda)$-frame. In particular $\K_{\ge \lambda}$ has amalgamation.
\end{thm} 

\begin{cor}
  Let $\K$ be a universal class categorical in cardinals of arbitrarily high cofinality. Then there exists $\lambda$ such that $\K_{\ge \lambda}$ has amalgamation.
\end{cor} 
\begin{proof}
  By Theorem \ref{good-frame-categ}, there is a cardinal $\lambda$ such that $\K$ has a good $\lambda$-frame with underlying class $\K_{\lambda}$. By Theorem \ref{tame-uc}, $\K$ is $\lambda$-tame and (it is easy to see), $\K$ has weak amalgamation. Now apply Theorem \ref{frame-transfer}.
\end{proof}
\begin{remark}
  Theorem \ref{frame-transfer} will be used even in the next steps, see the proof of Theorem \ref{unidim-from-categ}.
\end{remark}

\subsection{Step 2: Global independence and orthogonality calculus}\label{orthog-univ}

From the results so far, we see that we can replace $\K$ by $\K_{\ge \lambda}$ if necessary to assume without loss of generality that $\K$ is a universal class\footnote{There is a small wrinkle here: if $\K$ is a universal class, $\K_{\ge \lambda}$ is not necessarily a universal class. We ignore this detail here since $\K_{\ge \lambda}$ will have enough of the properties of a universal class to carry the argument through.} categorical in a proper class of cardinals that has amalgamation. Other structural properties such as joint embedding and no maximal models follow readily. In fact, we have just pointed out that we can assume there is a good $(\ge \LS (\K))$-frame with underlying class $\K$. In particular, $\K$ is Galois-stable in all cardinals and has a superstable-like forking notion for types of length one.

What is the next step to get a categoricity transfer? The classical idea is to show that all big-enough models are Galois-saturated (note that by the above we have stability everywhere, so the model in the categoricity cardinal is Galois-saturated). Take $M$ a model in a categoricity cardinal $\lambda$ and $p$ a nonalgebraic type over $M$. Assume that there exists $N \gta M$ of size $\lambda$ such that $p$ is omitted in $N$. If we can iterate this property $\lambda^+$-many times, we obtain a non $\lambda^+$-Galois-saturated models. If $\K$ was categorical in $\lambda^+$, this gives a contradiction. More generally, if we can iterate longer to find $N \gta M$ of size $\mu > \lambda$ such that $N$ omits $p$ and $\K$ is categorical in $\mu$, we also get a contradiction. This is reminiscent of a Vaughtian pair argument and more generally of Shelah's theory of unidimensionality. Roughly speaking, a class is \emph{unidimensional} if it has essentially only one Galois type. Then a model cannot have arbitrarily large extensions omitting the type. Conversely if the class is not unidimensional, then it has two ``orthogonal'' types and a model would be able to grow by adding more realizations of one type without realizing the other.

So we want to give a sense in which our class $\K$ is unidimensional. If $\K$ is categorical in a successor, this can be done much more easily than for the limit case using Vaughtian pairs. In fact a classical result of Grossberg and VanDieren for tame AECs says:

\begin{thm}\label{gv-upward}
Suppose $\K$ has amalgamation and no maximal models.  If $\K$ is a $\lambda$-tame AEC categorical in $\lambda$ and $\lambda^+$, then $\K$ is categorical in all $\mu \ge \lambda$.
\end{thm}

To study general unidimensionality, we will use a notion of orthogonality. As for forking, we focus on developing a theory of orthogonality for types of length one over models of a single size. 

We already have a good $(\ge \LS (\K))$-frame available but for our purpose this is not enough. We will also use a notion of primeness:

\begin{defin}\label{prime-def}
  We say an AEC $\K$ \emph{has primes} if whenever $M \lea N$ are in $\K$ and $a \in |N| \backslash |M|$, there is a prime model $M' \lea N$ over $|M| \cup \{a\}$. This means that if $\gtp (b / M; N') = \gtp (a / M; N)$, then\footnote{Why the formulation using Galois types? We have to make sure that the types of $M a$ in $N$ and $N'$ are the same.} there exists $f: M' \xrightarrow[M]{} N'$ so that $f (b) = a$. We call $(a, M, M')$ a \emph{prime triple}.
\end{defin}

Note that this makes sense even if the AEC does not have amalgamation. Some computations give us that:

\begin{prop}
  If $\K$ is a universal class, then $\K$ has primes.
\end{prop}

\begin{defin}\label{perp-def}
  Let $\K$ be an AEC with a good $\lambda$-frame. Assume that $\K$ has primes (at least for models of size $\lambda$). Let $M \in \K_\lambda$ and let $p, q \in \gS (M)$. We say that $p$ and $q$ are \emph{weakly orthogonal} if there exists a prime triple $(a, M, M')$ such that $\gtp (a / M; M') = q$ and $p$ has a unique extension to $\gS (M')$. We say that $p$ and $q$ are \emph{orthogonal} if for any $N \gea M$, the nonforking extensions $p'$, $q'$ to $N$ of $p$ and $q$ respectively are weakly orthogonal.
\end{defin}

Orthogonality and weak orthogonality coincide assuming categoricity:

\begin{thm}
  Let $\K$ be an AEC which has primes and a good $\lambda$-frame. Assume that $\K$ is categorical in $\lambda$. Then weak orthogonality and orthogonality coincide.
\end{thm}

We have arrived to a definition of unidimensionality (we say that a good $\lambda$-frame is categorical when the underlying class is categorical in $\lambda$):

\begin{defin}
  Let $\K$ be an AEC which has primes and a categorical good $\lambda$-frame. $\K_\lambda$ is \emph{unidimensional} if there does \emph{not} exist $M \in \K$, and types $p, q \in \gS (M)$ such that $p$ and $q$ are orthogonal.
\end{defin}

\begin{thm}\label{unidim-categ}
Let $\K$ be an AEC which has primes and a categorical good $\lambda$-frame. If $\K$ is unidimensional, then $\K$ is categorical in $\lambda^+$.
\end{thm}

Using the result of Grossberg and VanDieren, if in addition $\K$ is $\lambda$-tame, $\K$ will be categorical in every cardinal above $\lambda$. Therefore it is enough to prove unidimensionality. While step 2 was only happening locally in $\lambda$ and did not use tameness, tameness will again have a crucial use in the next step.

\subsection{Step 3: Proving unidimensionality}

Let us make a slight diversion from unidimensionality. Recall that we work in a universal class $\K$ categorical in a proper class of cardinals with a lot of structural properties (amalgamation and existence of good frames, even global). We want to show that all big-enough models are Galois-saturated. Let $M$ be a big model and assume it is not Galois-saturated, say it omits $p \in \gS (M_0)$, $M_0 \lea M$. Consider the class $\K_{\neg p}$ of all models $N \gea M_0$ that omit $p$. After adding constant symbols for $M_0$ and closing under isomorphisms, this is an AEC. We would like to show that it has arbitrarily large models, for then this means that there exists a categoricity cardinal $\lambda > \|M_0\|$ and $N \in \K$ of size $\lambda$ omitting $p$. This is a contradiction since we know that the model in the categoricity cardinal is Galois-saturated.

What are methods to show that a class has arbitrarily large models? A powerful one is again based on good frames: by definition, a good $\lambda$-frame has no maximal models in $\lambda$. If we can expand it to a good $(\ge \lambda)$-frame, then its underlying class $\K_{\ge \lambda}$ has no maximal models and hence models of arbitrarily large size. Recall that Theorem \ref{frame-transfer} gave mild conditions (tameness and weak amalgamation) under which a good frame can be transferred. Tameness was the key property there.

Moreover we know already that $\K$ itself has a good $\LS (\K)$-frame, amalgamation, and is $\LS (\K)$-tame. But it is not so clear that these properties transfer to $\K_{\neg p}$. Consider for example amalgamation: let $M_0 \lea M_\ell$, $\ell = 1,2$, and assume that $p \in \gS (M_0)$ is omitted in both $M_1$ and $M_2$. Even if in $\K$ there exists amalgams of $M_1$ and $M_2$ over $M_0$, it is not clear that any such amalgams will omit $p$. Similarly, even if $q_1, q_2 \in \gS (M)$ are Galois types in $\K_{\neg p}$ and they are equal in $\K$, there is no guarantee that they will be equal in $\K_{\neg p}$ (the amalgam witnessing it may not be a member of $\K_{\neg p}$). So it is not clear that $\K_{\neg p}$ is tame.

However we are interested in universal classes, so consider the last property if $\K$ is a universal class. Say $q_1 = \gtp (a_1 / M; N_1)$, $q_2 = \gtp (a_2 / M; N_2)$. If $q_1$ is equal to $q_2$ in $\K$, then since $\K$ has primes there exists $M_1 \lea N_1$ containing $a_1$ and $M$ and $f: M_1 \xrightarrow[M]{} N_2$ so that $f (a_1) = a_2$. If $N_1$ omits $p$, then $M_1$ also omits $p$ and so $q_1$ and $q_2$ are equal also in $\K_{\neg p}$. By the same argument, $\K_{\neg p}$ also has primes (in fact it is itself universal). Thus it has weak amalgamation. Similarly, since $\K$ is tame, $\K_{\neg p}$ is also tame\footnote{A technical remark: if we only knew that $\K$ was \emph{weakly} tame (i.e.\ tame for types over Galois-saturated models), we would not be able to conclude that $\K_{\neg p}$ was weakly tame: models that omit $p$ of size larger than $|\text{dom} (p)|$ are not Galois-saturated. Thus while many arguments in the study of tame AECs can be adapted to the weakly tame context, this one cannot.}.

The last problem to solve is therefore whether a good $\lambda$-frame in $\K$ is also a good $\lambda$-frame in $\K_{\neg p}$. This is where we use orthogonality calculus and unidimensionality. However the class we consider is slightly different than $\K_{\neg p}$: for $p \in \gS (M_0)$ nonalgebraic, we let $\K_{\neg^\ast p}$ be the class of $M$ such that $M_0 \lea M$ and $p$ has a \emph{unique} extension to $M$ (we add constant symbols for $M_0$ to make $\K_{\neg^\ast p}$ closed under isomorphisms). Note that $\K_{\neg^\ast p} \subseteq \K_{\neg p}$, as the unique extension must be the nonforking extension, which is nonalgebraic if $p$ is. Using orthogonality calculus, we can show:

\begin{thm}\label{not-unidim}
  Let $\K$ be an AEC which has primes and a categorical good $\lambda$-frame $\s$ for types at most $\lambda$. If $\K$ is \emph{not} unidimensional, then there exists $M \in \K_\lambda$ and a nonalgebraic $p \in \gS (M)$ such that $\s$ restricted to $\K_{\neg^\ast p}$ is still a good $\lambda$-frame.
\end{thm}

We are ready to conclude: 

\begin{thm}\label{unidim-from-categ}
  Suppose that $\K$ is an AEC which has primes and a categorical good $\lambda$-frame for types at most $\lambda$. If $\K$ is categorical in some $\mu > \lambda$ and is $\lambda$-tame, then $\K_\lambda$ is unidimensional, and therefore $\K$ is categorical in all $\mu' > \lambda$.
\end{thm}
\begin{proof}
  The last ``therefore'' follows from combining Theorem \ref{unidim-categ} and the Grossberg-VanDieren transfer (using tameness heavily) Theorem \ref{gv-upward}. To show that $\K_\lambda$ is unidimensional, suppose not. By Theorem \ref{not-unidim}, there exists $M \in \K_\lambda$ and a nonalgebraic $p \in \gS (M)$ such that there is a good $\lambda$-frame on $\K_{\neg^\ast p}$. By the argument above, $\K_{\neg^\ast p}$ is $\lambda$-tame and has weak amalgamation. But is $\K_{\neg^\ast p}$ an AEC? Yes! The only problematic part is if $\seq{M_i : i < \delta}$ is increasing in $\K_{\neg^\ast p}$, and we want to show that $M_\delta := \bigcup_{i < \delta} M_i$ is in $\K_{\neg^\ast p}$. Let $q_1, q_2 \in \gS (M_\delta)$ be extensions of $p$, we want to see that $q_1 = q_2$. By tameness, the good $\lambda$-frame of $\K$ transfers to a good $(\ge \lambda)$-frame. So we can fix $i < \delta$ such that $q_1, q_2$ do not fork over $M_i$. By definition of $\K_{\neg^\ast p}$, $q_1 \rest M_i = q_2 \rest M_i$. By uniqueness of nonforking extension, $q_1 = q_2$.

  By Theorem \ref{frame-transfer} (using that $\K_{\neg^\ast p}$ is tame), $\K_{\neg^\ast p}$ has a good $(\ge \lambda)$-frame. In particular, it has arbitrarily large models.  Thus, $\K$ has non-Galois-saturated models in every $\mu > \lambda$, hence cannot be categorical in any $\mu > \lambda$.
\end{proof}

We wrap up:

\begin{proof}[Proof of Theorem \ref{main-thm-2}]
  Let $\K$ be a universal class categorical in cardinals of arbitrarily high cofinality.
  \begin{enumerate}
    \item Just because it is a universal class, $\K$ has primes and is $\LS (\K)$-tame (recall Example \ref{examples-subsec}.(\ref{universalclasses})).
    \item By Theorem \ref{good-frame-categ}, there exists a good $\lambda$-frame on $\K$.
    \item By the upward frame transfer (Theorem \ref{frame-transfer}), $\K_{\ge \lambda}$ has amalgamation and in fact a good $(\ge \lambda)$-frame. This step uses $\lambda$-tameness.
    \item By orthogonality calculus, if $\K_{\lambda}$ is not unidimensional then there exists a type $p$ such that $\K_{\neg^\ast p}$ has a good $\lambda$-frame. 
    \item Since $\K$ has primes, $\K_{\neg^\ast p}$ is also $\lambda$-tame and has weak amalgamation, so by the upward frame transfer again (using tameness) it must have arbitrarily large models. So arbitrarily large models omit $p$, hence $\K$ has no Galois-saturated models of size above $\lambda$, so cannot be categorical above $\lambda$ (by stability, $\K$ has a Galois-saturated model in every categoricity cardinal). This is a contradiction, therefore $\K_{\lambda}$ is unidimensional.
    \item By Theorem \ref{unidim-categ}, $\K$ is categorical in $\lambda^+$.
    \item By the upward transfer of Grossberg and VanDieren (Theorem \ref{gv-upward}), $\K$ is categorical in all $\mu \ge \lambda$. This again uses tameness in a key way.
  \end{enumerate}
\end{proof}
\begin{remark}
  The proof can be generalized to abstract elementary classes which are tame and have primes. See Theorem \ref{event-categ-primes}.
\end{remark}

\section{Independence, stability, and categoricity in tame AECs} \label{tame-indep-sec}

We have seen that good frames are a crucial tool in the proof of Shelah's eventual categoricity conjecture in universal classes. In this section, we give the precise definition of good frames in a more general axiomatic independence framework. We survey when good frames and more global independence notions are known to exist (i.e.\ the best known answers to Question \ref{indep-quest}).

We look at what can be said in both strictly stable and superstable AECs. Along the way we look at stability transfers, and the equivalence of various definitions of superstability in tame AECs.

Finally, we survey the theory of categorical tame AECs and give the best known approximations to Shelah's categoricity conjecture in this framework.

\subsection{Abstract independence relations} \label{air-subsec}

To allow us to state precise results, we first fix some terminology. The terms used should be familiar to readers with experience in working with forking, either in the elementary or nonelementary context. One potentially unfamiliar notation: we sometimes refer to the pair $\is =(\K, \dnf)$ as an independence relation. This is particularly useful to deal with multiple classes as we can differentiate between the behavior of a possible forking relation on the class $\K$ compared to its behavior on the class $\Ksatp{\lambda}$ of $\lambda$-Galois-saturated models of $\K$.

\begin{defin}\label{indep-def}
  An \emph{independence relation} is a pair $\is = (\K, \nf)$, where:

    \begin{enumerate}
      \item $\K$ is an AEC\footnote{We may look at independence relations where $\K$ is not an AEC (e.g.\ it could be a class of Galois-saturated models in a strictly stable AEC).} with amalgamation (we say that $\is$ is \emph{on $\K$} and write $\K_{\is} = \K$).
      \item $\nf$ is a relation on quadruples of the form $(M, A, B, N)$, where $M \lea N$ are all in $\K$ and $A, B \subseteq |N|$. We write $\nf(M, A, B, N)$ or $\nfs{M}{A}{B}{N}$ instead of $(M, A, B, N) \in \nf$.
      \item The following properties hold:

        \begin{enumerate}
        \item \underline{Invariance}: If $f: N \cong N'$ and $\nfs{M}{A}{B}{N}$, then $\nfs{f[M]}{f[A]}{f[B]}{N'}$.
        \item \underline{Monotonicity}: Assume $\nfs{M}{A}{B}{N}$. Then:
          \begin{enumerate}
          \item Ambient monotonicity: If $N' \gea N$, then $\nfs{M}{A}{B}{N'}$. If $M \lea N_0 \lea N$ and $A \cup B \subseteq |N_0|$, then $\nfs{M}{A}{B}{N_0}$.
          \item Left and right monotonicity: If $A_0 \subseteq A$, $B_0 \subseteq B$, then $\nfs{M}{A_0}{B_0}{N}$.
          \item Base monotonicity: If $\nfs{M}{A}{B}{N}$ and $M \lea M' \lea N$, $|M'| \subseteq B \cup |M|$, then $\nfs{M'}{A}{B}{N}$.
          \end{enumerate}
        \item \underline{Left and right normality}: If $\nfs{M}{A}{B}{N}$, then $\nfs{M}{AM}{BM}{N}$.
        \end{enumerate}
    \end{enumerate}
    When there is only one relation to consider, we sometimes write ``$\nf$ is an independence relation on $\K$'' to mean ``$(\K, \nf)$ is an independence relation''.
\end{defin}
\begin{defin}
  Let $\is = (\K, \nf)$ be an independence relation. Let $M \lea N$, $B \subseteq |N|$, and $p \in \gS^{<\infty} (B; N)$ be given. We say that \emph{$p$ does not $\is$-fork over $M$} if whenever $p = \gtp (\ba / B; N)$, we have that $\nfs{M}{\ran (\ba)}{B}{N}$. When $\is$ is clear from context, we omit it. 
\end{defin}
\begin{remark}
  By the ambient monotonicity and invariance properties, this is well-defined (i.e.\ the choice of $\ba$ and $N$ does not matter).
\end{remark}

An independence relation can satisfy several natural properties:

\begin{defin}[Properties of independence relations]\label{indep-props-def}
  Let $\is = (\K, \nf)$ be an independence relation.

  \begin{enumerate}
    \item $\is$ has \emph{disjointness} if $\nfs{M}{A}{B}{N}$ implies $A \cap B \subseteq |M|$.
    \item $\is$ has \emph{symmetry} if $\nfs{M}{A}{B}{N}$ implies $\nfs{M}{B}{A}{N}$.

    \item $\is$ has \emph{existence} if $\nfs{M}{A}{M}{N}$ for any $A \subseteq |N|$.
    \item $\is$ has \emph{uniqueness} if whenever $M_0 \lea M \lea N_\ell$, $\ell = 1,2$, $|M_0| \subseteq B  \subseteq |M|$, $q_\ell \in \gS^{<\infty} (B; N_\ell)$, $q_1 \rest M_0 = q_2 \rest M_0$, and $q_\ell$ does not fork over $M_0$, then $q_1 = q_2$.
    \item $\is$ has \emph{extension} if whenever $p \in \gS^{<\infty} (M B; N)$ does not fork over $M$ and $B \subseteq C \subseteq |N|$, there exists $N' \gea N$ and $q \in \gS^{<\infty} (M C; N')$ extending $p$ such that $q$ does not fork over $M$.
    \item $\is$ has \emph{transitivity} if whenever $M_0 \lea M_1 \lea N$, $\nfs{M_0}{A}{M_1}{N}$ and $\nfs{M_1}{A}{B}{N}$ imply $\nfs{M_0}{A}{B}{N}$.
    \item $\is$ has the \emph{$<\kappa$-witness property} if whenever $M \lea N$, $A,B \subseteq |N|$, and $\nfs{M}{A_0}{B_0}{N}$ for all $A_0 \subseteq A$, $B_0 \subseteq B$ of size strictly less than $\kappa$, then $\nfs{M}{A}{B}{N}$. The \emph{$\lambda$-witness property} is the $(<\lambda^+)$-witness property.
  \end{enumerate}

\end{defin}

The following cardinals are also important objects of study:

\begin{defin}[Locality cardinals]\label{loc-card-def}
  Let $\is = (\K, \nf)$ be an independence relation and let $\alpha$ be a cardinal.

  \begin{enumerate}
    \item Let $\slc{\alpha} (\is)$ be the minimal cardinal $\mu \ge \alpha^+ + \LS (\K)^+$ such that for any $M \lea N$ in $\K$, any $A \subseteq |N|$ with $|A| \le \alpha$, there exists $M_0 \lea M$ in $\K_{<\mu}$ with $\nfs{M_0}{A}{M}{N}$. When $\mu$ does not exist, we set $\slc{\alpha_0} (\is) = \infty$.
      \item Let $\clc{\alpha} (\is)$ be the minimal cardinal $\mu \ge \alpha^+ + \aleph_0$ such that for any regular $\delta \ge \mu$, any increasing continuous chain $\seq{M_i : i \le \delta}$ in $\K$, any $N \gea M_\delta$, and any $A \subseteq |N|$ of size at most $\alpha$, there exists $i < \delta$ such that $\nfs{M_i}{A}{M_\delta}{N}$. When $\mu$ does not exist, we set $\clc{\alpha_0} (\is) = \infty$.
  \end{enumerate}

  We also let $\slc{<\alpha} (\is) := \sup_{\alpha_0 < \alpha} \slc{\alpha_0} (\is)$. Similarly define $\clc{<\alpha} (\is)$.   When clear, we may write $\kappa_\alpha(\nf)$, etc., instead of $\kappa_\alpha(\is)$.

\end{defin}
\begin{defin}
  Let us say that an independence relation $\is$ has \emph{local character} if $\slc{\alpha} (\is) < \infty$ for all cardinals $\alpha$.
\end{defin}

Compared to the elementary framework, we differentiate between two local character cardinals, $\kappa$ and $\bkappa$. The reason is that we do not in general (but see Theorem \ref{univ-classes-goodness}) know how to make sense of when a type does not fork over an arbitrary set (as opposed to a model). Thus we cannot (for example) define superstability by requiring that every type does not fork over a finite set: looking at unions of chains is a replacement.

We make precise when an independence relation is ``like forking in a first-order stable theory'':

\begin{defin}\label{stable-indep}
  We say that  $\is$ is a \emph{stable independence relation} if it is an independence relation satisfying uniqueness, extension, and local character.
\end{defin}

We could also define the meaning of a \emph{superstable independence relation}, but here several nuances arise so to be consistent with previous terminology we will call it a \emph{good} independence relation, see Definition \ref{good-indep-def}.

As defined above, independence relations are \emph{global objects}: they define an independence notion ``$p$ does not fork over $M$'' for $M$ of any size and $p$ of any length. This is a strong requirement. In fact, the following refinement of Question \ref{indep-quest} is still open:

\begin{question}\label{indep-question}
  Let $\K$ be a fully tame and short AEC with amalgamation. Assume that $\K$ is categorical in a proper class of cardinals. Does there exists a $\lambda$ and a stable independence relation $\is$ on $\K_{\ge \lambda}$?
  \end{question}

It is known that one can construct such an $\is$ with local character and uniqueness, but proving that it satisfies extension seems hard in the absence of compactness. Note in passing that $\is$ as above must be unique:

\begin{thm}[Canonicity of stable independence] \label{indep-canon-thm}
  If $\is$ and $\is'$ are stable independence relations on $\K$, then $\is = \is'$.
\end{thm}

As seen in Example \ref{examples-subsec}.(\ref{stabletame}), we know that uniqueness and local character are enough to conclude some tameness and there are several relationships between the properties. We give one example:

\begin{prop}\label{stable-indep-props}
  Assume that $\dnf$ is a stable independence relation on $\K$.

  \begin{enumerate}
    \item $\dnf$ has symmetry, existence, and transitivity.
    \item\label{tameshort-witness} If $\K$ is fully $<\kappa$-tame and -type short, then $\dnf$ has the $<\kappa$-witness property.
    \item For every $\alpha$, $\clc{\alpha} (\dnf) \le \slc{\alpha} (\dnf)$.
    \item $\dnf$ has disjointness over sufficiently Galois-saturated models: if $M$ is $\LS (\K)^+$-Galois-saturated and $\nfs{M}{A}{B}{N}$, then $A \cap B \subseteq |M|$.
  \end{enumerate}
\end{prop}
\begin{proof}[Proof sketch for (\ref{tameshort-witness})] \
    %\item Existence follows from local character and transitivity follows easily from extension and uniqueness. Symmetry can be obtained by showing that a failure of symmetry implies the order property, and hence $\K$ is unstable in  all sufficiently high cardinals, contradicting Proposition \ref{indep-struct}.
     By symmetry and extension it is enough to show that for a given $A$, $\nfs{M_0}{A_0}{M}{N}$ for all $A_0 \subseteq A$ of size less than $\kappa$ implies $\nfs{M_0}{A}{M}{N}$. By extension, pick $N' \gea N$ and $A' \subseteq |N'|$ so that $\nfs{M_0}{A'}{M}{N'}$ and $\gtp (\ba' / M_0; N') = \gtp (\ba / M_0; N)$ (where $\ba$, $\ba'$ are enumerations of $A$ and $A'$ respectively). By the uniqueness property, $\gtp (\ba' \rest I / M; N') = \gtp (\ba \rest I / M_0; N)$ for all $I \subseteq \dom{\ba}$ of size less than $\kappa$. Now use by shortness this implies $\gtp (\ba / M; N) = \gtp (\ba' / M; N')$, hence by invariance $\nfs{M_0}{A}{M}{N}$.
%    \item Use the extension property of splitting\svnote{To complete.}.
\end{proof}

In what follows, we consider several approximations to Question \ref{indep-question} in the stable and superstable contexts. We also examine consequences on categorical AECs. It will be convenient to \emph{localize} Definition \ref{indep-def} so that:

\begin{enumerate}
  \item The relation $\nf$ is only defined on types of certain lengths (that is, the size of the left hand side is restricted).
  \item The relation $\nf$ is only defined on types over domains of certain sizes (that is, the size of the right hand side and base is restricted).
\end{enumerate}

More precisely:

\begin{notation}
  Let $\F = [\lambda, \theta)$ be an interval of cardinals. We say that $\is = (\K, \nf)$ is a \emph{$(<\alpha, \F)$-independence relation} if it satisfies Definition \ref{indep-def} localized to types of length less than $\alpha$ and models in $\K_{\F}$ (so only amalgamation in $\F$ is required). We always require that $\theta \ge \alpha$ and $\lambda \ge \LS (\K)$.

$(\le \alpha, \F)$ means $(<\alpha^+, \F)$, and if $\F = [\lambda, \lambda^+)$, then we say that $\is$ is a $(\le \alpha, \lambda)$-independence relation. Similar variations are defined as expected, e.g.\ $(\le \alpha, \ge \lambda)$ means $(\le \alpha, [\lambda, \infty))$. 

We often say that $\is$ is a $(<\alpha)$-ary independence relation on $\K_\F$ rather than a $(<\alpha, \F)$-independence relation. We write $\alpha$-ary rather than $(\le \alpha)$-ary.
\end{notation}

The properties in Definition \ref{indep-props-def} can be adapted to such localized independence relations. For example, we say that $\is$ has \emph{symmetry} if $\nfs{M}{A}{B}{N}$ implies that for every $B_0 \subseteq B$ of size less than $\alpha$, $\nfs{M}{B_0}{A}{N}$.

Using this terminology, we can give the definition of a \emph{good $\lambda$-frame} (see Section \ref{frame-sec}), and more generally of a good $\F$-frame for $\F$ an interval of cardinals\footnote{Note that the definition here is different (but equivalent to) Shelah's notion of a \emph{type-full} good $\lambda$-frame, see the historical remarks for more.}:

\begin{defin}\label{good-frame-def}
  Let $\F = [\lambda, \theta)$ be an interval of cardinals. A \emph{good $\F$-frame} is a $1$-ary independence relation $\is$ on $\K_{\F}$ such that:
    \begin{enumerate}
      \item $\is$ satisfies disjointness, symmetry, existence, uniqueness, extension, transitivity, and $\clc{1} (\is) = \aleph_0$.
      \item $\K_{\F}$ has amalgamation in $\F$, joint embedding in $\F$, no maximal models in $\F$, and is Galois-stable in every $\mu \in \F$. Also of course $\K_{\F} \neq \emptyset$. 
    \end{enumerate}
    
    When $\F = [\lambda, \lambda^+)$, we talk of a good $\lambda$-frame, and when $\F = [\lambda, \infty)$, we talk of a good $(\ge \lambda)$-frame. As is customary, we may use the letter $\s$ rather than $\is$ to denote a good frame.
\end{defin}

\subsection{Stability}\label{stab-indep-sec} We compare results for stability in tame classes with those in general classes, summarized in Section \ref{genstab-subsec}.  At a basic level, tameness strongly connects types over domains of different cardinalities.  While a general AEC might be Galois-stable in $\lambda$ but not in $\lambda^+$ (see the Hart-Shelah example in Section \ref{counterex-ssec}), this cannot happen in tame classes:

\begin{theorem}\label{stab-transfer}
Suppose that $\K$ is an AEC with amalgamation which is $\lambda$-tame\footnote{For the first part, weak tameness suffices.} and Galois-stable in $\lambda$. Then:

\begin{enumerate}
\item $\K$ is Galois-stable in $\lambda^+$.
\item $\K$ is Galois-stable in every $\mu > \lambda$ such that $\mu = \mu^{\lambda}$.
\end{enumerate}
\end{theorem}

There is also a partial stability spectrum theorem for tame AECs:

\begin{theorem}\label{stab-spectrum}
Let $\K$ be an AEC with amalgamation that is $\LS (\K)$-tame. The following are equivalent:
\begin{enumerate}
	\item $\K$ is Galois-stable in some cardinal $\lambda \ge \LS (\K)$.
	\item $\K$ does not have the order property (see Definition \ref{op-def}).
	\item There are $\mu \leq \lambda_0 < H_1$ such that $\K$ is Galois-stable in any $\lambda = \lambda^{<\mu} + \lambda_0$.
\end{enumerate}
\end{theorem}

The proof makes heavy use of the Galois Morleyization (Theorem \ref{morleyization-thm}) to connect ``stability theory inside a model'' (results about formal, syntactic types within a particular model) to Galois types in an AEC. This allows the translation of classical proofs connecting the order property and stability.

This achieves two important generalizations from the elementary framework.  First, it unites the characterizations of stability in terms of counting types and no order property from first-order, a connection still lacking in general AECs.  Second, it gives one direction of the stability spectrum theorem by showing that, given stability in any one place, there are many stability cardinals, and some of the stability cardinals are given by satisfying some cardinal arithmetic above the first stability cardinal. Still lacking from this is a converse saying that the stability cardinals are exactly characterized by some cardinal arithmetic.

Another important application of the Galois Morleyization in stable tame AECs is that \emph{averages} of suitable sequences can be analyzed. Roughly speaking, we can work inside the Galois Morleyization of a monster model and define the \emph{$\chi$-average over $A$ of a sequence $\BI$} to be the set of formulas $\phi$ over $A$ so that strictly less than $\chi$-many elements of $\BI$ satisfy $\phi$. If $\chi$ is big-enough and under reasonable conditions on $\BI$ (i.e.\ it is a Morley sequence with respect to nonsplitting), we can show that the average is complete and (if $\BI$ is long-enough), realized by an element of $\BI$. Unfortunately, a detailed study is beyond the scope of this paper, see the historical remarks for references.

Turning to independence relations in stable AECs, there are two main candidates.  The first is the familiar notion of splitting (see Definition \ref{splitting-def}).  Tameness simplifies the discussion of splitting by getting rid of the cardinal parameter: it is impossible for a type to $\lambda^+$-split over $M$ and also not $\lambda$-split over $M$ in a $(\lambda, \lambda^+)$-tame AEC, as the witness to $\lambda^+$-splitting could be brought down to size $\lambda$.  This observation allows for a stronger uniqueness result in non-splitting.  Rather than just having unique extensions in the same cardinality as in Theorem \ref{gen-unique-dns-fact}, we get a cardinal-free uniqueness result.

\begin{thm} \label{ns-uniq-tame-fact}
Suppose $\K$ is a $\LS (\K)$-tame AEC with amalgamation and that $M_0 \lea M_1 \lea M_2$ are in $\K_{\ge \LS (\K)}$ with $M_{1}$ universal over $M_0$.  If $p,q \in \gS(M_2)$ do not split over $M_0$ and $p \rest M_1 = q \rest M_1$, then $p = q$.
\end{thm}
\begin{proof}[Proof sketch]
If $p \neq q$, then there is a small $M^- \lea M_2$ with $p \rest M^- \neq q \rest M^-$; Without loss of generality pick $M^-$ to contain $M_0$.  By universality, we can find $f: M^- \xrightarrow[M_0]{} M_1$.  By the nonsplitting, 
$$p \rest f(M^-) =  f(p \rest M^-) \neq f(q \rest M^-) = q \rest f(M^-)$$
Since $f(M^-) \lea M_1$, this contradicts the assumption they have equal restrictions. 
\end{proof}

Attempting to use splitting as an independence relation for $\K$ runs into the issue that several theorems require that the extension be \emph{universal} (such as the above theorem).  This can be mitigated by moving to the class of saturated enough models and looking at a localized version of splitting.

\begin{defin}\label{mu-forking-def} \
\begin{enumerate} 
	\item   Let $\K$ be an AEC with amalgamation. For $\mu > \LS (\K)$, let $\Ksatp{\mu}$ denote the class of $\mu$-Galois-saturated models in $\K_{\ge \mu}$.
	\item   Let $\K$ be an AEC with amalgamation and let $\mu \ge \LS (\K)$ be such that $\K$ is Galois-stable in $\mu$. For $M_0 \lea M$ both in $\Ksatp{\mu^+}$ and $p \in \gS (M)$, we say that \emph{$p$ does not $\mu$-fork over $M_0$} if there exists $M_0' \lea M_0$ with $M_0' \in \K_\mu$ such that $p$ does not $\mu$-split over $M_0'$ (see Definition \ref{splitting-def}). 
\end{enumerate}
\end{defin}

Note that, by the $\mu^+$-saturation of $M_0$, we have guaranteed that $M_0$ is a universal extension of $M_0'$.  This gives us the following result.

\begin{thm}\label{stable-forking-def}
  Let $\K$ be an AEC with amalgamation and let $\mu \ge \LS (\K)$ be such that $\K$ is Galois-stable in $\mu$ and $\K$ is $\mu$-tame. Let $\nf^\mu$ be the $\mu$-nonforking relation restricted to the class $\Ksatp{\mu^+}$.  Then
  \begin{enumerate}
  	\item $\nf^{\mu}$ is a 1-ary independence relation that further satisfies disjointness, existence, uniqueness, and transitivity {\emph when} all models are restricted to $\Ksatp{\mu^+}$ (in the precise language of Section \ref{indep-def}, this says that $(\Ksatp{\mu^+}, \nf^{\mu})$ is an independence relation with these properties).
	\item $\nf^{\mu}$ has set local character in $\Ksatp{\mu^+}$: Given $p \in \gS(M)$, there is $M_0 \in \Ksatp{\mu^+}$ such that $M_0 \lea M$ and $p$ does not $\mu$-fork over $M_0$.
	\item $\nf^{\mu}$ has a local extension property: If $M_0 \lea M$ are both Galois-saturated and $\|M_0\| = \|M\| \ge \mu^+$ and $p \in \gS (M_0)$, then there exists $q \in \gS (M)$ extending $p$ and not $\mu$-forking over $M_0$.
  \end{enumerate}

\end{thm}
\begin{proof}[Proof sketch]
  Tameness ensures that $\mu$-splitting and $\lambda$-splitting coincide when $\lambda \ge \mu$. The local extension property uses the extension property of splitting (see Theorem \ref{gen-unique-dns-fact}). Local character and uniqueness are also translations of the corresponding properties of splitting. Disjointness is a consequence of the moreover part in the extension property of splitting. Finally, transitivity is obtained by combining the extension and uniqueness properties of splitting.
\end{proof}

The second candidate for an independence relation, drawing from stable first-order theories, is a notion of coheir, which we call $<\kappa$-satisfiability.

\begin{defin}\label{kapcoheir-def}
  Let $M \lea N$ and $p \in \gS^{<\infty} (N)$.

  \begin{enumerate}
    \item We say that $p$ is a \emph{$<\kappa$-satisfiable over $M$} if for every $I \subseteq \ell (p)$ and $A \subseteq |N|$ both of size strictly less than $\kappa$, we have that $p^I \rest A$ is realized in $M$. 
    \item We say that $p$ is a \emph{$<\kappa$-heir over $M$} if for every $I \subseteq \ell (p)$  and every $A_0 \subseteq |M|$, $N_0 \lea N$, with $A_0 \subseteq |N_0|$ and $I, A_0, N_0$ all of size less than $\kappa$, there is some $f: N_0 \xrightarrow[A_0]{} M$ such that $$f (p^I \rest N_0) = p^I \rest f[N_0]$$
  \end{enumerate}
\end{defin}

$<\kappa$-satisfiable is also called $\kappa$-coheir.  As expected from first-order, these notions are dual\footnote{$A$ must be a model for the question to make sense.} and they are equivalent under the $\kappa$-order property of length $\kappa$.

$<\kappa$-satisfiability turns out to be an independence relation in the stable context.

\begin{theorem}\label{coheir-thm}
  Let $\K$ be an AEC and $\kappa > \LS(\K)$. Assume:
\begin{enumerate}
\item $\K$ has a monster model and is fully $<\kappa$-tame and -type short.
\item $\K$ does not have the $\kappa$-order property of length $\kappa$.
\end{enumerate}

Let $\dnf$ be the independence relation induced by $<\kappa$-satisfiability on the $\kappa$-Galois-saturated models of $\K$. Then $\dnf$ has disjointness, symmetry, local character, transitivity, and the $\kappa$-witness property. Thus \emph{if} $\dnf$ also has extension, then it is a stable independence relation on the $\kappa$-Galois-saturated models of $\K$.
\end{theorem}

If $\kappa = \beth_\kappa$, then it turns out that not having the $\kappa$-order property of length $\kappa$ is equivalent to not having the order property, which by Theorem \ref{stab-spectrum} is equivalent to stability. 

Note that the conclusion gives already that the AEC is stable. Similarly, the $<\kappa$-satisfiability relation analyzes a type by breaking it up into its $\kappa$-sized components, so the tameness and type shortness assumptions seem natural\footnote{Although it is open if they are necessary.}.  %% The Extension assumption, on the other hand, is a strong assumption.  Initially, Boney and Grossberg used Extension to derive other properties, but Vasey \cite{vinfstab} was later able to derive the Symmetry from the no order property assumption and show that was enough to derive the other properties from it.

Theorem \ref{coheir-thm} does not tell us if $<\kappa$-satisfiability has the extension property. At first glance, it seems to be a compactness result about Galois types. In fact:

\begin{thm}\label{coheir-ext-lc}
  Under the hypotheses of Theorem \ref{coheir-thm}, if $\kappa$ is a strongly compact cardinal, then $<\kappa$-satisfiability has the extension property.
\end{thm} 

Extension also holds in some nonelementary classes (such as averageable classes) and we will see that it ``almost'' follows from superstability (see Section \ref{global-indep-sec}).

The existence of a reasonable independence notion for stable classes can be combined with averages to obtain a result on chains of Galois-saturated models:

\begin{thm}\label{chainsat-stable}
  Let $\K$ be a $\LS (\K)$-tame AEC with amalgamation. If $\K$ is Galois-stable in some $\mu \ge \LS (\K)$, then there exists $\chi < H_1$ satisfying the following property:

  If $\lambda \ge \chi$ is such that $\K$ is Galois-stable in $\mu$ for unboundedly many $\mu < \lambda$, then whenever $\seq{M_i : i < \delta}$ is a chain of $\lambda$-Galois-saturated models and $\cf{\delta} \ge \chi$, we have that $\bigcup_{i < \delta} M_i$ is $\lambda$-Galois-saturated. 
\end{thm}
\begin{proof}[Proof sketch]
First note that Theorem \ref{stab-transfer}.(2) and tameness imply that $\K$ is Galois-stable in stationary many cardinals.  Then, develop enough of the theory of averages (and also investigate their relationship with forking) to be able to imitate Harnik's first-order proof \cite{harniksat}.
\end{proof}

We will see that this can be vastly improved in the superstable case: the hypothesis that $\K$ be Galois-stable in $\mu$ for unboundedly many $\mu < \lambda$ can be removed and the Hanf number improved. Moreover, there is a proof of a version of the above theorem using only independence calculus and not relying on averages. Nevertheless, the use of averages has several other applications (for example getting solvability from superstability, see Theorem \ref{gvsuperstab}).

\subsection{Superstability}  As noted at the beginning of Section \ref{class-thy-sec}, Shelah has famously stated that superstability in AECs suffers from ``schizophrenia''.  However superstability is much better behaved in tame AECs than in general. Recall Definition \ref{ss-def} which gave a definition of superstability in a single cardinal using local character of splitting. Recall also that there are several other local candidates such as the uniqueness of limit models (Definition \ref{lim-def}) and the existence of a good frame (Section \ref{frame-sec} and Definition \ref{good-frame-def}). Theorem \ref{chainsat-stable} suggests another definition saying that the union of a chain of $\mu$-Galois-saturated models is $\mu$-Galois-saturated. As noted before, it is unclear whether these definitions are equivalent cardinal by cardinal, that is, $\mu$-superstability and $\lambda$-superstability for $\mu \neq \lambda$ are potentially different notions and it is not easy to combine them. With tameness, this difficulty disappears:

\begin{thm}\label{ss-implies-all}
  Assume that $\K$ is $\mu$-superstable, $\mu$-tame, and has amalgamation. Then for every $\lambda > \mu$:

  \begin{enumerate}
    \item\label{ss-1} $\K$ is $\lambda$-superstable.
    \item\label{ss-2} If $\seq{M_i : i < \delta}$ is an increasing chain of $\lambda$-Galois-saturated models, then $\bigcup_{i < \delta} M_i$ is $\lambda$-Galois-saturated.
    \item\label{ss-3} There is a good $\lambda$-frame with underlying class $\Ksatp{\lambda}$.
    \item\label{ss-4} $\K$ has uniqueness of limit models in $\lambda$. In fact, $\K$ also has uniqueness of limit models in $\mu$.
  \end{enumerate}
\end{thm}
\begin{proof}[Proof sketch]
  Fix $\lambda > \mu$. We can first prove an approximation to (\ref{ss-3}) by defining forking as in Definition \ref{mu-forking-def} and following the proof of Theorem \ref{stable-forking-def}. We obtain an independence relation $\is$ on 1-types whose underlying class is $\Ksatp{\lambda}$ (at that point we do not yet know yet if it is an AEC), and which satisfies all the properties from the definition of a good frame\footnote{Note that tameness was crucial to obtain the uniqueness property.} (including the structural properties on $\K$) except perhaps symmetry.

  Still we can use this to prove that $\K$ satisfies (\ref{split assm}) in Definition \ref{ss-def}. Using just this together with uniqueness, we can show that $\K$ is Galois-stable in $\lambda$. Joint embedding follows from amalgamation and no maximal models holds by a variation on a part of the proof of Theorem \ref{frame-transfer}. Therefore (\ref{ss-1}) holds: $\K$ is $\lambda$-superstable. We can prove the symmetry property of the good $\lambda$-frame by proving that a failure of it implies the order property. This also give the symmetry property for splitting, and hence by Theorem \ref{uq-limit-symmetry} the condition (\ref{ss-4}), uniqueness of limit models in $\lambda$, holds. Uniqueness of limit models can in turn be used to obtain (\ref{ss-2}), hence the underlying class of $\is$ is really an AEC so (\ref{ss-3}) holds.
\end{proof}

Strikingly, a converse to Theorem \ref{ss-implies-all} holds. That is, several definitions of superstability are eventually equivalent in the tame framework:

\begin{theorem}\label{gvsuperstab}
Let $\K$ be a tame AEC with a monster model and assume that $\K$ is Galois-stable in unboundedly many cardinals.  The following are equivalent:
\begin{enumerate}
	\item For all high enough $\lambda$, the union of a chain of $\lambda$-Galois-saturated models is $\lambda$-Galois-saturated.
	\item \label{ulm} For all high enough $\lambda$, $\K$ has uniqueness of limit models in $\lambda$.
	\item For all high enough $\lambda$, $\K$ has a superlimit model of size $\lambda$.
	\item\label{solvable-cond} There is $\theta$ such that, for all high enough $\lambda$, $\K$ is $(\lambda, \theta)$-solvable.
	\item For all high enough $\lambda$, $\K$ is $\lambda$-superstable.
	\item For all high enough $\lambda$, there is $\kappa = \kappa_\lambda \leq \lambda$ such that there is a good $\lambda$-frame on $\K_\lambda^{\kappa\text{-sat}}$.
\end{enumerate}
Any of these equivalent statements also implies that $\K$ is Galois-stable in all high enough $\lambda$.
\end{theorem}

Note that the ``high enough'' threshold can potentially vary from item to item.  Also, note that the stability assumption in the hypothesis is not too important: in several cases, it follows from the assumption and, in others (such as the uniqueness of limit models), it is included to ensure that the condition is not vacuous. Finally, if $\K$ is $\LS (\K)$-tame, we can add in each of that $\lambda < H_1$ in each of the conditions (except in (\ref{solvable-cond}) where we can say that $\theta < H_1$).

Superlimit models and solvability both capture the notion of the AEC $\K$ having a ``categorical core'', a sub-AEC $\K_0$ that is categorical in some $\kappa$.  In the case of superlimits, $M \in \K_\kappa$ is superlimit if and only if $M$ is universal\footnote{That is, every model of size $\kappa$ embeds into $M$.}  and the class of models isomorphic to $M$ generates a nontrivial AEC. That is, the class: 
$$\{N \in \K_{\ge \kappa} \mid \forall N_0 \in P_{\kappa^+}^* (N) \exists  N_1 \in P_{\kappa^+}^* N : N_0 \lea N_1 \land N_1 \cong M\}$$
 is an AEC with a model of size $\kappa^+$\footnote{An equivalent definition: $M \in \K_{\kappa}$ is superlimit if and only if it is universal, has a proper extension isomorphic to it, and for any limit $\delta < \kappa^+$, and any increasing continuous chain $\seq{M_i : i \le \delta}$, if $M_i \cong M$ for all $i < \delta$, then $M_\delta \cong M$.}. $(\lambda, \kappa)$-solvability further assumes that this superlimit is isomorphic to $\EM_{L(\K)}(I, \Phi)$ for some proper $\Phi$ of size $\kappa$ and \emph{any} linear order $I$ of size $\lambda$.

Note that although we did not mention them in Section \ref{primer-no-tameness}, superlimits and especially solvable AECs play a large role in the study of superstability in general AECs (see the historical remarks).

The proof that superstability implies solvability relies on a characterization of Galois-saturated models using averages (essentially, a model $M$ is Galois-saturated if and only if for every type $p \in \gS (M)$, there is a long-enough Morley sequence $\BI$ inside $M$ whose average is $p$). We give the idea of the proof that a union of Galois-saturated models being Galois-saturated implies superstability. This can also be used to derive superstability from categoricity in the tame framework (without using the much harder proof of the Shelah-Villaveces Theorem \ref{shvi}). 

\begin{lem}\label{tame-superstab}
  Let $\K$ be an AEC with a monster model. Assume that $\K$ is $\LS (\K)$-tame and let $\kappa = \beth_\kappa > \LS (\K)$ be such that $\K$ is Galois-stable in $\kappa$. Assume that for all $\lambda \ge \kappa$ and all limit $\delta$, if $\seq{M_i : i < \delta}$ is an increasing chain of $\lambda$-Galois-saturated models, then $\bigcup_{i < \delta} M_i$ is $\lambda$-Galois-saturated. 

  Then $\K$ is $\kappa^+$-superstable.
\end{lem} 
\begin{proof}[Proof sketch]
  By tameness and Theorem \ref{stab-transfer}.(1), we have that $\K$ is Galois-stable in $\kappa^+$.  Thus, we only have to show Definition \ref{ss-def}.(\ref{split assm}), that there are no long splitting chains. There is a Galois-saturated model in $\kappa^+$ and, by a back and forth argument, it is enough to show  Definition \ref{ss-def}.(\ref{split assm}) when all the models are $\kappa^+$-Galois-saturated. 
  
Let $\delta < \kappa^{++}$ be a limit ordinal and let $\seq{M_i : i \le \delta}$ be an increasing continuous chain of Galois-saturated models in $\K_{\kappa^+}$; that we can make the models at limit stages Galois-saturated crucially uses the assumption. Let $p \in \gS (M_\delta)$. We need to show that there is $i < \delta$ such that $p$ does not $\kappa^+$-split over $M_i$.  By standard means, one can show that there is an $i < \delta$ such that $p$ is $<\kappa^+$ satisfiable in $M_i$.  Tameness  gives the uniqueness of $<\kappa$-satisfiability, which allows us to conclude that $p$ is $<\kappa$-satisfiable in $M_i$, which in turn implies that $p$ does not $\kappa^+$-split over $M_i$, as desired.
\end{proof}

\begin{remark}\label{rmk-chainsat}
  From the argument, we obtain the following intriguing consequence in first-order model theory\footnote{Hence showing that perhaps the study of AEC can also lead to new theorems in first-order model theory.}: if $T$ is a stable first-order theory, $\seq{M_i : i \le \delta}$ is an increasing continuous chain of $\aleph_1$-saturated models (so $M_i$ is $\aleph_1$-saturated also for limit $i$), then for any $p \in \Ss (M_\delta)$, there exists $i < \delta$ so that $p$ does not fork over $M_i$. This begs the question of whether any such chain exists in strictly stable theories. 
\end{remark}

We now go back to the study of good frames. One can ask when instead of a good $\lambda$-frame, we obtain a good $(\ge \lambda)$-frame (i.e.\ forking is defined for types over models of all sizes). It turns out that the proof of Theorem \ref{ss-implies-all} gives a good $(\ge \mu^+)$-frame on $\Ksatp{\mu^+}$. This still has the disadvantage of looking at Galois-saturated models. The next result starts from a good $\mu$-frame and shows that $\mu$-tameness can transfer it up (note that this was already stated as Theorem \ref{good-frame-transfer}):

\begin{theorem}\label{good-frame-transfer-2}
Assume $\K$ is an AEC with $\LS(\K) \leq \lambda$ and $\s$ is a good $\lambda$-frame on $\K$.  If $\K$ has amalgamation, then $\K$ is $\lambda$-tame if and only if there is a good $(\ge \lambda)$-frame $\geq \s$ on $\K$ that extends $\s$.
\end{theorem}

\begin{proof}[Proof sketch]
  That tameness is necessary is discussed in Example \ref{examples-subsec}.(\ref{stabletame}).
  
  For the other direction, it is easy to check that if there is any way to extends forking to models of size at least $\lambda$, the definition must be the following:
  \begin{center} $p \in \gS (M)$ does not fork over $M_0$ if and only if there exists $M_0' \lea M_0$ with $M_0' \in \K_{\lambda}$ and $p \rest M'$ does not fork over $M_0'$ for all $M' \lea M$ with $M' \in \K_\lambda$.
  \end{center}
Several frame properties transfer without tameness; however, the key properties of uniqueness, extension, stability, and symmetry can fail.
$\lambda$-tameness can be easily seen to be equivalent to the transfer of uniqueness from $\s$ to $\geq \s$.  Using uniqueness, extension and stability can easily be shown to follow.  Symmetry is harder and the proof goes through independent sequences (see below and the historical remarks).

As an example, we show how to prove the extension property. Note that one of the key difficulties in proving extension in general is that upper bounds of types need not exist; while this is trivial in first-order, such AECs are called compact (see Definition \ref{tame-var-def}).  To solve this problem, we use the forking machinery of the frame to build a chain of types with a canonical extension at each step.  This canonicity provides the existence of types.

  Let $M \in \K_{\ge \lambda}$ and let $p \in \gS (M)$. Let $N \gea M$. We want to find a nonforking extension of $p$ to $N$. By local character and transitivity, without loss of generality $M \in \K_{\lambda}$. We now work by induction on $\mu := \|N\|$. If $\mu = \lambda$, we know that $p$ can be extended to $N$ by definition of a good frame, so assume $\mu > \lambda$. Write $N = \bigcup_{i < \mu} N_i$, where $N_i \in \K_{\lambda + |i|}$. By induction, let $p_i \in \gS (N_i)$ be the nonforking extension of $p$ to $N_i$. Note that by uniqueness $p_j \rest N_i = p_i$ for $i \le j < \mu$. We want to take the ``direct limit'' of the $p_i$'s: build $\seq{f_i : i < \mu}$, $\seq{N_i' : i < \mu}$, $\seq{a_i : i < \mu}$ such that $p_i = \gtp (a_i / N_i; N_i')$, $f_i : N_i' \xrightarrow[N_i]{} N_{i + 1}'$ such that $f_i (a_i) = a_{i + 1}$. If this can be done, then taking the direct limit of the system induced by $\seq{f_i, a_i N_i' : i < \mu}$, we obtain $a_\mu, N_{\mu}'$ such that $\gtp (a_\mu / N_\mu; N_{\mu}')$ is a nonforking extension of $p$. How can we build such a system? The base and successor cases are no problem, but at limits, we want to take the direct limit and prove that everything is still preserved. This will be the case because of the local character and uniqueness property. 

\end{proof}

This should be compared to Theorem \ref{successful-cond} which achieves the more modest goal of transferring $\s$ to $\lambda^+$ (over Galois-saturated models and with a different ordering) with assumptions on the number of models and some non-ZFC hypotheses. 

An interesting argument in the proof of Theorem \ref{good-frame-transfer-2} is the transfer of the symmetry property. One could ignore that issue and use that failure of the order property implies symmetry, however this would make the argument non-local in the sense that we require knowledge about the AEC near the Hanf of $\lambda$ to conclude good property at $\lambda$. A more local (but harder) approach is to study \emph{independent sequences}.  

Given a good $(\geq \lambda)$-frame and $M_0 \lea M \lea N$, we want to say that a sequence $\seq{a_i \in N: i < \alpha}$ is independent in $(M_0, M, N)$ if and only if $\gtp (a_i / |M| \cup \{a_j : j < i\}; N)$ does not fork over $M_0$. However, forking behaves better for types over models so instead, we require that there is a sequence of models $M \lea N_i \lea N$ growing with the sequence $\seq{a_i : i < \alpha}$ such that $a_i \in |N_{i+1}| \backslash |N_i|$ and require $\gtp(a_i/N_i; N)$ does not fork over $M_0$.

%\begin{defin}\label{indep-seq-def}
%  Let $\s$ be a good $(\ge \lambda)$-frame. Let $N \in \K_{\ge \lambda}$.

%  \begin{enumerate}
%    \item We say that $\BI = \seq{a_i : i < \alpha}$ is \emph{independent in $N$} if $\BI \in \fct{\alpha}{|N|}$ and there exists an increasing %continuous chain $\seq{N_i : i \le \alpha}$ and $N' \gea N$ such that $N_\delta \lea N'$, and for all $i < \alpha$, $a_i \in |N_{i + 1}|$ and $\gtp (a_i / N_i; N_{i + 1})$ does not fork over $N_0$.
%    \item More generally, we say that $\BI = \seq{a_i : i < \alpha}$ is \emph{independent in $(M_0, M, N)$} if above we require in addition that $M \lea N_0$ and $\gtp (a_i / N_i; N_{i + 1})$ does not fork over $M_0$ (so $\BI$ is independent in $N$ if and only if it is independent in $(M_0, M, N)$ for some $M_0 \lea M \lea N$).
%    \item We say that a set $I$ is \emph{independent in $(M_0, M, N)$} if some enumeration of it is independent in $(M_0, M, N)$. Similarly define being independent in $N$.
%  \end{enumerate}
%\end{defin}

The study of independent sequences shows that under tameness they themselves form (in a certain technical sense) a good frame. That is, from an independence relation for types of length one, we obtain an independence relation for types \emph{of independent sequences} of all lengths. One other ramification of the study of independence sequence is the isolation of a good notion of dimension: inside a fixed model, any two infinite maximal independent sets must have the same size.

\begin{thm}\label{dim-thm}
  Let $\K$ be an AEC, $\lambda \ge \LS (\K)$. Assume that $\K$ is $\lambda$-tame and has amalgamation. Let $\s$ be a good $(\ge \lambda)$-frame on $\K$. Let $M_0 \lea M \lea N$ all be in $\K_{\ge \lambda}$.

  \begin{enumerate}
    \item Symmetry of independence: For a fixed set $I$, $I$ is independent in $(M_0, M, N)$ if and only if \emph{all} enumerations are independent in $(M_0, M, N)$.
    \item Let $p \in \gS (M)$. Assume that $I_1$ and $I_2$ are independent in $(M_0, M, N)$ and every $a \in I_1 \cup I_2$ realizes $p$. If both $I_1$ and $I_2$ are $\subseteq$-maximal with respect to that property and $I_1$ is infinite, then $|I_1| = |I_2|$.
  \end{enumerate}
\end{thm}

%% Without tameness, Jarden and Sitton \cite{jasi} are able to prove the same result (and did so first), but need to replace tameness with the assumptions that the frame has conjugation and existence for $\K^{3,uq}$, which typically requires combinatorial principles beyond ZFC.

\subsection{Global independence and superstability}\label{global-indep-sec}

Combined with Theorem \ref{ss-implies-all}, Theorem \ref{good-frame-transfer-2} shows that every tame superstable AEC has a good $(\ge \lambda)$-frame. It is natural to ask whether this frame can also be extended in the other direction: to types of length larger than one. More precisely, we want to build a superstability-like global independence relation (i.e.\ the global version of a good frame):

\begin{defin}\label{good-indep-def}
  We say an independence relation $\dnf$ on $\K$ is \emph{good} if:

  \begin{enumerate}
    \item $\K$ is an AEC with amalgamation, joint embedding, and arbitrarily large models.
    \item $\K$ is Galois-stable in all $\mu \ge \LS (\K)$.
    \item $\dnf$ has disjointness, symmetry, existence, uniqueness, extension, transitivity, and the $\LS (\K)$-witness property.
    \item For all cardinals $\alpha > 0$:
      \begin{enumerate}
        \item $\slc{\alpha} (\nf) = (\alpha + \LS (\K))^+$.
        \item $\clc{\alpha} (\nf) = \alpha^+ + \aleph_0$.
      \end{enumerate}
  \end{enumerate}

  We say that an AEC $\K$ is \emph{good} if there exists a good independence relation on $\K$.
\end{defin}

We would like to say that if $\K$ is a $\LS (\K)$-superstable AEC with amalgamation that is fully tame and short, then there exists $\lambda$ such that $\K_{\ge \lambda}$ is good. At present, we do not know if this is true (see Question \ref{indep-question}). All we can conclude is a weakening of good:

\begin{defin}\label{almost-good-def}
  We say an independence relation $\nf$ is \emph{almost good} if it satisfies all the conditions of Definition \ref{good-indep-def} except it only has the following weakening of extension: If $p \in \gS^\alpha (M)$ and $N \gea M$, we can find $q \in \gS^\alpha (N)$ extending $p$ and not forking over $M$ provided that at least one of the following conditions hold:

  \begin{enumerate}
    \item $M$ is Galois-saturated.
    \item $M \in \K_{\LS (\K)}$.
    \item $\alpha < \LS (\K)^+$.
  \end{enumerate}

  An AEC $\K$ is \emph{almost good} if there is an almost good independence relation on $\K$.
\end{defin}
\begin{remark}
  Assume that $\is$ is an independence relation on $\K$ which satisfies all the conditions in the definition of good except extension, and it has extension for types over Galois-saturated models. Then we can restrict $\is$ to $\Ksatp{\LS (\K)^+}$ and obtain an almost good independence relation. Thus extension over Galois-saturated models is the important condition in Definition \ref{almost-good-def}.
\end{remark}

We can now state a result on existence of global independence relation: 

\begin{thm}\label{almost-good-thm}
  Let $\K$ be a fully $\LS (\K)$-tame and short AEC with amalgamation. Let $\lambda := \left(2^{\LS (\K)}\right)^{+4}$. If $\K$ is $\LS (\K)$-superstable, then $\Ksatp{\lambda}$ is almost good.
\end{thm}

We try to describe the proof. For simplicity, we will work with $<\kappa$-satisfiability, so will obtain a Hanf number approximately equal to a fixed point of the beth function. The better bound is obtained by looking at splitting but this makes the proof somewhat more complicated. So let $\kappa = \beth_\kappa > \LS (\K)$. We know that the $<\kappa$-satisfiability independence relation is an independence relation on $\Ksatp{\kappa}$ with uniqueness, local character, and symmetry (but not extension). Let $\is$ denote this relation independence relation. Furthermore we can show that $\clc{1} (\is) = \aleph_0$. In fact, $\is$ restricted to types of length one induces a good $\kappa$-frame $\s$ on $\Ksatp{\kappa}_\kappa$. We would like to extend $\s$ to types of length at most $\kappa$.

To do this, we need to make use of the notion of domination and successful frames\footnote{Note that the definitions here do not coincide with Shelah's, although they are equivalent in our context. The equivalence uses tameness again, including a result of Adi Jarden. See the historical remarks for more.}:

\begin{defin}\label{successful-def} Suppose $\nf$ is an independence relation on $\K$. Work inside a monster model\footnote{So if $\sea$ is the monster model, $a \nf_M B$ means $\nfs{M}{a}{B}{\sea}$.}.
  \begin{enumerate}
    \item For $M \lea N$ $\kappa$-Galois-saturated and $a \in |N|$, \emph{$a$ dominates $N$ over $M$} if for any $B$, $a \nf_M B$ implies $N \nf_M B$. 
    \item $\s$ is \emph{successful} if for every Galois-saturated $M \in \K_\kappa$, every nonalgebraic type $p \in \gS (M)$, there exists $N \gea M$ and $a \in |N|$ with $N \in \K_\kappa$ Galois-saturated such that $a$ dominates $N$ over $M$.
    \item $\s$ is \emph{$\omega$-successful} if $\s^{+n}$ is successful for all $n < \omega$. Here, $\s^{+n}$ is the good $\kappa^{+n}$ induced on the Galois-saturated models of size $\kappa^{+n}$ by $<\kappa$-satisfiability.
  \end{enumerate}
\end{defin}

An argument of Makkai and Shelah \cite[Proposition 4.22]{makkaishelah} shows that $\s$ is successful (in fact $\omega$-successful), and a deep result of Shelah shows that if $\s$ is successful, then we can extend $\s$ to a $\kappa$-ary independence relation $\is'$ which has extension, uniqueness, symmetry, and for all $\alpha \le \kappa$, $\clc{\alpha} (\is') = \alpha^+ + \aleph_0$. This completes the first step of the proof.  Note that we have taken $\is$ (which was built on $<\kappa$-satisfiability), restricted it to 1-types and then ``lengthened'' it to $\kappa$-ary types.  However, we do not necessarily get $<\kappa$-satisfiability back!  We do get, however, an independence relation with a better local character property.

From $\omega$-successfulness, we could extend the frame $\s$ to models of size $\kappa^{+n}$.  Now we would like to extend $\is'$ to models of all sizes above $\kappa$.  However, the continuity of $\is'$ is not strong enough. The missing property is:

\begin{defin}
  An independence relation $\is = (\K, \nf)$ has \emph{full model continuity} if for any limit ordinal $\delta$, for any increasing continuous chain $\seq{M_i^\ell : i \le \delta}$ with $\ell < 4$, and $M_i^0 \lea M_i^k \lea M_i^3$ for $k = 1,2$ and $i \le \delta$, if $\nfs{M_i^0}{M_i^1}{M_i^2}{M_i^3}$ for all $i < \delta$, then $\nfs{M_\delta^0}{M_\delta^1}{M_\delta^2}{M_\delta^3}$.

  Let us say that $\is$ is \emph{fully good} [\emph{almost fully good}] if it is good [almost good] and has full model continuity. As before, $\K$ is \emph{[almost] fully good} if it there is an [almost] fully good independence relation on $\K$.
\end{defin}

Another powerful result of Shelah \cite[III.8.19]{shelahaecbook} connects $\omega$-successful good frames with full model continuity.  Suppose that $\s$ is an $\omega$-successful good $\kappa$-frame (as we have).  We do not know that $\is'$ defined above has full model continuity, but it we move to the (still $\omega$-successful) good $\kappa^{+3}$-frame $\s^{+3}$ and ``lengthen'' this to an independence relation $\is'_{+3}$ on $\kappa^{+3}$-ary types, then $\is'_{+3}$ has full model continuity!

This allows us to transfer all of the nice properties of $\is'_{+3}$ to a $\kappa^{+3}$-ary independence relation $\is''$ on models of all sizes above $\kappa^{+3}$.  To get a truly global independence relation, we can define an independence relation $\is'''$ on types of \emph{all} lengths by specifying that $p \in \gS^{\alpha} (M)$ do not $\is'''$-fork over $M_0 \lea M$ if and only if $p \rest I$ does not $\is''$-fork over $M_0$ for every $I \subseteq \alpha$ with $|I| \le \kappa^{+3}$. With some work, we can show that $\is'''$ is almost fully good (thus ``fully'' can be added to the conclusion of Theorem \ref{almost-good-thm}).

What about getting the extension over property over all models (not just the Galois-saturated models). It is known how to do it by making one more locality hypothesis:

\begin{defin}[Type locality]\label{type-loc-def} \
  \begin{enumerate}
    \item Let $\delta$ be a limit ordinal, and let $\bar{p} := \seq{p_i : i < \delta}$ be an increasing chain of Galois types, where for $i < \delta$, $p_i \in \gS^{\alpha_i} (M)$ and $\seq{\alpha_i : i \le \delta}$ are increasing. We say $\bar{p}$ is \emph{$\kappa$-type-local} if $\cf{\delta} \ge \kappa$ and whenever $p, q \in \gS^{\alpha_\delta} (M)$ are such that $p^{\alpha_i} = q^{\alpha_i} = p_i$ for all $i < \delta$, then $p = q$.
    \item We say $\K$ is \emph{$\kappa$-type-local} if every $\bar{p}$ as above is $\kappa$-type-local.
  \end{enumerate}
\end{defin}

We think of $\kappa$-type-locality as the dual to $\kappa$-locality (Definition \ref{tame-var-def}.(3)) in the same sense that shortness is the dual to tameness. 

\begin{remark}
  If $\kappa$ is a regular cardinal and $\K$ is $<\kappa$-type short, then $\K$ is $\kappa$-type-local. In particular, if $\K$ is fully $<\aleph_0$-tame and -type short, then $\K$ is $\aleph_0$-type-local.
\end{remark}
\begin{remark}
If there is a good $\lambda$-frame on $\K$, then $\K_\lambda$ is $\aleph_0$-local (use local character and uniqueness), and thus assuming $\lambda$-tameness $\K$ is $\aleph_0$-local. This is used in the transfer of a good $\lambda$-frame to a good $(\ge \lambda)$-frame. Unfortunately, an analog for this fact is missing when looking at $\aleph_0$-type-locality, i.e.\ it is not clear that even a fully good AEC is $\aleph_0$-type-local.
\end{remark}

Using type-locality, we can start from a fully good $\LS (\K)$-ary independence relation on $\K$ and prove extension for types of all lengths. Thus we obtain the following variation of Theorem \ref{almost-good-thm}:

\begin{thm}\label{fully-good-thm}
  Let $\K$ be a fully $\LS (\K)$-tame and short AEC with amalgamation. Assume that $\K$ is $\aleph_0$-type-local. Let $\lambda := \left(2^{\LS (\K)}\right)^{+4}$. If $\K$ is $\LS (\K)$-superstable, then $\Ksatp{\lambda}$ is fully good.
\end{thm}
\begin{remark}
  It is enough to assume that $\aleph_0$-type-locality holds ``densely'' in a certain technical sense. See the historical remarks. 
\end{remark}

Finally, we know of at least two other ways to obtain extension: using total categoricity and large cardinals. We collect all the results of this section in a corollary:

\begin{cor}\label{fully-good-cor}
  Let $\K$ be an AEC. Assume that $\K$ is $\LS (\K)$-superstable and fully $\LS (\K)$-tame and short.

  \begin{enumerate}
    \item If $\kappa > \LS (\K)$ is a strongly compact cardinal, then $\Ksatp{\kappa}$ is fully good.
    \item If either $\K$ is $\aleph_0$-type-local (e.g.\ it is fully $(<\aleph_0)$-tame and short) or $\K$ is totally categorical, then $\Ksatp{\lambda}$ is fully good, where $\lambda := \left(2^{\LS (\K)}\right)^{+4}$.
  \end{enumerate}
\end{cor}
\begin{proof}[Proof sketch]
  By Theorem \ref{almost-good-thm}, $\Ksatp{\lambda}$ is almost good, and in fact (as we have discussed) almost fully good. If $\K$ is totally categorical, all the models are Galois-saturated and hence by definition of almost fully good, $\K$ is fully good. If $\K$ is $\aleph_0$-type-local, then apply Theorem \ref{fully-good-thm}. Finally, if $\kappa > \LS (\K)$ is strongly compact, then the extension property for $<\kappa$-satisfiability holds (see Theorem \ref{coheir-ext-lc}) and using a canonicity result similar to Theorem \ref{indep-canon-thm} one can conclude that $\Ksatp{\kappa}$ is fully good.
\end{proof}

Since the existence of a strongly compact cardinal implies full tameness and shortness (see Theorem \ref{tamelc-fact}), we can state a version of the first part of Corollary \ref{fully-good-cor} as follows: 

\begin{thm}\label{fully-good-strong-compact}
If $\K$ is an AEC which is superstable in every $\mu \ge \LS (\K)$ and $\kappa > \LS (\K)$ is a strongly compact cardinal, then $\Ksatp{\lambda}$ is fully good, where $\lambda := \left(2^{\LS (\K)}\right)^{+4}$.
\end{thm}

Note that in all of the results above, we are restricting ourselves to classes of sufficiently saturated models. This is related to the fact that the uniqueness property is required in the definition of a good independence relation, i.e.\ all types must be stationary. But what if we relax this requirement? Can we obtain an independence relation that specifies what it means to fork over an arbitrary set? A counterexample of Shelah \cite[Section 4]{hyttinen-lessmann} shows that this cannot be done in general. However this is possible for universal classes:

\begin{thm}\label{univ-classes-goodness}
  If $\K$ is an almost fully good universal class, then:

  \begin{enumerate}
    \item $\K$ is fully good.
    \item We can define $\nfs{A_0}{A}{B}{N}$ (for $A_0$ an arbitrary set) to hold if and only if $\nfs{\cl^{N} (A_0)}{\cl^{N} (A_0 A)}{\cl^{N} (A_0 B)}{N}$. Here $\cl^N$ is the closure under the functions of $N$. This has the expected properties (extension, existence, local character).
    \item This also has the finite witness property: $\nfs{A_0}{A}{B}{N}$ if and only if $\nfs{A_0}{A'}{B'}{N}$ for all $A' \subseteq A$, $B' \subseteq B$ finite.
  \end{enumerate}
\end{thm}
\begin{remark}
  It is enough to assume that $\K$ admits intersections, i.e.\ for any $N \in \K$ and any $A \subseteq |M|$, $\bigcap\{M \lea N \mid A \subseteq |M|\} \lea N$.
\end{remark}

\subsection{Categoricity}\label{categ-tame-subsec}

One of the first marks made by tame AEC was the theorem by Grossberg and VanDieren \cite{tamenessthree} that tame AECs (with amalgamation) satisfy an \emph{upward} categoricity transfer from a successor (see Theorem \ref{gv-upward}). Combining it with Theorem \ref{shelah-succ-downward}, we obtain that tame AECs satisfy Shelah's eventual categoricity conjecture from a successor:

\begin{theorem}\label{succ-transfer}
  Let $\K$ be an $H_2$-tame AEC with amalgamation. If $\K$ is categorical in \emph{some} \emph{successor} $\lambda \ge H_2$, then $\K$ is categorical in \emph{all} $\lambda' \ge H_2$.
\end{theorem}

Recall that categoricity implies superstability below the categoricity cardinal (Theorem \ref{shvi}). A powerful result is that assuming tameness, superstability also holds above, while this need not be true without tameness; recall the discussion after Theorem \ref{shvi}. In particular, Question \ref{sat-categ} has a positive answer: the model in the categoricity cardinal is Galois-saturated.

\begin{thm}\label{ss-categ}
  Let $\K$ be a $\LS (\K)$-tame AEC with amalgamation and no maximal models. If $\K$ is categorical in some $\lambda > \LS (\K)$, then:

  \begin{enumerate}
    \item $\K$ is superstable in every $\mu \ge \LS (\K)$.
    \item For every $\mu > \LS (\K)$, there is a good $\mu$-frame with underlying class $\Ksatp{\mu}$.
    \item The model of size $\lambda$ is Galois-saturated.
  \end{enumerate}
\end{thm}
\begin{proof} \
  \begin{enumerate}
  \item By Theorem \ref{shvi}, $\K$ is superstable in $\LS (\K)$. Now apply Theorem \ref{ss-implies-all}.
  \item As above, using Theorem \ref{ss-implies-all}.
  \item $\K$ is $\lambda$-superstable, so in particular Galois-stable in $\lambda$. It is not hard to build a $\mu^+$-Galois-saturated model in $\lambda$ for every $\mu < \lambda$ so the result follows from categoricity.
  \end{enumerate}
\end{proof}

Theorem \ref{ss-categ} allows one to show that a tame AEC categorical in some cardinal is categorical in a closed unbounded set of cardinals of a certain form. This already plays a key role in Shelah's proof of Theorem \ref{shelah-succ-downward}. The key is what we call Shelah's omitting type theorem, a refinement of Morley's omitting type theorem.  Note that a version of this theorem is also true without tameness, but removing the tameness assumption changes the condition on $p$ being omitted to requiring that the small approximations to $p$ be omitted\footnote{In the sense that each element omits \emph{some} small approximation of $p$.}.

\begin{thm}[Shelah's omitting type theorem]\label{shelah-omit-type}
  Let $\K$ be a $\LS (\K)$-tame AEC with amalgamation. Let $M_0 \lea M$ and let $p \in \gS (M_0)$. Assume that $p$ is omitted in $M$. If $\|M_0\| \ge \LS (\K)$ and $\|M\| \ge \beth_{\left(2^{\LS (\K)}\right)^+} (\|M_0\|)$, then %there exists $M_0' \in \K_{\LS (\K)}$ such that $M_0' \lea M_0$ and $p \rest M_0'$ is omitted in models of arbitrarily large sizes. In particular, 
  there is a non-$\LS (\K)^+$-Galois-saturated model in every cardinal.
\end{thm}
%\begin{remark}
%  A more localized statement can be given when instead of requiring $\LS (\K)$-tameness, we only ask that every element of $M$ omits some $\LS (\K)$-sized restriction of $p$.
%\end{remark}

\begin{cor}\label{omit-type-cor}
  Let $\K$ be a $\LS (\K)$-tame AEC with amalgamation and no maximal models. If $\K$ is categorical in \emph{some} $\lambda > \LS (\K)$, then $\K$ is categorical in all cardinals of the form $\beth_{\delta}$, where $\left(2^{\LS (\K)}\right)^+$ divides $\delta$.
\end{cor}
\begin{proof}
  Let $\delta$ be divisible by $\left(2^{\LS (\K)}\right)^+$. If there is a model $M \in \K_{\beth_\delta}$ which is not Galois-saturated, then by Shelah's omitting type theorem we can build a non $\LS (\K)^+$-Galois-saturated model in $\lambda$. This contradicts Theorem \ref{ss-categ}.
\end{proof}

In section \ref{universal-class-sec}, a categoricity transfer was proven \emph{without} assuming that the categoricity cardinal is a successor. As hinted at there, this generalizes to tame AECs that \emph{have primes} (recall from Definition \ref{prime-def} that an AEC has primes if there is a prime model over every set of the form $M \cup \{a\}$):

\begin{thm}\label{event-categ-primes}
  Let $\K$ be an AEC with amalgamation and no maximal models. Assume that $\K$ is $\LS (\K)$-tame and has primes. If $\K$ is categorical in some $\lambda > \LS (\K)$, then $\K$ is categorical in all $\lambda' \ge \min (\lambda, H_1)$.
\end{thm}
\begin{remark}\label{event-categ-primes-rmk}
  A partial converse is true: if a fully tame and short AEC with amalgamation and no maximal models is categorical on a tail, then it has primes on a tail.
\end{remark}

We deduce Shelah's categoricity conjecture in homogeneous model theory (see Section \ref{examples-subsec}.(\ref{hommod})):

\begin{cor}\label{categ-homog}
  Let $D$ be a homogeneous diagram in a first-order theory $T$. If $D$ is categorical in a $\lambda > |T|$, then $D$ is categorical in all $\lambda' \ge \min (\lambda, \hanf{|T|})$.
\end{cor}

Using a similar argument, we can also get rid of the hypothesis that $\K$ has primes if the categoricity cardinal is a successor. This allows us to obtain a downward transfer for tame AECs which improves on Theorem \ref{succ-transfer} (there $H_1$ was $H_2$). The price to pay is to assume more tameness.

\begin{thm}\label{succ-transfer-2}
  Let $\K$ be a $\LS (\K)$-tame AEC with amalgamation and no maximal models. If $\K$ is categorical in a \emph{successor} $\lambda > \LS (\K)^+$, then $\K$ is categorical in all $\lambda' \ge \min (\lambda, H_1)$.
\end{thm}
\begin{proof}[Proof sketch]
  Let us work in a good $(\ge \LS (\K)^+)$-frame $\s$ on $\Ksatp{\LS (\K)^+}$ (this exists by Theorems \ref{shvi}, \ref{ss-implies-all}, and \ref{good-frame-transfer-2}). As in Section \ref{orthog-univ}, we can define what it means for two types $p$ and $q$ to be orthogonal (written $p \perp q$) and say that $\s$ is \emph{$\mu$-unidimensional}\footnote{In this framework, this definition need not be equivalent to categoricity in the next successor but we use it for illustrative purpose.} if no two types are orthogonal. We can show that $\s$ is $\mu$-unidimensional if and only if $\Ksatp{\LS (\K)^+}$ is categorical in $\mu^+$, and argue by studying the relationship between forking and orthogonality that $\s$ is unidimensional in some $\mu$ if and only if it is unidimensional in all $\mu'$ (this uses tameness, since we are moving across cardinals). Thus $\Ksatp{\LS (\K)^+}$ is categorical in every successor cardinal, hence also (by a straightforward argument using Galois-saturated models) in every limit. We conclude by using a version of Morley's omitting type theorem to transfer categoricity in $\Ksatp{\LS (\K)^+}$ to categoricity in $\K$ (this is where the $H_1$ comes from).
\end{proof}

What if we do not want to assume that the AEC has primes or that it is categorical in a successor? Then the best known results are essentially Shelah's results from Section \ref{frame-sec}. We show how to obtain a particular case using the results presented in this section:

\begin{thm}\label{categ-conj-tame-short}
  Assume Claim \ref{claim-xxx} and $2^{\theta} < 2^{\theta^+}$ for every cardinal $\theta$. Let $\K$ be a fully $\LS (\K)$-tame and short AEC with amalgamation and no maximal models. If $\K$ is categorical in some $\lambda \ge H_1$, then $\K$ is categorical in all $\lambda' \ge H_1$.
\end{thm} 
\begin{proof}[Proof sketch]
  By Theorem \ref{shvi}, $\K$ is $\LS (\K)$-superstable. By the proof of Theorem \ref{almost-good-thm}, we can find an $\omega$-successful good $\mu$-frame (see Definition \ref{successful-def}) on $\Ksatp{\mu}_{\mu}$ for some $\lambda < H_1$. By Claim \ref{claim-xxx}, $\Ksatp{\mu^{+\omega}}$ is categorical in every $\mu' > \mu^{+\omega}$. Using a version of Morley's omitting type theorem, we get that $\K$ must be categorical on a tail of cardinals, hence in a successor above $H_1$. By Theorem \ref{succ-transfer-2}, $\K$ is categorical in all $\lambda' \ge H_1$.
\end{proof}
\begin{remark}
  Slightly different arguments show that the locality assumption can be replaced by only $\LS (\K)$-tameness. Moreover, it can be shown that categoricity in some $\lambda > \LS (\K)$ implies categoricity in all $\lambda' \ge \min (\lambda, H_1)$.
\end{remark}

The proof shows in particular that almost fully good independence relations can be built in fully tame and short categorical AECs:

\begin{thm}\label{fully-good-indep}
  Let $\K$ be a fully $\LS (\K)$-tame and -type short AEC with amalgamation and no maximal models. If $\K$ is categorical in a $\lambda \ge \left(2^{\LS (\K)}\right)^{+4}$, then:
  \begin{enumerate}
    \item $\K_{\ge \min (\lambda, H_1)}$ is almost fully good.
    \item If $\K$ is fully $<\aleph_0$-tame and -type short, then $\K_{\ge \min (\lambda, H_1)}$ is fully good.
  \end{enumerate}
\end{thm}
\begin{proof}
  As in the proof above. Note that by Corollary \ref{omit-type-cor}, $\K$ is categorical in $H_1$.
\end{proof}

Using large cardinals, we can remove all the hypotheses except categoricity:

\begin{cor}\label{lc-cor}
  Let $\K$ be an AEC. Let $\kappa > \LS (\K)$ be a strongly compact cardinal. If $\K$ is categorical in a $\lambda \ge \hanf{\kappa}$, then:
  \begin{enumerate}
    \item $\K_{\ge \lambda}$ is fully good.
    \item If $2^{\theta} < 2^{\theta^+}$ for every cardinal $\theta$ and Claim \ref{claim-xxx} hold, then\footnote{The same conclusion holds assuming only that $\kappa$ is a measurable cardinal. Moreover, if $\K$ is axiomatized by an $L_{\kappa, \omega}$ theory, we can replace $\hanf{\kappa}$ with $\hanf{\kappa + \LS (\K)}$ and do not need to assume that $\kappa > \LS (\K)$.} $\K$ is categorical in all $\lambda' \ge \hanf{\kappa}$.
  \end{enumerate}
\end{cor}
\begin{proof}[Proof sketch]
  By a result similar to Theorem \ref{tameness-from-categ-2}, $\K_{\ge \kappa}$ has amalgamation and no maximal models. By Theorem \ref{tamelc-fact}, $\K$ is fully $<\kappa$-tame and -type short. Now Corollary \ref{fully-good-cor} gives the first part. Theorem \ref{categ-conj-tame-short} gives the second part.
\end{proof}

\section{Conclusion}\label{conclusion-sec}

The classification theory of tame AECs has progressed rapidly over the last several years.  The categoricity transfer results of Grossberg and VanDieren indicated that tameness (along with amalgamation, etc.) is a powerful tool to solve questions that currently seem out of reach for general AECs.

Looking to the future, there are several open question and lines of research that we believe deserve to be further explored.
\begin{enumerate}
	\item {\bf Levels of tameness}\\
	This problem is less grandiose than other concerns, but still concerns a basic unanswered question about tameness: are there nontrivial relationships between the parametrized versions of tameness in Definition \ref{tame-var-def}? For example, does $\kappa$-tameness for $\alpha$-types imply $\kappa$-tameness for $\beta$-types when $\alpha < \beta$?  This question reveals a divide in the tameness literature: some results only use tameness for 1-types (such as the categoricity transfer of Grossberg and VanDieren Theorem \ref{gv-upward} and the deriving a frame from superstability Theorem \ref{ss-implies-all}), while others require full tameness and shortness (such as the development of $<\kappa$-satisfiability Theorem \ref{coheir-thm}). Answering this question would help to unify these results.

        Another stark divide is revealed by examining the list of examples of tame AECs in Section \ref{examples-subsec}: the list begins with general results that give some form of locality at a cardinal $\lambda$.  However, once the list reaches concrete classes of AECs, every example turns out to be $<\aleph_0$-tame (often this is a result of a syntactic characterization of Galois types, but not always).  This suggests the question of what lies between or even if the general results can be strengthened down to $<\aleph_0$-tameness.  For the large cardinal results, this seems impossible: the large cardinal $\kappa$ should give no information about the low properties of $\K$ below it because this cardinal disappears in $\K_\kappa$.  The other results also seem unlikely to have this strengthening, but no counter example is known.  Indeed, the following is still open: Is there an AEC $\K$ that is $\aleph_0$-tame but not $<\aleph_0$-tame?

	\item {\bf Dividing lines}\\
	This direction has two prongs.  The first prong is to increase the number of dividing lines.  So far, the classification of tame AECs (and AECs in general) has focused on the superstable or better case with a few forays into strictly stable \cite{bg-v9, limit-strictly-stable-v3}.  Towards the goal of proving Shelah's Categoricity Conjecture, this focus makes sense.  However, this development pales in comparison to the rich structure of classification theory in first-order\footnote{Part of this structure is represented graphically at \url{http://forkinganddividing.com} by Gabe Conant.}.  Exploring the correct generalizations of NIP, NTP${}_2$, etc.\ may help fill out the AEC version of this map. It might be that stronger locality hypotheses than tameness will have to be used: as we have seen already in the superstable case, it is only known how to prove the existence of a global independence relation assuming full $(<\aleph_0)$-tameness and shortness.
	
	The other prong is to turn classification results into true dividing lines. In the first-order case, dividing lines correspond to nice properties of forking on one side and to chaotic non-structure results on the other. In AECs, the non-structure side of dividing lines is often poorly developed and most results either revolve around the order property or use non-ZFC combinatorial principle. While these combinatorial principles seem potentially necessary in arbitrary AECs\footnote{For instance, result the statement ``Categoricity in $\lambda$ and less than the maximum number of models in $\lambda^+$ implies $\lambda$-AP'' holds under weak diamond, but fails under Martin's axiom \cite[Conclusion I.6.13]{shelahaecbook}.}, a reasonable hope is that tame AECs will allow the development of stronger ZFC nonstructure principles. For example, Shelah claims that in AECs with amalgamation, the order property (or equivalently in the tame context stability, see Theorem \ref{stab-spectrum}) implies many models on a tail of cardinals. However there is no known analog for superstability: does unsuperstability imply many models?
	
	\item {\bf Interaction with other fields}\\
	Historically, examples have not played a large role in the study of AECs.  Examples certainly exist because $\Ll_{\kappa, \omega}$ sentences provide them, but the investigation of specific classes is rarely carried out\footnote{A large exception to this is the study of quasiminimal classes (see Example \ref{examples-subsec}.(\ref{quasimin})) by Zilber and others, which are driven by questions from algebra.}.  A better understanding of concrete examples would help advance the field in two ways. First, nontrivial applications would help provide more motivation for exploring AECs\footnote{It should be noted that some prominent AEC theorists disagree with this as a motivating principle.}. Second, interesting applications can help drive the isolation of new AEC properties that might, \emph{a priori}, seem strange.
	
	This interaction has the potential to go the other way as well: one can attempt to study a structure or a class of structures by determining where the first-order theory lies amongst the dividing lines and using the properties of forking there.  However, if the class is not elementary, then the first-order theory captures new structures that have new definable objects.  These definable objects might force the elementary class into a worse dividing line.  However, AECs offer the potential to look at a narrower, better behaved class.  For instance, an interesting class might only have the order property up to some length $\lambda$\footnote{Like example the algebraically closed valued fields of rank one, Example \ref{examples-subsec}.(\ref{valfields}).} or only be able to define short and narrow (but infinite) trees.  Looking at the first-order theory loses this extra information and looking at the class as an AEC might move it from an unstable elementary class to a stable AEC.

	  \item {\bf Reverse mathematics of tameness}\\
	  The compactness theorem of first-order logic is equivalent to a weak version of the axiom of choice (Tychonoff's theorem for Hausdorff spaces). If we believe that tameness is a natural principle, then maybe the first-order version of ``tameness'' is also, in the choiceless context, equivalent to some topological principle: what is this principle?
	  
            \item\label{open-q-natural} {\bf How ``natural'' is tameness?}\\
              We have seen that all the known counterexamples of tameness are pathological. Is this a general phenomenon? Are there natural mathematical structures that are, in some sense, well-behaved and should be amenable to a model-theoretic analysis, but are not tame? Would this example then satisfy a weaker version of tameness? 
            
     	\item {\bf Categoricity and tameness}\\
	We have seen that tameness helps with Shelah's Categoricity Conjecture, but there are still unanswered questions about eliminating the successor assumption and amalgamation property. For example, does the categoricity conjecture hold in fully $<\aleph_0$-tame and -type short AECs with amalgamation?

  Going the other way, what is the impact of categoricity on tameness?  Grossberg has conjectured that amalgamation should follow from high enough categoricity.  Does something like this hold for tameness?	
      \item {\bf On stable and superstable tame AECs} \\
        From the work discussed in this survey, several more down to earth questions arise:

              \begin{enumerate}
              \item Can one build a global independence relation in a fully tame and short superstable AEC? See also Question \ref{indep-question}.
              \item Is there a stability spectrum theorem for tame AECs (i.e.\ a converse to Theorem \ref{stab-spectrum})?
              \item In superstable tame AECs, can one develop the theory of forking further, say by generalizing geometric stability theory to that context?
              \end{enumerate}
\end{enumerate}

\section{Historical remarks} \label{hist-sec}

\subsection{Section \ref{primer-no-tameness}}

Abstract elementary classes were introduced by Shelah \cite{sh88}; see Grossberg \cite{grossberg2002} for a history. Shelah \cite{sh88} (an updated version appears in \cite[Chapter I]{shelahaecbook}) contains most of the basic results in this Section \ref{primer-no-tameness}, including Theorem \ref{pres-thm}. Notation \ref{hanf-notation} is due to Baldwin and is used in \cite[Chapter 14]{baldwinbook09}. Galois types are implicit  in \cite{sh300-orig} where Theorem \ref{mod-homog-sat} also appears. Existence of universal extensions (Lemma \ref{univ-exist}) is also due to Shelah and has a similar proof (\cite[I.2.2.(4)]{sh394}). 

Splitting (Definition \ref{splitting-def}) is introduced by Shelah in \cite[Definition 3.2]{sh394}. Lemma \ref{stabsplit} is \cite[Claim 3.3]{sh394}. The extension and uniqueness properties of splitting (Theorem \ref{gen-unique-dns-fact}) are implicit in \cite{sh394} but were first explicitly stated by VanDieren \cite[I.4.10, I.4.12]{vandierennomax}. The order property is first defined for AECs in \cite[Definition 4.3]{sh394}. 

Definition \ref{ss-def} is implicit already in \cite{shvi635}, but there amalgamation in $\mu$ is not assumed (only a weak version: the density of amalgamation bases). It first appears\footnote{With minor variations: joint embedding and existence of a model in $\mu$ is not required.} explicitly (and is given the name ``superstable''\footnote{This can be seen as a somewhat unfortunate naming, as there are several potentially non-equivalent definitions of superstability in AECs. Some authors use ``no long splitting chains'', but this omits the conditions of amalgamation, no maximal models, and joint embedding in $\mu$. Perhaps it is best to think of the definition as a weak version of having a good $\mu$-frame.}) in \cite[Definition 7.12]{grossberg2002}. Limit models appear in \cite[Definition 4.1]{kosh362} under the name ``$(\theta, \sigma)$-saturated''. The ``limit'' terminology is used starting in \cite{shvi635}. The reader should consult \cite{gvv-toappear-v1_2} for a history of limit models and especially the question of uniqueness. Theorem \ref{uq-limit-symmetry} is due to VanDieren \cite{vandieren-symmetry-v4-toappear}.

Shelah's eventual categoricity conjecture can be traced back to a question of Łoś \cite{los-conjecture} which eventually became Morley's categoricity theorem \cite{morley-cip}. See the introduction to \cite{ap-universal-v9} for a history.  Conjecture \ref{conj-inf-logic} appears as an open problem in \cite{shelahfobook78}. Theorem \ref{sh87-thm} is due to Shelah \cite{sh87a, sh87b}. Conjecture \ref{categ-conj-aec} appears as \cite[Question 6.14]{sh702}. Theorem \ref{shelah-succ-downward} is the main result of \cite{sh394}. Theorem \ref{shvi} appears in \cite[Theorem 2.2.1]{shvi635} under GCH but without amalgamation. Assuming amalgamation (but in ZFC), the proof is similar, see \cite[Corollary 6.3]{gv-superstability-v3}. The proof of Shelah and Villaveces omits some details. A clearer version will appear in \cite{shvi-note}. An easier proof exists if the categoricity cardinal has high-enough cofinality, see \cite[Lemma 6.3]{sh394}. Question \ref{sat-quest} is stated explicitly as an open problem in \cite[Problem D.1.(2)]{baldwinbook09}. Theorems \ref{sat-categ} and \ref{uq-limit-categ} are due to VanDieren and the second author \cite[Section 7]{vv-symmetry-transfer-v3}. 

Good frames are the main concept in Shelah's book on classification theory for abstract elementary classes \cite{shelahaecbook}. The definition appears at the beginning of Chapter II there, which was initially circulated as Shelah \cite{sh600}. There are some minor differences with the definition we give here, see the notes for Section \ref{tame-indep-sec} for more. Question \ref{baldwin-question} originates in the similar question Baldwin asked for $L(Q)$ \cite[Question 21]{one-hundred-problems}. For AECs, this is due to Grossberg (see the comments around \cite[Problem 5]{sh576}). A version also appears as \cite[Problem 6.13]{sh702}%\footnotei{WB: Again, this is not the history I've heard, but I suppose we should ask.  Also, you later cite it as appearing in a paper that appears two years earlier (\cite{sh702}) SV: Rami claims that Shelah took it from him without giving credit... Anyway we should probably ask Rami to review our history section.}
. Theorem \ref{shelah-local-good-frame} is due to Shelah \cite[Theorem VI.0.2]{shelahaecbook2}. A weaker version with the additional hypothesis that the weak diamond ideal is not $\lambda^{++}$-saturated appears is proved in Shelah \cite{sh576}, see \cite[Theorem II.3.7]{shelahaecbook}. Corollary \ref{baldwin-q-answer} is the main result of \cite{sh576}. Theorem \ref{successful-cond} is the main result of \cite[Chapter II]{shelahaecbook}. Claim \ref{claim-xxx} is implicit in \cite[Discussion III.12.40]{shelahaecbook} and a proof should appear in \cite{sh842}. Theorem \ref{good-frame-categ} is due to Shelah and appears in \cite[Section IV.4]{shelahaecbook}. Shelah claims that categoricity in a proper class of cardinals is enough but this is currently unresolved, see \cite{categ-infinitary-v2} for a more in-depth discussion. Theorems \ref{good-frame-categ}, \ref{good-frame-categ-2}, and \ref{categ-transfer-wgch} appear in \cite[Section IV.7]{shelahaecbook}. However we have not fully checked Shelah's proofs. A stronger version of Theorem \ref{good-frame-categ} has been shown by VanDieren and the second author in \cite[Section 7]{vv-symmetry-transfer-v3}, while \cite[Section 11]{downward-categ-tame-v4} gives a proof of Theorems \ref{good-frame-categ-2} and \ref{categ-transfer-wgch} (with alternate proofs replacing the hard parts of Shelah's argument). %\wbnote{I'm not sure if this counts as history. SV: Well, we are giving attribution and references here. This last sentence gives a reference proving stronger theorems and giving readable proofs so I think it is relevant.}

\subsection{Section \ref{tame-sec}}

The version of Definition \ref{tamedef} using types over sets is due to the second author \cite[Definition 2.22]{sv-infinitary-stability-afml}. Type-shortness was isolated by the first author \cite[Definition 3.3]{tamelc-jsl}. Locality and compactness appear in \cite{non-locality}. Proposition \ref{injective-map} is folklore. As for Proposition \ref{easy-implications}, the first part appears as \cite[Theorem 3.5]{tamelc-jsl}, the second and third first appear in Baldwin and Shelah \cite{non-locality}, and the third is implicit in \cite{sh394} and a short proof appears in \cite[Lemma 11.5]{baldwinbook09}. 

In the framework of AECs, the Galois Morleyization was introduced by the second author \cite{sv-infinitary-stability-afml} and Theorem \ref{morleyization-thm} is proven there. After the work was completed, we were informed that a 1981 result of Rosický \cite{rosicky81} also uses a similar device to present concrete categories as syntactic classes.  That tameness can be seen as a topological principle (Theorem \ref{tameness-topo}) appears in Lieberman \cite{lieberman2011}. 

On Section \ref{examples-subsec}:

\begin{enumerate}
	\item Tameness could (but historically was not) also have been extracted from the work of Makkai and Shelah on the model theory of $L_{\kappa, \omega}$, $\kappa$ a strongly compact \cite{makkaishelah}. There the authors prove that Galois types are, in some sense, syntactic \cite[Proposition 2.10]{makkaishelah}\footnote{This was another motivation for developing the Galois Morleyization.}. The first author \cite{tamelc-jsl} generalized these results to AECs and later observations in \cite{btr-almost-compact-v2-toappear, lc-tame-v2} slightly weakened the large cardinal hypotheses.
	\item Theorem \ref{tameness-from-categ} is due to Shelah. The first part appears essentially as \cite[II.2.6]{sh394} and the second is \cite[IV.7.2]{shelahaecbook}. The statement given here appears as \cite[Theorem 8.4]{downward-categ-tame-v4}.
        \item Theorem \ref{tameness-from-categ-2} is essentially \cite[Conclusion 3.7]{sh472}.
	\item This is folklore and appears explicitly on \cite[p.~15]{superior-aec}.
	\item The study of the classification theory of universal classes starts with \cite{sh300-orig} (an updated version appears as \cite[Chapter V]{shelahaecbook2}), where Shelah claims a main gap for this framework (the details have not fully appeared yet). Theorem \ref{tame-uc} is due to the first author \cite{tameness-groups}. A full proof appears in \cite[Theorem 3.7]{ap-universal-v9}.
	\item Finitary AECs were introduced by Hyttinen and Kesälä \cite{finitary-aec}. That $\aleph_0$-Galois-stable $\aleph_0$-tame finitary AECs are $<\aleph_0$-tame is Theorem 3.12 there. The categoricity conjecture for finitary AECs appears in \cite{categ-finit}. The beginning of a geometric stability theory for finitary AECs appears in \cite{group-config-kangas-apal}.
        \item Homogeneous model theory was introduced by Shelah in \cite{sh3}. See \cite{grle-homog} for an exposition.  The classification theory of this context is well developed, see, for instance \cite{lessmann2000, hs-independence, hysh676,buechler-lessmann, hlsh821}. For connections with continuous logic, see \cite{berenstein-buechler, shus837}.
	\item Averageable classes are introduced by the first author in \cite{gammault-v2}.
	\item A summary of continuous first-order logic in its current form can be found in \cite{fourguys-metric}.  Metric AECs were introduced in \cite{maec-hihy} and tameness there is first defined in \cite{za12}.
        \item Quasiminimal classes were introduced by Zilber \cite{zil05}. See \cite{quasimin} for an exposition and \cite{quasimin-five} for a proof of the excellence axiom (and hence of tameness).
	\item That the $\lambda$-saturated models of a first-order superstable theory forms an AEC is folklore. That it is tame is easy using that the Galois types are the same as in the original first-order class.
        \item Superior AECs are introduced in \cite{superior-aec}.
        \item Hrushovski fusions are studied as AECs in \cite{vzhrush}.
	\item This appears in \cite{bet}.
	\item This is analyzed in  \cite{cacvf}.
\end{enumerate}

The Hart-Shelah example appears in \cite{hs-example}, where the authors show that it is categorical in $\aleph_0, \ldots, \aleph_n$ but not in $\aleph_{n + 1}$. The example was later extensively analyzed by Baldwin and Kolesnikov \cite{bk-hs} and the full version of Theorem \ref{hart-shelah-thm} appears there. %An expository account is in \cite[Chapter 26]{baldwinbook09}\svnote{I don't see the need to remove references to better expositions...}. 
The Baldwin-Shelah example appears in \cite{non-locality}. The Shelah-Boney-Unger example was first introduced by Shelah \cite{sh932} for a measurable cardinal and adapted by Boney and Unger \cite{lc-tame-v2} for other kinds of large cardinals.

\subsection{Section \ref{universal-class-sec}}

The categoricity transfer for universal classes is due to the second author \cite{ap-universal-v9}\footnote{Previous version of this preprint claimed the full categoricity conjecture but gaps have been found and a complete solution has been delayed to a sequel.}. This section presents a proof incorporating ideas from the later paper \cite{categ-primes-v3}. If not mentioned otherwise, results and definitions there are due to the second author. 

Lemma \ref{galois-ext} is folklore when atomic equivalence is transitive but is \cite[Theorem 4.14]{ap-universal-v9} in the general case. As for Theorem \ref{good-frame-transfer}, one direction is folklore. The other direction (tameness implies that the good frame can be extended) is due to the authors, see the notes on Theorem \ref{good-frame-transfer-2} below. The version with weak amalgamation (Theorem \ref{frame-transfer}) is due to the second author.

Theorem \ref{gv-upward} is due to Grossberg and VanDieren \cite{tamenessthree}. Definition \ref{prime-def} is due to Shelah \cite[Definition III.3.2]{shelahaecbook}. The account of orthogonality and unidimensionality owes much to Shelah's development in \cite[Sections III.2,III.6]{shelahaecbook} but differs in some technical points explained in details in \cite{categ-primes-v3}. Theorem \ref{unidim-categ} is due to Shelah \cite[III.2.3]{shelahaecbook}. Theorem \ref{not-unidim} is due to Shelah with stronger hypotheses \cite[Claim III.12.39]{shelahaecbook} and to the second author as stated \cite[Theorem 2.15]{categ-primes-v3}. 

\subsection{Section \ref{tame-indep-sec}}

Question \ref{indep-quest} is implicit in \cite[Remark 4.9]{sh394}. A more precise statement appears as \cite[Question 7.1]{bgkv-apal}.

The presentation of abstract independence given here appears in \cite{indep-aec-v6-toappear}\footnote{There independence relations are not required to satisfy base monotonicity.}. The definition of a good frame given here (Definition \ref{good-frame-def}) also appears in \cite[Definition 8.1]{indep-aec-v6-toappear}. Compared to Shelah's original definition (\cite[Definition II.2.1]{shelahaecbook}), the definition given here is equivalent \cite[Remarks 3.5, 8.2]{indep-aec-v6-toappear} except for three minor differences:

\begin{itemize}
  \item The existence of a superlimit model is not required. This has been adopted in most subsequent works on good frames, including e.g.\ \cite{jrsh875}.
  \item Shelah's definition requires forking to be defined for types over models only. However it is possible to close the definition to types over sets (see for example \cite[Lemma 5.4]{bgkv-apal}).
  \item Shelah defines forking only for a subclass of all types which he calls \emph{basic}. They are required to satisfy a strong density property (if $M \lta N$, then there is a basic type over $M$ realized in $N$). If the basic types are all the (nonalgebraic) types, Shelah calls the good frame \emph{type-full}. In the tame context, a type-full good frame can always be built (see \cite[Remark 3.10]{gv-superstability-v3}). Even in Theorem \ref{shelah-local-good-frame}, the frame can be taken to be type-full (see \cite[Claim III.9.6]{shelahaecbook}). The bottom line is that in all cases where a good frames is known to exist, a type-full one also exists.
\end{itemize}

Question \ref{indep-question} appears (in a slightly different form) as \cite[Question 7.1]{bgkv-apal}. Theorem \ref{indep-canon-thm} is Corollary 5.19 there\footnote{Of course the general idea of looking at forking as an abstract independence relation which turns out to be canonical is not new (see for example Lascar's proof that forking is canonical in superstable theories \cite[Theorem 4.9]{lascar76}).}. As for Proposition \ref{stable-indep-props}, all are folklore except (\ref{tameshort-witness}) which appears as \cite[Lemma 4.5]{indep-aec-v6-toappear} and symmetry which in this abstract framework is \cite[Corollary 5.18]{bgkv-apal} (in the first-order case, the result is due to Shelah \cite{shelahfobook78} and uses the same method of proof: symmetry implies failure of the order property).

Galois stability was defined for the first time in \cite{sh394}. The second part of Theorem \ref{stab-transfer} is due to Grossberg and VanDieren \cite[Corollary 6.4]{tamenessone}. Later the argument was refined by Baldwin, Kueker, and VanDieren \cite{b-k-vd-spectrum} to prove the first part. Theorem \ref{stab-spectrum} is due to the second author \cite[Theorem 4.13]{sv-infinitary-stability-afml}. 

Averages in the nonelementary framework were introduced by Shelah (for stability theory inside a model) in \cite{sh300-orig}, see \cite[Chapter V.A]{shelahaecbook2}. They are further used in the AEC framework in \cite[Chapter IV]{shelahaecbook}. The Galois Morleyization is used by the authors to translate Shelah's results from stability theory inside a model to tame AECs in \cite[Section 5]{bv-sat-v3}. They are further used in \cite{gv-superstability-v3}.

That tameness can be used to obtain a global uniqueness property for splitting (Theorem \ref{ns-uniq-tame-fact}) is due to Grossberg and VanDieren \cite[Theorem 6.2]{tamenessone}. $<\kappa$-satisfiability was introduced as $\kappa$-coheir in the AEC framework by Grossberg and the first author \cite{bg-v9}. This was strongly inspired from the work of Makkai and Shelah \cite{makkaishelah} on coheir in $L_{\kappa, \omega}$, $\kappa$ a strongly compact. A weakening of Theorem \ref{coheir-thm} appears in \cite{bg-v9}, assuming that coheir has the extension property. The version stated here is due to the second author and appears as \cite[Theorem 5.15]{sv-infinitary-stability-afml}. Theorem \ref{coheir-ext-lc} is \cite[Theorem 8.2]{bg-v9}. The definition of $\mu$-forking (\ref{mu-forking-def}) is due to the second author and appears in \cite{ss-tame-jsl}. Theorem \ref{stable-forking-def} is proven in \cite[Section 7]{indep-aec-v6-toappear}. Theorem \ref{chainsat-stable} is due to the authors \cite[Theorem 6.10]{bv-sat-v3}.

Theorem \ref{ss-implies-all}.(\ref{ss-1}) is due to the second author \cite[Proposition 10.10]{indep-aec-v6-toappear}. Theorem \ref{ss-implies-all}.(\ref{ss-2}) is due to VanDieren and the second author \cite[Corollary 6.10]{vv-symmetry-transfer-v3} (an eventual version appears in \cite{bv-sat-v3}, and an improvement of VanDieren \cite{vandieren-chainsat-apal} can be used to obtain the full result). Theorem \ref{ss-implies-all}.(\ref{ss-3}-\ref{ss-4}) are also due to VanDieren and the second author \cite{vv-symmetry-transfer-v3}, although (\ref{ss-3}) and (\ref{ss-4}) were observed by the second author in \cite{ss-tame-jsl} in the categorical case (i.e.\ when we know that the union of a chain of $\lambda$-Galois-saturated models is $\lambda$-Galois-saturated). 

Theorem \ref{gvsuperstab} and Remark \ref{rmk-chainsat} are due to Grossberg and the second author \cite{gv-superstability-v3}. The notion of a superlimit model appears already in Shelah's original paper on AECs \cite{sh88} (see \cite[Chapter I]{shelahaecbook}). Shelah introduces solvability in \cite[Definition IV.1.4]{shelahaecbook}. Lemma \ref{tame-superstab} appears in \cite[Lemma 3.8]{gv-superstability-v3}. When $\kappa$ is strongly compact, it can be traced back to Makkai-Shelah \cite[Proposition 4.12]{makkaishelah}. 

Theorem \ref{good-frame-transfer-2} is due to the authors and appears in full generality in \cite{tame-frames-revisited-v5}. Rami Grossberg told us that he privately conjectured the result in 2006 and told it to Adi Jarden and John Baldwin (see also the account in the introduction to \cite{jarden-tameness-apal}). In \cite[Theorem 8.1]{ext-frame-jml}, the first author proved the theorem with an additional assumption of tameness for \emph{two} types used to transfer symmetry. Later, \cite{tame-frames-revisited-v5} showed that symmetry transfer holds without this extra assumption. At about the same time as \cite{tame-frames-revisited-v5} was circulated, Adi Jarden gave a proof of symmetry from tameness assuming an extra property called the continuity of independence (he also showed that this property followed from the existence property for uniqueness triples). The argument in \cite{tame-frames-revisited-v5} shows that the continuity of independence holds under tameness and hence also completes Jarden's proof.

Independent sequences %(Definition \ref{indep-seq-def}) 
were introduced by Shelah in the AEC framework \cite[Definition III.5.2]{shelahaecbook}. A version of Theorem \ref{dim-thm} for models of size $\lambda$ is proven as \cite[III.5.14]{shelahaecbook} with the assumption that the frame is weakly successful. This is weakened in \cite{jasi}, showing that the so-called continuity property of independence is enough. In \cite{tame-frames-revisited-v5}, the continuity property is proven from tameness and hence Theorem \ref{dim-thm} follows, see \cite[Corollary 6.10]{tame-frames-revisited-v5}.

Definition \ref{good-indep-def} is due to the second author \cite[Definition 8.1]{indep-aec-v6-toappear}. The definition of almost good (Definition \ref{almost-good-def}) is implicit there and made explicit in \cite[Definition A.2]{ap-universal-v9}. Fully good and almost fully good are also defined there. Theorem \ref{almost-good-thm} and Theorem \ref{fully-good-thm} are due to the second author. A statement with a weaker Hanf number is the main result of \cite{indep-aec-v6-toappear} (the proof uses ideas from Shelah \cite[Chapter III]{shelahaecbook} and Adi Jarden \cite{jarden-tameness-apal}). The full result is proven in \cite[Appendix A]{ap-universal-v9}. What it means for a frame to be successful (Definition \ref{successful-def}) is due to Shelah \cite[Definition III.1.1]{shelahaecbook} but we use the equivalent definition from \cite[Section 11]{indep-aec-v6-toappear}. Type locality (Definition \ref{type-loc-def}) is introduced by the second author in \cite[Definition 14.9]{indep-aec-v6-toappear}. Corollary \ref{fully-good-cor} and Theorem \ref{fully-good-strong-compact} is implicit in \cite{indep-aec-v6-toappear} (with the Hanf number improvement in \cite[Appendix A]{ap-universal-v9}). Theorem \ref{univ-classes-goodness} is due to the second author \cite[Appendix C]{ap-universal-v9}.

Theorem \ref{succ-transfer} appears in \cite{tamenesstwo}. A version of Theorem \ref{ss-categ} is already implicit in \cite[Section 10]{indep-aec-v6-toappear}. Shelah's omitting type theorem (Theorem \ref{shelah-omit-type}) appears in its AEC version as \cite[Lemma II.1.6]{sh394} and has its roots in \cite[Proposition 3.3]{makkaishelah}, where a full proof already appears. Corollary \ref{omit-type-cor} is due to the second author \cite[Theorem 9.8]{downward-categ-tame-v4}. The categoricity conjecture for tame AECs with primes (Theorem \ref{event-categ-primes}) is due to the second author (the result as stated here was proven in a series of papers \cite{ap-universal-v9, categ-primes-v3, downward-categ-tame-v4}, see \cite[Corollary 10.9]{downward-categ-tame-v4}). The converse from Remark \ref{event-categ-primes-rmk} is stated in \cite{prime-categ-v6-toappear}. The categoricity conjecture for homogeneous model theory is more or less implicit in \cite{sh3} and is made explicit by Hyttinen in \cite{hyt-categ-homog} (when the language is countable, this is due to Lessmann \cite{lessmann2000}\footnote{In that case, a stronger statement can be proven: if $D$ is categorical in some uncountable cardinal, then it is categorical in all uncountable cardinals.}). More precisely, Hyttinen proves that categoricity in a $\lambda > |T|$ with $\lambda \neq |T|^{+\omega}$ implies categoricity in all $\lambda' \ge \min (\lambda, h (|T|))$. Corollary \ref{categ-homog} is stronger (as it includes the case $\lambda = |T|^{+\omega}$) and is due to the second author \cite[Theorem 0.2]{categ-primes-v3}. Theorem \ref{succ-transfer-2} is due to the second author \cite[Corollary 10.6]{downward-categ-tame-v4}. Theorem \ref{categ-conj-tame-short} is due to the second author (although the main idea is due to Shelah, and the only improvement given by tameness is the Hanf number, see Theorem \ref{categ-transfer-wgch}). With full tameness and shortness, a weaker version appears in \cite[Theorem 1.6]{indep-aec-v6-toappear}, and the full version using only tameness is \cite[Corollary 11.9]{downward-categ-tame-v4}. The second part of Corollary \ref{lc-cor} also appears there.

\subsection{Section \ref{conclusion-sec}}

Several of these questions have been in the air and there is some overlap with the list \cite[Appendix D]{baldwinbook09}. The question about the length of tameness appears in the first author's Ph.D.\ thesis \cite{willthesis}. A question about examples of tameness and nontameness appear in \cite[Appendix D.2]{baldwinbook09}. Whether failure of superstability implies many models is conjectured in \cite{sh394} (see the remark after Claim 5.5 there) and further discussed at the end of \cite{gv-superstability-v3}. 

The idea of exploring the reverse mathematics of tameness (and the specific question of what tameness corresponds to if compactness is the Tychnoff theorem for Hausdorff spaces) was communicated to the second author by Rami Grossberg. That tameness follows from categoricity was conjectured by Grossberg and VanDieren \cite[Conjecture 1.5]{tamenessthree}. That one can build a global independence relation in a fully tame and short superstable AEC is conjectured by the second author in \cite[Section 15]{indep-aec-v6-toappear}.

\bibliographystyle{amsalpha}
\bibliography{survey-tame-aecs}

\end{document}